\documentclass[reqno]{amsart}

\usepackage{amsfonts}
\usepackage{amscd}
\usepackage{amsbsy}
\usepackage{amsxtra}
\usepackage{amssymb}
\usepackage{epsfig}
\usepackage{epic,eepic}
\usepackage{graphicx}
\usepackage{psfrag}
\usepackage{amsmath}
\usepackage{amsthm}
\usepackage{setspace}
\usepackage{url}
\usepackage{color}
\usepackage{here}
\usepackage{todonotes}
\usepackage{dsfont}
\newlength{\halfbls}
\usepackage{tikz}
\usetikzlibrary{calc}
\usetikzlibrary{decorations.pathreplacing,decorations.markings}
\usetikzlibrary{arrows}

\DeclareRobustCommand{\SkipTocEntry}[9]{}

   \newcommand{\CC}{\mathbb{C}}
  \newcommand{\HH}{\mathbb{H}}

 \newcommand{\PP}{\mathbb{P}}  \newcommand{\NN}{\mathbb{N}}
\newcommand{\QQ}{\mathbb{Q}} \newcommand{\RR}{\mathbb{R}}  

\newcommand{\ZZ}{\mathbb{Z}}

    \newcommand{\Aut}{{\rm Aut}}

\newcommand{\Cov}{{\rm Cov}}
\newcommand{\Hur}{{\rm Hur}}
\newcommand{\TR}{{\rm TR}}

\newcommand{\cHH}{{\mathcal H}}
 
 \newcommand{\cQQ}{{\mathcal Q}}
\newcommand{\cRR}{{\mathcal R}}

\newcommand{\bfe}{{\bf e}}

\newcommand{\bfu}{{\bf u}}
\newcommand{\bfw}{{\bf w}}

\newcommand{\ual}{{\boldsymbol{\alpha}}}
\newcommand{\ube}{{\boldsymbol{\beta}}}
\newcommand{\uga}{{\boldsymbol{\gamma}}}

\newcommand{\Hmu}{\Pi}

   \newcommand{\N}{{\rm N}} 
\newcommand{\Part}{{\mathcal P}}

\newcommand{\mm}{{\rm{\bf{m}}}}
\newcommand{\zz}{{\rm{\bf{z}}}}

    \newcommand{\ve}{{\varepsilon}}
\newcommand{\ol}{\overline}

\newtheorem{Defi}{Definition}[section]  
    \newtheorem{Prop}[Defi]{Proposition}
\newtheorem{Lemma}[Defi]{Lemma}    \newtheorem{Cor}[Defi]{Corollary}
\newtheorem{Thm}[Defi]{Theorem} 
\newtheorem*{Thm*}{Theorem}


\def\={\;=\;}  \def\+{\,+\,}       \def\h{\tfrac12}  
          
 \def\t#1{\tilde{#1}}

  \newcommand\C{\mathbb C}  \newcommand\Z{\mathbb Z}   
\renewcommand\N{\mathbb N}  \renewcommand\P{\mathbb P}   

           \def\l{\lambda}
          \def\2{\pi_2}
\def\G{\Gamma}           
\def\zz{\mathbf z}

\def\SL#1{{\rm SL}(2,#1)}       
\def\sm#1#2#3#4{\bigl(\smallmatrix#1&#2\\#3&#4\endsmallmatrix\bigr)}

\def\be{\begin{equation}}   \def\ee{\end{equation}}     \def\bes{\begin{equation*}}    \def\ees{\end{equation*}}
\def\ba{\be\begin{aligned}} \def\ea{\end{aligned}\ee}   \def\bas{\bes\begin{aligned}}  \def\eas{\end{aligned}\ees}

\def\bq#1{\bigl\langle#1\bigr\rangle_q}   \def\sbq#1{\langle#1\rangle_q}

\def\bG#1{\bigl\langle#1\bigr\rangle_{q,G}}

\newcounter{savedtocdepth}
\newcommand*{\SaveTocDepth}[1]{%
  \addtocontents{toc}{%
    \protect\setcounter{savedtocdepth}{\protect\value{tocdepth}}%
    \protect\setcounter{tocdepth}{#1}%
  }%
}

\definecolor{qqzzqq}{rgb}{0,0.6,0}
\definecolor{ccttqq}{rgb}{0.8,0.2,0}
\definecolor{ccqqcc}{rgb}{0.8,0,0.8}
\definecolor{ffcczz}{rgb}{1,0.8,0.6}
\definecolor{qqttff}{rgb}{0,0.2,1}
\definecolor{uququq}{rgb}{0.25,0.25,0.25}
\definecolor{ffqqqq}{rgb}{1,0,0}
\definecolor{gray}{rgb}{0.5,0.5,0.5}
\definecolor{cqcqcq}{rgb}{0.75,0.75,0.75}

\title[Counting Feynman-like graphs]
{Counting Feynman-like graphs:\\ Quasimodularity and Siegel-Veech weight}

\author{Elise Goujard}
\thanks{Research of the first author is supported by the Fondation Math\'ematique Jacques Hadamard.}
\address{
Laboratoire de Math{\'e}matiques d'Orsay,
Universit{\'e} Paris-Sud,
F-91405 Orsay Cedex, France\\
}
\email{elise.goujard@gmail.com}

\author{Martin M\"oller}
\thanks{Research  of  the second author is partially supported  
by the DFG-project MO 1884/1-1}
\address{
Institut f\"ur Mathematik, Goethe--Universit\"at Frankfurt,
Robert-Mayer-Str. 6--8,
60325 Frankfurt am Main, Germany\\
}
\email{moeller@math.uni-frankfurt.de}

\begin{document}
\bibliographystyle{halpha}

\begin{abstract}
We prove the quasimodularity of generating functions for counting torus covers, with and without Siegel-Veech weight. Our proof is based on analyzing decompositions of flat surfaces into horizontal cylinders. The quasimodularity arise as contour integral of quasi-elliptic functions. It provides an alternative proof of the 
quasimodularity results of Bloch-Okounkov, Eskin-Okounkov and 
Chen-M\"oller-Zagier, and generalizes the results of 
B\"ohm-Bring\-mann-Buchholz-Markwig for simple ramification covers.

\end{abstract}
\maketitle

\tableofcontents

\noindent
\SaveTocDepth{1}

\section{Introduction}

The generating series counting the number of torus coverings first attracted 
attention with Dijkgraaf's work (\cite{Dijk}) on mirror symmetry for elliptic 
curves. It was shown rigorously by Kaneko-Zagier (\cite{KanZag}) that these 
functions are quasimodular forms. This statement was generalized
by Eskin-Okounkov (\cite{eo}) from simple branch points to arbitrary branching
profile.
\par
In this paper we show that the quasimodularity property of counting
functions generalizes in two ways. We first analyze to which extent
the quasimodularity holds when counting the contributions of each 
underlying global graph separately. The precise statement requires 
a correspondence theorem between covers and decorated graphs. In the case 
of simple branching, the global graphs are trivalent and sometimes 
referred to as Feynman graphs. In this case our correspondence theorem 
boils down to the correspondence theorem for tropical Hurwitz numbers 
of torus covers proved in~\cite{BBBM}, as we explain in 
Section~\ref{sect:tropical}.
\par
The second generalization counts coverings with a Siegel-Veech weight, 
motivated by Siegel-Veech constants for flat surfaces. Here again,
our method provides a different approach and a refinement of the 
quasimodularity shown in \cite{CMZ}, by counting the contributions 
of each (``Feynman'') graph separately.
\par
We give some motivation for why we care about quasimodularity statements.
Obviously, knowing the first few coefficients of a quasimodular form
determines the whole series and thus provides a computational approach
to the counting problems. Second, the asymptotic behaviour of the 
coefficients of a
quasimodular form is well-understood (\cite[Section~9]{CMZ}). Those
coefficient asymptotics are important e.g.\ to compute the Masur-Veech
volumes of moduli spaces of flat surfaces. Despite some recent advances
(\cite{aez}, \cite{goujard}, \cite{CMZ}) many refined questions, 
concerning e.g.\ 
large genus asymptotics, spin structure distinction and Masur-Veech volumes 
of quadratic differential spaces in general, are still wide open. 
We plan to apply the techniques presented here to these
cases in a sequel to this paper.
\par
\medskip
A covering $p:X \to E$ of the square torus provides $X$ with a flat 
metric~$\omega = p^* \omega_E$. The flat surface $(X,\omega)$ is 
swept out by horizontal cylinders. We obtain the {\em global graph
of the covering} by letting the vertices be the branch points of~$p$
and the edges these horizontal cylinders.
Our correspondence theorem  shows roughly that decorating
the graph with widths and heights at the edges and with local data
(triple Hurwitz numbers) at the vertices defines a bijection with
torus covers, 
see Proposition~\ref{prop:corrTCGraph} for the precise statement.
In the case of simple branching, our global graphs are the
tropical covers of e.g.\ \cite{BBBM}.
\par
For the counting problems the following special case of
Theorem~\ref{thm:coeff0QM} is the core of the quasimodularity statements. 
\par
\begin{Thm} \label{thm:introPint}
Let $P$ be a product of derivatives $\wp^{(m)}(z_i-z_j)$
 of the Weierstrass $\wp$-functions. Then the
constant term\footnote{We refer to Section~\ref{sec:intEF} for 
the conventions on the ``heights'' $\ve_i$ of integration paths.}
 with respect to the variables  $\zeta_j = e^{2\pi i z_j}$
$$[\zeta_n^0, \ldots,\zeta_1^0]\, P \= \frac{1}{(2\pi i)^n}\,
\oint_{0 +i \ve_n}^{1 +i \ve_n}  \dots \oint_{0 +i \ve_1}^{1 +i \ve_1}  P(z_1, \dots, z_n;\tau)dz_1\dots dz_n
$$
is a quasimodular form. More precisely,  if $P$ consists of 
$\ell$~factors, where the $k$-th factor involves the $m_k$-th derivative, 
then the quasimodular form has mixed weight less or equal to $\sum_{k=1}^\ell (2+m_k)$. 
\end{Thm}
\par
Note that in general these constant terms are not of pure weight\footnote{Purity of the weight of the quasimodular form is claimed in 
Theorem~3.2 of \cite{BBBM}, but it relies on Proposition~3.3, which
has a gap.} 
as we show in the example in Section~\ref{sec:Ex1111Pcomp}, even if all $m_k=0$.
The statement of the theorem above involves only the elliptic Weierstra\ss\
$\wp$-function. Nevertheless our proof 
requires Theorem~\ref{thm:coeff0QM} about the quasimodularity 
of constants terms for quasi-elliptic functions in full generality,
since taking the coefficient $[\zeta_1^0]$ of an elliptic function may 
no longer be elliptic.
\par
\medskip
We let $N^\circ(\Hmu) = \sum N^\circ_d(\Hmu) q^d$ be the generating series
of torus covers with branching profile $\Hmu$. Our first geometric application
is an independent proof of the following result of Kaneko-Zagier
and Eskin-Okounkov, based on Theorem~\ref{thm:introPint}.
\par
\begin{Thm}  \label{intro:count} (= Corollary~\ref{cor:noname})
For any ramification profile~$\Hmu$ the counting 
function $N^\circ(\Hmu)$ for connected torus covers of profile~$\Hmu$ 
is a quasimodular form of mixed weight less or equal to
$|\Hmu| + \ell(\Hmu)$.
\end{Thm}
\par
The generating series for torus covers $N^\circ(\Hmu)= \sum_\Gamma 
N^\circ(\Gamma, \Hmu)$ can be decomposed as
according to the associated global graph~$\Gamma$. 
In general, the individual contributions
$N^\circ(\Gamma, \Hmu)$ are not quasimodular forms. Already
genus two surfaces and $\Hmu$ consisting of a $3$-cycle provides an example, 
see Section~\ref{sec:H2stratum}. Our method of proof gives a refinement
of the quasimodularity statement for the case of simple branch points.
Along with the interpretation in terms of tropical covers we
show in Section~\ref{sect:tropical}:
\par
\begin{Thm}[= Corollary~\ref{cor:QMindiv}]
In the case $\Hmu=((2),\dots, (2))$,  for any trivalent graph~$\Gamma$ 
the contribution $N'(\Hmu,\Gamma)$ of 
the graph $\Gamma$ to the total counting is a quasimodular form 
of mixed weight at most $|\Hmu| + \ell(\Hmu)$.
\end{Thm}
\par
\medskip
Siegel-Veech constants measure the asymptotic number of immersed
cylinders in a flat surface of bounded length of the waist curve.
They are important characteristic quantities of the dynamics
of billiards and flat surfaces, see Section~\ref{sec:reltoSV} for 
a brief summary and \cite{eskinmasur}, \cite{emz}, \cite[Section~1]{CMZ} 
for more details. The Siegel-Veech constants for a general
flat surface in a given stratum can be computed by determining
the asymptotics of Siegel-Veech constants for spaces of torus
covers. This in turn requires counting torus covers with
a combinatorial constant, the Siegel-Veech weight (depending on an
integer parameter~$p \geq -1$), that we define
in Section~\ref{sec:SVgeneral}. In analogy with the simple counting
problem we consequently define the generating series 
$c^\circ_{p}(\Hmu)$ of Siegel-Veech weighted coverings. Showing that a such
series is a quasimodular form is important because of the good
control of the coefficient asymptotics of quasimodular 
forms (see \cite[Section~9]{CMZ}).
Counting Siegel-Veech weighted graphs gives a new proof of the
following theorem (see \cite[Theorem~6.4]{CMZ}) and the refinement 
graph by graph in the trivalent case that we state in 
Corollary~\ref{cor:QMindiv}. 
\par
\begin{Thm} \label{intro:SV}[= Corollary~\ref{Cor:SVmain}]
For any ramification profile $\Hmu$ and any odd integer $p \geq -1$
the generating series $c^\circ_{p}(\Hmu)$ for counting connected
covers with $p$-Siegel-Veech weight is a quasimodular form 
of mixed weight at most $|\Hmu| + \ell(\Hmu)+p+1$.
\end{Thm}
\par
We conclude with an outline of the proof of Theorem~\ref{intro:count} and
Theorem~\ref{intro:SV}. Using the correspondence theorem 
Proposition~\ref{prop:corrTCGraph} 
our problem is converted into counting decorated graphs whose vertex
labels are triple Hurwitz numbers. These are piecewise polynomials
in the input data, i.e.\ the edge labels of the graph. If these were 
globally polynomials (as they are in the trivalent case), the graph
sums can be interpreted as the constant coefficient of a polynomial
in the Weierstra\ss\ $\wp$-function and its derivatives, see 
Proposition~\ref{prop:S1}. The polynomiality can be restored using
{\em completed cycles~$p_k$} instead of the weighted symmetric group 
characters~$f_k$ in the Burnside formula for counting coverings.
Using the notion of $q$-bracket and the fact that both $f_k$
and $p_k$ generate the algebra of shifted symmetric functions,
the arguments of Section~\ref{sec:DHN} allow to come back to the
true counting problem while maintaining quasimodularity.
\par
\subsection*{Acknowledgements} 
We are very grateful to Alex Eskin for sharing with us the manuscript of an
old project with Andrei Okounkov that was at the origin of the notion of
completed cycles (see the reference to [12] in \cite{OkPand}). We also 
thank Don Zagier for many fruitful 
conversations on quasimodular forms and Dmitry Zvonkine and Kathrin 
Bringmann for useful suggestions and comments. We moreover thank the referee whose comments helped to improve the exposition of the paper.
\par
This research was conducted at the {\em Max-Planck-Institut
f\"ur Mathematik, Bonn}, whose hospitality we gratefully acknowledge.
\par
\subsection*{Notation}
For a partition $\lambda = (\lambda_1 \geq \lambda_2 \geq \cdot)$
we let $|\lambda| = \sum_{i \geq 0} \lambda_i$ be the number that
$\lambda$ is a partition of, i.e., $\lambda \vdash  |\lambda|$. We denote by 
$\ell(\lambda) = \max(\{i: \lambda_i \geq 0\})$ the {\em length} of the
partition. We also need the {\em weight} ${\rm wt}(\lambda) = |\lambda| 
+ \ell(\lambda)$ of a partition. We adopt the corresponding notation
for tuples, i.e.\ if $\bfw = (w_1,\ldots,w_n)$ then $|\bfw| = \sum w_i$
and $n = \ell(\bfw)$.

\section{Counting Covers of elliptic curves by global graphs} 
\label{sec:CovElliptic}

In this section we recall  basic facts about enumeration of covers, 
both for torus coverings and coverings of the projective line with
three marked points. The aim of this section is the correspondence
theorem Proposition~\ref{prop:corrTCGraph} that gives a bijection between torus
coverings and decorated graphs. This proposition holds on the level
of covers {\em without unramified components only}, but (for general
branching profile) neither on the level of connected coverings nor on
the level of all coverings. This fact requires us to set up quite
a bit of notation before giving the statement.

\subsection{Covers of elliptic curves and their Hurwitz tuples} 
\label{sec:CovEll}

Here we recall basic facts about enumeration of covers using tuples of
elements in the symmetric group. Our aim here is to explain the passage
between the number of connected and non-connected coverings and
to express these numbers in terms of characters on the symmetric group.
We focus on torus coverings in this section.
\par
Let $\Hmu = (\mu^{(1)}, \cdots, \mu^{(n)})$ consist of partitions $\mu^{(i)} = 
(\mu^{(i)}_{1}, \mu^{(i)}_{2}, \cdots )$ such that each entry $\mu^{(i)}_{j}$ is 
a non-negative integer and for later use we define~$g$ by 
$\sum_{i,j}(\mu^{(i)}_{j} -1) = 2g-2$. We call such a tuple $\Hmu$ 
a {\em ramification profile}. 
\par
A covering $p: X \to E$ of the torus~$E$ has ramification profile $\Hmu$, if the
covering has $n$ numbered branch points and over the $i$-th branch point the sheets 
coming together form 
the partition $\mu^{(i)}$ (completed by singletons, if $|\mu^{(i)}| < \deg(p)$). Let 
$$\rho: \pi_1(E \setminus \{P_1,\ldots,P_n\}) \to S_d$$
be the monodromy representation in the symmetric group of $d$ elements associated
with a covering~$p$ and some base point~$P$, that we suppress in notation.
We use the convention that loops (and elements of the symmetric group) are composed from right to left. 
The elements $(\alpha, \beta, \gamma_1, \cdots, \gamma_n)$ as in 
Figure~\ref{fig:pointpos} generate the fundamental group  
$\pi_1(E\setminus \{P_1,\ldots,P_n\})$ with the relation 
\be\label{eq:FRel}
\beta^{-1}\alpha^{-1}\beta\alpha = \gamma_n \cdots \gamma_1.
\ee
\begin{figure}[h]
\begin{centering}
\begin{tikzpicture}
\tikzset{
    >=latex',
    firstarrow/.style={
           ->,
           shorten >=2pt,},
    secondarrow/.style={
           postaction={decorate},
           decoration={markings,mark=at position .65 with
           {\arrow[line width=.7pt]{>}}}}, 
    thirdarrow/.style={
           postaction={decorate},
           decoration={markings,mark=at position .90 with
           {\arrow[line width=.7pt]{>}}}} 
}    

\draw (0,0) node(1){} -- (0,4) node(2){} -- (4,4) node(3){} -- (4,0) node(4){} -- cycle;
\draw (4.5,0) node(5){} -- (4.5,4) node(6){} -- (8.5,4) node(7){} -- (8.5,0) node(8){} -- cycle;
\draw (0,0.3) node(9){} -- (4,0.3) node(10){};
\draw (0.3,0) node(11){} -- (0.3,4) node(12){};
\draw (4.5,0.3) node(13){} -- (8.5,0.3) node(14){};
\draw (4.8,0) node(15){} -- (4.8,4) node(16){};

\draw (9) -- (10) [secondarrow];
\draw (11) -- (12) [secondarrow];
\draw (13) -- (14) [secondarrow];
\draw (15) -- (16) [secondarrow];

\draw (4.8,0.3) -- (5.6,3) [firstarrow];
\draw (4.8,0.3) -- (6.35,2.9) [firstarrow];
\draw (4.8,0.3) -- (7.3,2.1) [firstarrow];

\fill (1.1,3) circle (1.7pt)
      (1.85,2.9) circle (1.7pt)
      (2.7,2.1) circle (1.7pt)
      (5.6,3) circle (1.7pt)
      (6.35,2.9) circle (1.7pt)
      (7.3,2.1) circle (1.7pt);

\draw (0.3,0.3) .. controls (.77,2.32) and (.81,2.49) .. (.84,2.65) [thirdarrow];
\draw (0.3,0.3) .. controls (.82,2.15) and (.9,2.53) .. (1.1,2.63); 
\draw plot [smooth, tension=1] coordinates { (.84,2.65) (.9,3.15) (1.28,3.25) (1.32,2.87) (1.1,2.63)};

\draw (0.3,0.3) .. controls (1.3,2.32) and (1.49,2.49) .. (1.53,2.6) [thirdarrow];
\draw (0.3,0.3) .. controls (1.3,2.15) and (1.63,2.53) .. (1.65,2.5); 
\draw plot [smooth, tension=1] coordinates { (1.53,2.6) (1.65,3.05) (2.1,3.15) (2.06,2.7) (1.65,2.5)};

\draw (0.3,0.3) .. controls (2,1.52) and (2.19,1.69) .. (2.24,1.83) [thirdarrow];
\draw (0.3,0.3) .. controls (2,1.35) and (2.33,1.73) .. (2.50,1.7); 
\draw plot [smooth, tension=1] coordinates { (2.24,1.83) (2.45,2.25) (2.9,2.35) (2.9,1.9) (2.50,1.7)};

\tikzstyle{every node}=[font=\scriptsize] 
\node (0) at (-.15,-.15) {$~$};
\node (P) at (0.15,0.42) {$P$};
\node (P) at (4.65,0.42) {$P$};
\node (P1) at (1.2,3.5) {$P_1$}; 
\node (P2) at (2.2,3.3) {$P_2$}; 
\node (P3) at (3.15,2.4) {$P_3$};
\node (P1) at (5.7,3.5) {$P_1$}; 
\node (P2) at (6.7,3.3) {$P_2$}; 
\node (P3) at (7.65,2.4) {$P_3$}; 
\node (b) at (0.15,2.75) {$\beta$};
\node (b) at (4.65,2.75) {$\beta$};
\node (a) at (2.8,0.42) {$\alpha$};
\node (a) at (7.3,0.42) {$\alpha$};
\node (g1) at (.65,2.5) {$\gamma_1$};
\node (g2) at (1.23,2.4) {$\gamma_2$};
\node (gn) at (1.93,1.72) {$\gamma_n$};
\node (d1) at (5.3,2.5) {$\delta_1$};
\node (d2) at (5.85,2.4) {$\delta_2$};
\node (dn) at (6.47,1.72) {$\delta_n$};

\fill (2.5,3.01) circle (.7pt)
      (2.65,2.87) circle (.7pt)
      (2.8,2.72) circle (.7pt)   
      (7,3.01) circle (.7pt)
      (7.15,2.87) circle (.7pt)
      (7.3,2.72) circle (.7pt);
\end{tikzpicture} 
\end{centering}
\caption{Standard presentation of $\pi_1(E\setminus \{P_1,\ldots,P_n\})$ 
}
\label{fig:pointpos}
\end{figure}
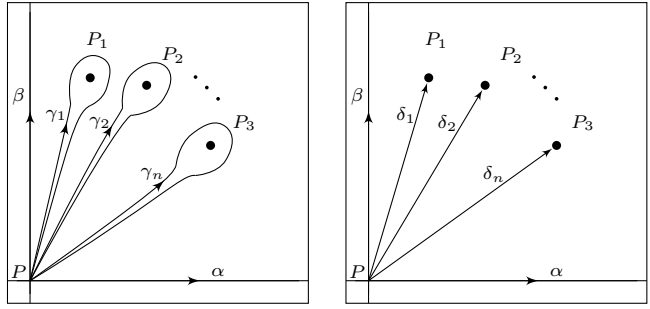
Given such a homomorphism $\rho$, we let $\ual = \rho(\alpha)$, $\ube = \rho(\beta)$
and $\uga_i = \rho(\gamma_i)$ and call the tuple
\be \label{eq:HT}
h \=(\ual, \ube, \uga_1, \cdots, \uga_n) \in (S_d)^{n+2}
\ee
the {\em Hurwitz tuple} corresponding to $\rho$ and the choice of generators.
Our main goal is to count coverings, or rather the corresponding Hurwitz tuples, and so we let
\be
\Hur_d(\Hmu) \= \{ H \=(\ual, \ube, \uga_1, \cdots, \uga_n) \in (S_d)^{n+2} \,\,
\text{of profile $\Hmu$} \} \,,
\ee
where we say that $H$ has profile~$\Hmu$ if the partitions $[\uga_i]$ 
associated with conjugacy class of $\uga_i$ are equal to $\mu^{(i)}$ for $i=1,\ldots,n$. 
Here we use the general convention to call two partitions of different sizes $d_1 \leq d_2$ equal
if they differ by $d_2 - d_1$ parts of length one.
\par
So far we have made no connectedness assumption, but we will ultimately be interested
in counting connected coverings, hence transitive monodromy representations. We indicated this 
subset by an upper circle. As important technical intermediate notion we need covers
without unramified components, indicated by a prime, so we let
\bas
\Hur^\circ_d(\Hmu) &\= \{ H \in \Hur_d(\Hmu)\,:\, \langle H \rangle \,\,\text{acts transitively on
$\{1,\ldots,d\}$} \}  \\
\Hur'_d(\Hmu) &\= \{ H \in \Hur_d(\Hmu)\,:\, \text{ $\langle  \uga_1, \cdots, \uga_n \rangle$
acts non-trivially on every $H$-orbit} \}  
\eas
\par
The corresponding countings of covers (as usual with weight $1/\Aut(p)$) differ from the
cardinalities of these sets of Hurwitz tuples by the simultaneous conjugation of the 
Hurwitz tuple, hence by a factor of $d!$. Consequently, we let
\be \label{eq:NfromCov}
N_d(\Hmu) \= \frac{|\Hur^0_d(\Hmu)|}{d!}, \quad
N'_d(\Hmu) \= \frac{|\Hur'_d(\Hmu)|}{d!}, \quad
N^\circ_d(\Hmu) \= \frac{|\Hur^\circ_d(\Hmu)|}{d!}\,,
\ee
and package these data into the generating series 
\be \label{eq:genserN}
N(\Hmu) \= \sum_{d=0}^\infty N_d(\Hmu) q^d, \quad N'(\Hmu) \= \sum_{d=0}^\infty 
N_d'(\Hmu) q^d, \quad N^0(\Hmu) \= \sum_{d=0}^\infty N^0_d(\Hmu) q^d \,.
\ee
\par
From $|\Hur_d(\Hmu)| \= \sum_{j=0}^d \binom{d}{j} |\Hur'_j(\Hmu)|\,| \Hur_{d-j}()| $
one deduces that
\begin{equation} \label{eq:NNN}
N'(\Hmu) \= N(\Hmu)/N()\,.
\end{equation}
\par
In order to state the passage from connected counting to counting without
unramified components we need to define the set of ramification points and
its partitions. For $i \in \{1,\ldots,n\}$ let $\mu_j^{(i)}$ for $j \in J=J(i)$
be the parts of $\mu^{(i)}$ of length greater than one and let 
$$\cRR(\Hmu) \= \{(i,j), \, i \in \{1,\ldots,n\}\,\,\,\text{and}\,\,\, j \in J(i) \}$$
be the index set of ramification points of the profile~$\Hmu$.
We let $\P(\cRR)$ be the set of partitions of the set $\cRR = \cRR(\Hmu)$
that are finer than the partition by different first index. For any part~$A$
of such a partition we let $\Hmu_A$ be the profile consisting of the 
partitions $\mu_A^{(i)} = \{\mu_j^{(i)}, \,(i,j) \in A\}$ grouped together 
according to~$A$. We omit those~$i$ for which there is no~$j$ with $(i,j) \in A$, 
so that $\Hmu_A$ is a profile with~$n$ or less branch points.
\par
\begin{Prop} The generating function for counting covers  without
unramified components can be expressed in terms of counting functions
for connected covers as
$$ N'(\Hmu) \= \sum_{\alpha \in \P(\cRR)} \prod_{A \in \alpha} N^\circ(\Hmu_A) $$
\end{Prop}
\par
\begin{proof} Any covering~$p$ without unramified components induces a
partition $\alpha \in \P(\cRR)$ of the ramification points according to
its connected components. We label the sheets of the covering and
obtain the identity of the sets of Hurwitz numbers
$$ |\Hur_d'(\Hmu)| \= \sum_{\alpha \in \P(\cRR)} \sum_{(d_A)_{A \in \alpha}}
 \binom{d}{(d_A)_{A \in \alpha}} \prod_{A \in \alpha} |\Hur_{d_A}^\circ(\Hmu_A)|\,,$$
from which the claim follows immediately.
\end{proof}
\par
Since the summand $N^\circ(\Hmu)$ appears on the right side, 
classical inclusion-exclusion allows to invert this formula and
to write $N^\circ(\Hmu)$ as a linear combination of products of $N'(\Hmu_A)$
for subsets~$A$ of~$R$.
\par
\smallskip
With the aim of connecting counting problems to the representation theory
of the symmetric group, we recall the classical Burnside Lemma  (see e.g.\ 
\cite[Theorem~A.1.10]{LanZvon}) that the number of Hurwitz tuples with ramification
profile $\Hmu$ is given by
\be \label{eq:Burnside}
 \Hur_d(\Hmu) \= d! \sum_{\lambda \in \Part(d)} \prod_{i=1}^n f_{\mu^{(i)}}(\lambda), 
\ee
where a conjugacy class $\sigma$ is completed with singletons
to form a partition of $|\lambda|$ and where
\be \label{eq:defssf}
f_{\sigma}(\lambda) \= z_{\sigma} \chi^\lambda(\sigma)/\dim \chi^\lambda\,. 
\ee
Here $z_\sigma$ denotes the size of the conjugacy class of $\sigma$ and 
$\dim \chi^\lambda$ is the dimension of representation $\lambda$. We
also write $f_k$ for the special case that $\sigma$ is a $k$-cycle.


\subsection{Covers of the projective line with three marked points} 
\label{sec:CovPP1}

Covers of the projective line $\pi: S \to \PP^1$  can, of course, also be described 
by their monodromy. The main point
here is to introduce some notation and to highlight the fact that
we consider one of the points ($z=1$) to have a fixed ramification 
profile given by a partition~$\mu$, whereas the ramification over the
other two branch points ($z=0$ and $z= \infty)$ are prescribed by two 
('input' and 'output')
tuples of variables $\bfw^- = (w_1^-,\ldots,w^-_{n^-})$ and 
$\bfw^+ =  (w_1^+,\ldots,w^+_{n^+})$. 

 We conclude again, with the passage
between the number of connected and non-connected coverings and
expressions in terms of characters of the symmetric group.
\par
The use of the terminology double and triple Hurwitz numbers is not completely consistent in the litterature. Most classically, {\em double Hurwitz numbers} count coverings with prescribed behaviour over two points and besides only simple branching. In e.g. \cite{SSZ} (that we will use later), this notion is generalized allowing instead of simple branching several cycles of fixed length $r$. We call these {\em generalized double Hurwitz numbers}. We also need {\em triple Hurwitz numbers}  that count covers with three prescribed ramification points. In the literature simple branch points might be allowed besides, but we will not need this case and do not consider it. We will frequently use the special case of triple Hurwitz numbers where the ramification profile over one of the points is a cycle. These are obviously (special cases of) generalized double Hurwitz numbers.
\par
Our general notation convention is that $\Cov$ denotes a set
of coverings, $\TR$ ('triple ramification') is the set of tuples describing the monodromy 
of a covering and~$A$ denotes the number of coverings, always 
with weight $1/\Aut(\pi)$. We use here the upper indices $\circ$ and prime
as in the previous section, to denote connected covers and covers without
unramified components, respectively.
\par
We need to highlight one more detail, the numbering of preimages of
the branch points. Suppose that $\deg(\pi) = d$. For a partition~$\mu$ 
and a point $x \in \PP^1$ we write $\pi^{-1}(x)  =  [\mu]$ if the cycles 
in $\pi^{-1}(x)$ agree with the partition $\mu$, completed by ones to form 
a partition of~$d$.
We say that~$x$ has {\em unnumbered profile}~$\mu$ in this case.
If $\bfw$ is a tuple of integers with $\sum w_i =d$, we may consider it 
as a partition $[\bfw]$ and write $\pi^{-1}(x)  =  [\bfw]$ to specify
an unnumbered profile. More frequently we will  write that covering
$\pi$ has the property $\pi^{-1}(x)  =  \bfw$ and we will say that~$x$ has 
{\em numbered profile}~$\bfw$ over~$x$, if the covering comes with
a labeling $\sigma_x$ of $\pi^{-1}(x)$ such that at the $i$-th point 
$w_i$~sheets come together. 
\par
We will consider most of the time the profile over $z=0$ and $z=\infty$
to be numbered and over $z=1$ to be unnumbered. If all of these points
have unnumbered profiles, we add the subscript~$un$. 
\par
Consequently, we have explained the conventions for our notations
\ba 
\Cov_{un }(\bfw^-, \bfw^+,\mu) &\= \bigl\{\pi: S \to \PP^1, \, \deg(\pi) = 
\sum w_i^+ = \sum w_i^-\,, \\
& \phantom{\= \bigl\{\pi:} \pi^{-1}(1) = [\mu], \,\,\pi^{-1}(0) = [\bfw^-],\,\,\
\pi^{-1}(0) = [\bfw^-] \bigr\}
\ea
and
\ba
\Cov(\bfw^-, \bfw^+,\mu) &\= \bigl\{(\pi: S \to \PP^1,\sigma_0,\sigma_\infty)\, :  \, \deg(\pi) = 
\sum w_i^+ = \sum w_i^-\,, \\
& \phantom{\= \bigl\{\pi:} \pi^{-1}(1) = [\mu], \,\,\pi^{-1}(0) = \bfw^-,\,\,\
\pi^{-1}(0) = \bfw^- \bigr\}
\ea
for unnumbered and numbered covering. The same notation convention is used for 
\ba 
\TR_{un}(\bfw^-, \bfw^+,\mu) & \=  \{ T=(\ual,\ube,\uga) \in S_d^3; \quad 
[\ual] = [\bfw^-], \,\, [\ube] = [\bfw^+],\,\,  [\uga] = \mu \} \\
\TR(\bfw^-, \bfw^+,\mu) & \=  \{ (T=(\ual,\ube,\uga), \sigma_0, \sigma_\infty); \quad [\uga] = \mu,\,\,
[\ual] = \bfw^-, \,\, [\ube] = \bfw^+\,\,  \} 
\ea
for Hurwitz tuples. Moreover, we define
\ba \label{eq:defDHN}
A_{un}(\bfw^-, \bfw^+,\mu) &\= \sum_{\pi \in \Cov_{un}(\bfw^-, \bfw^+,\mu)} \frac{1}{\Aut(\pi)}
\= \frac{1}{d!}\, |\TR_{un}(\bfw^-, \bfw^+,\mu)|  \\
 A(\bfw^-, \bfw^+,\mu) &\= \sum_{\pi \in \Cov(\bfw^-, \bfw^+,\mu)} \frac{1}{\Aut(\pi)}
\= \frac{1}{d!}\, |\TR(\bfw^-, \bfw^+,\mu)| 
\ea
for the weighted number of Hurwitz tuples if $d=\sum w_i^+ = \sum w_i^-$, 
and we let $A(\bfw^-, \bfw^+,\mu) =0$ if this condition does not hold.
\par
As a consequence of the Burnside Lemma we can again write those cardinalities 
in terms of characters of the symmetric group as
\ba
A_{un}(\bfw^-,\bfw^+,\mu) &\= \frac{z_{\bfw^-} z_{\bfw^+}}{d!^2} \sum_{|\lambda| = d} 
\chi_{\bfw^-}^\lambda \chi_{\bfw^+}^\lambda  f_\mu(\lambda) \\
&\= 
({\prod_i w^-_i\prod_i w^+_i\, \prod_j r^-_j \prod_j r^+_j })^{-1}
\sum_{|\lambda| = d} \chi_{\bfw^-}^\lambda
\chi_{\bfw^+}^\lambda  f_\mu(\lambda) 
\ea 
if $d = \sum_i w_i^+ =  \sum_j w_j^-$, and zero otherwise, where $r_j^\pm$ are the multiplicities of the parts of $\bfw^\pm$. Correspondingly, 
\be 
A(\bfw^-,\bfw^+,\mu) \= \frac{1}{\prod_i w^-_i\prod_i w^+_i }
\sum_{|\lambda| = d} \chi_{\bfw^-}^\lambda \chi_{\bfw^+}^\lambda  f_\mu(\lambda)
\ee
if $d = \sum_i w_i^+ =  \sum_j w_j^-$, and zero otherwise. This formula
is the point of departure for the counting problems.
\par
\medskip
As in the case of covers of elliptic curves we conclude this section with
a discussion on the passage between connected and non-connected version. 
\par
\begin{Lemma}The double Hurwitz number with no ramification point over $z=1$ is
\be \label{eq:Aemptyset}
A(\bfw^-, \bfw^+, \emptyset)\=\frac{1}{\prod_i w_i^+}\delta_{\bfw^+}^{\bfw^-}\,.
\ee
\end{Lemma}
\par
\begin{proof}
This is a straightforward application of the second orthogonality relation for 
characters. 
\end{proof}
\par
\medskip
Since a general disconnected cover can be decomposed as a disjoint union
of a cover without unramified components and a collection of unramified
covers (i.e.\ of cylinders), we obtain the following lemma.
\par
\begin{Lemma} The triple Hurwitz numbers can be written in terms of 
triple Hurwitz numbers without unramified components for subsets of
the ramification profile as 
\be \label{eq:AAprime}
A(\bfw^-, \bfw^+, \mu)=\sum_{\substack{\bfu^+\subset\bfw^+, \; \bfu^-\subset\bfw^-\\|\mu|\leq|\bfu^-|=|\bfu^+|}}A'(\bfu^-, \bfu^+,\mu)\,A(\bfw^-\setminus \bfu^-, \bfw^+\setminus\bfu^+, \emptyset)\,,
\ee
where we have set $A(\emptyset,\emptyset,\emptyset) = 1$.
\end{Lemma}
\par
Consequently, we can apply inclusion-exclusion (or M\"obius inversion)
to this formula and write triple Hurwitz numbers without unramified components
in terms of (non-connected) Hurwitz numbers. That is, there exists
a M\"obius function $M(\bfu^+,\bfu^-,\bfw^+,\bfw^-)$ such that
\be\label{eq:AprimeA}
A'(\bfw^-, \bfw^+, \mu)=\sum_{\substack{\bfu^+\subset\bfw^+, \; \bfu^-\subset\bfw^-\\|\mu|\leq|\bfu^-|=|\bfu^+|}} M(\bfu^+,\bfu^-,\bfw^+,\bfw^-)\,
A(\bfu^-, \bfu^+,\mu)\,.
\ee 

\subsection{Global graphs and cylinder decompositions} \label{sec:globalgraph}

We now suppose moreover that the base of the covering~$p$ is the square torus 
$E = \CC/(\ZZ+i\ZZ)$ and the $p$ is a cover without unramified components. 
We fix the holomorphic one-form $\omega_E$ on~$E$ with 
period lattice $\Lambda = \ZZ+i\ZZ$ and provide~$X$ with the flat structure 
$\omega =  p^* \omega_E$. To such a situation we will associate a graph with 
decorations as follows. The horizontal foliation of $\omega$ is completely 
periodic. We select from each homotopy class of horizontal cylinders~$c$ 
one representative, the core  curve $\gamma_c$. We let $X^0 = \setminus \cup_{c} 
\gamma_c$ be the complement of all the core curves.
For merely counting covers the precise location of the 
branch points is irrelevant. For concreteness, we use in the sequel the {\em branch 
point normalization} that the $i$-th branch point has fixed coordinates 
$z_i = x_i + \sqrt{-1} \ve_i$ with $0\leq \ve_1 < \ve_2 < \cdots < \ve_n <1$ and any $x_i\in[0,1)$. 
\par
The {\em global graph~$\Gamma$} associated with the flat 
surface $(X,\omega = p^* \omega_E)$ of ramification profile~$\Hmu$ is the 
graph~$\Gamma$ with $n=|\Hmu|$ vertices, labeled by $1,\ldots,n=|\Hmu|$, 
see Figure~\ref{cap:coverGlobal} below.
\begin{figure}[h]
\begin{minipage}[t]{0.40\linewidth}
\begin{tikzpicture}[scale = 1, every node/.style={transform shape}, line cap=round,line join=round,>=triangle 45,x=1.0cm,y=1.0cm]
\draw [color=cqcqcq,dash pattern=on 2pt off 2pt, xstep=1.0cm,ystep=1.0cm] (-1.25,-4.68) grid (3.81,4.88);
\clip(-1.25,-4.68) rectangle (3.81,4.88);
\draw (-1,0)-- (-0.5,2.3);
\draw (-0.5,2.3)-- (-0.7,4.5);
\draw (-0.7,4.5)-- (0.3,4.5);
\draw (0.3,4.5)-- (0.5,2.3);
\draw (0.5,2.3)-- (1.3,3.5);
\draw (1.3,3.5)-- (2.3,3.5);
\draw (2.3,3.5)-- (1.5,2.3);
\draw (1.5,2.3)-- (1,0);
\draw (1,0)-- (2,1);
\draw (2,1)-- (3,1);
\draw (3,1)-- (2,0);
\draw (2,0)-- (3.3,-1.5);
\draw (3.3,-1.5)-- (2.3,-1.5);
\draw (2.3,-1.5)-- (1.3,-1.5);
\draw (1.3,-1.5)-- (0,0);
\draw (0,0)-- (-1,0);
\draw [->] (0.5,-1.26) -- (0.5,-2.4);
\draw (-0.23,4.86) node[anchor=north west] {$a$};
\draw (1.71,-1.54) node[anchor=north west] {$a$};
\draw (1.73,3.91) node[anchor=north west] {$b$};
\draw (2.75,-1.54) node[anchor=north west] {$b$};
\draw (2.41,1.4) node[anchor=north west] {$c$};
\draw (-0.6,-0.05) node[anchor=north west] {$c$};
\draw (0,-4)-- (0,-3);
\draw (1,-3)-- (1,-4);
\draw (0,-4)-- (1,-4);
\draw (0,-3)-- (1,-3);
\draw (1.45,-0.59) node[anchor=north west] {$C_1$};
\draw (1.7,0.73) node[anchor=north west] {$C_2$};
\draw (0.16,1.47) node[anchor=north west] {$C_3$};
\draw (-0.28,3.6) node[anchor=north west] {$C_4$};
\draw (1,3.17) node[anchor=north west] {$C_5$};
\begin{scriptsize}
\draw [color=black] (-1,0)-- ++(-2.0pt,-2.0pt) -- ++(4.0pt,4.0pt) ++(-4.0pt,0) -- ++(4.0pt,-4.0pt);
\draw [color=black] (0,0)-- ++(-2.0pt,-2.0pt) -- ++(4.0pt,4.0pt) ++(-4.0pt,0) -- ++(4.0pt,-4.0pt);
\draw [color=black] (1,0)-- ++(-2.0pt,-2.0pt) -- ++(4.0pt,4.0pt) ++(-4.0pt,0) -- ++(4.0pt,-4.0pt);
\draw [color=black] (2,0)-- ++(-2.0pt,-2.0pt) -- ++(4.0pt,4.0pt) ++(-4.0pt,0) -- ++(4.0pt,-4.0pt);
\draw [color=black] (2,1)-- ++(-2.0pt,-2.0pt) -- ++(4.0pt,4.0pt) ++(-4.0pt,0) -- ++(4.0pt,-4.0pt);
\draw [color=black] (3,1)-- ++(-2.0pt,-2.0pt) -- ++(4.0pt,4.0pt) ++(-4.0pt,0) -- ++(4.0pt,-4.0pt);
\draw [color=black] (-0.5,2.3) circle (2.0pt);
\draw [color=black] (1.5,2.3) circle (2.0pt);
\draw [color=black] (0.5,2.3) circle (2.0pt);
\fill [color=black] (-0.7,4.5) circle (2.0pt);
\fill [color=black] (0.3,4.5) circle (2.0pt);
\fill [color=black] (1.3,3.5) circle (2.0pt);
\fill [color=black] (2.3,3.5) circle (2.0pt);
\fill [color=black] (1.3,-1.5) circle (2.0pt);
\fill [color=black] (3.3,-1.5) circle (2.0pt);
\fill [color=black] (2.3,-1.5) circle (2.0pt);
\draw [color=black] (0,-4)-- ++(-2.0pt,-2.0pt) -- ++(4.0pt,4.0pt) ++(-4.0pt,0) -- ++(4.0pt,-4.0pt);
\draw [color=black] (0.57,-3.74) circle (2.0pt);
\fill [color=black] (0.31,-3.46) circle (2.0pt);
\end{scriptsize}
\end{tikzpicture}
\end{minipage}
\begin{minipage}[t]{0.40\linewidth}\vbox{
\begin{tikzpicture}[scale=0.75, line cap=round,line join=round,>=triangle 45,x=1.0cm,y=1.0cm]
\clip(-3.15,-2.62) rectangle (2.36,4.78);
\draw (0,-1)-- (0,1);
\draw [shift={(0.75,2)}] plot[domain=2.21:4.07,variable=\t]({1*1.25*cos(\t r)+0*1.25*sin(\t r)},{0*1.25*cos(\t r)+1*1.25*sin(\t r)});
\draw [shift={(-0.75,2)}] plot[domain=-0.93:0.93,variable=\t]({1*1.25*cos(\t r)+0*1.25*sin(\t r)},{0*1.25*cos(\t r)+1*1.25*sin(\t r)});
\draw(0.63,-1) circle (0.63cm);
\draw [shift={(-1,-1)}] plot[domain=-3.14:0,variable=\t]({1*1*cos(\t r)+0*1*sin(\t r)},{0*1*cos(\t r)+1*1*sin(\t r)});
\draw [shift={(-1,3)}] plot[domain=0:3.14,variable=\t]({1*1*cos(\t r)+0*1*sin(\t r)},{0*1*cos(\t r)+1*1*sin(\t r)});
\draw (-2,3)-- (-2,-1);
\draw [->] (0,0) -- (0,0.25);
\draw [->] (-0.5,1.91) -- (-0.5,1.98);
\draw [->] (0.49,1.89) -- (0.5,1.96);
\draw [->] (-2,1) -- (-2,0.84);
\draw (0.25,0.25) node[anchor=north west] {$C_3$};
\draw (1.52,-0.82) node[anchor=north west] {$C_2$};
\draw (-1.79,0.91) node[anchor=north west] {$C_1$};
\draw (-1.4,2.21) node[anchor=north west] {$C_4$};
\draw (0.7,2.19) node[anchor=north west] {$C_5$};
\begin{scriptsize}
\draw [color=black] (0,-1)-- ++(-3.0pt,-3.0pt) -- ++(6.0pt,6.0pt) ++(-6.0pt,0) -- ++(6.0pt,-6.0pt);
\draw [color=black] (0,1) circle (3.0pt);
\fill [color=black] (0,3) circle (3.0pt);
\end{scriptsize}
\end{tikzpicture}

\begin{tikzpicture}[scale=0.35,line cap=round,line join=round,>=triangle 45,x=1.0cm,y=1.0cm]
\clip(-11.87,-6.09) rectangle (3.11,7.02);
\draw (-7.86,3.8)-- (-8.12,5.06);
\draw [dash pattern=on 5pt off 5pt] (-8.12,5.06)-- (-4.12,5.06);
\draw (-4.12,5.06)-- (-3.86,3.8);
\draw (-3.86,3.8)-- (-3.52,5.06);
\draw [dash pattern=on 5pt off 5pt] (-3.52,5.06)-- (-1.52,5.06);
\draw (-1.52,5.06)-- (-1.86,3.8);
\draw (-1.86,3.8)-- (-1.52,2.06);
\draw [dash pattern=on 5pt off 5pt] (-1.52,2.06)-- (-5.52,2.06);
\draw (-5.52,2.06)-- (-5.86,3.8);
\draw (-5.86,3.8)-- (-6.12,2.06);
\draw (-7.86,3.8)-- (-8.12,2.06);
\draw [dash pattern=on 5pt off 5pt] (-8.12,2.06)-- (-6.12,2.06);
\draw (-10.88,-2.09)-- (-10.88,-3.63);
\draw [dash pattern=on 5pt off 5pt] (-10.88,-3.63)-- (-6.88,-3.63);
\draw (-6.88,-3.63)-- (-6.88,-2.09);
\draw (-6.88,-2.09)-- (-6.58,-0.63);
\draw [dash pattern=on 5pt off 5pt] (-6.58,-0.63)-- (-8.58,-0.63);
\draw (-8.58,-0.63)-- (-8.88,-2.09);
\draw (-8.88,-2.09)-- (-9.18,-0.63);
\draw [dash pattern=on 5pt off 5pt] (-9.18,-0.63)-- (-11.18,-0.63);
\draw (-11.18,-0.63)-- (-10.88,-2.09);
\draw (-2.88,-2.09)-- (-3.18,-3.63);
\draw [dash pattern=on 5pt off 5pt] (-3.18,-3.63)-- (-1.18,-3.63);
\draw (-1.18,-3.63)-- (-0.88,-2.09);
\draw (-0.88,-2.09)-- (-0.58,-3.63);
\draw [dash pattern=on 5pt off 5pt] (-0.58,-3.63)-- (1.42,-3.63);
\draw (1.42,-3.63)-- (1.12,-2.09);
\draw (1.12,-2.09)-- (1.12,-0.63);
\draw [dash pattern=on 5pt off 5pt] (1.12,-0.63)-- (-2.88,-0.63);
\draw (-2.88,-0.63)-- (-2.88,-2.09);
\draw (-5.34,1.72) node[anchor=north west] {$S_1$};
\draw (-9.28,-4.18) node[anchor=north west] {$S_2$};
\draw (-1.31,-4.1) node[anchor=north west] {$S_3$};
\begin{scriptsize}
\draw [color=black] (-7.86,3.8)-- ++(-5.0pt,-5.0pt) -- ++(10.0pt,10.0pt) ++(-10.0pt,0) -- ++(10.0pt,-10.0pt);
\draw [color=black] (-5.86,3.8)-- ++(-5.0pt,-5.0pt) -- ++(10.0pt,10.0pt) ++(-10.0pt,0) -- ++(10.0pt,-10.0pt);
\draw [color=black] (-3.86,3.8)-- ++(-5.0pt,-5.0pt) -- ++(10.0pt,10.0pt) ++(-10.0pt,0) -- ++(10.0pt,-10.0pt);
\draw [color=black] (-1.86,3.8)-- ++(-5.0pt,-5.0pt) -- ++(10.0pt,10.0pt) ++(-10.0pt,0) -- ++(10.0pt,-10.0pt);
\draw [color=black] (-10.88,-2.09) circle (5.0pt);
\draw [color=black] (-8.88,-2.09) circle (5.0pt);
\draw [color=black] (-6.88,-2.09) circle (5.0pt);
\fill [color=black] (-2.88,-2.09) circle (5.0pt);
\fill [color=black] (-0.88,-2.09) circle (5.0pt);
\fill [color=black] (1.12,-2.09) circle (5.0pt);
\end{scriptsize}
\end{tikzpicture}
}
\end{minipage}
\caption{A torus cover, its global graph, and the local surfaces} \label{cap:coverGlobal}
\end{figure}
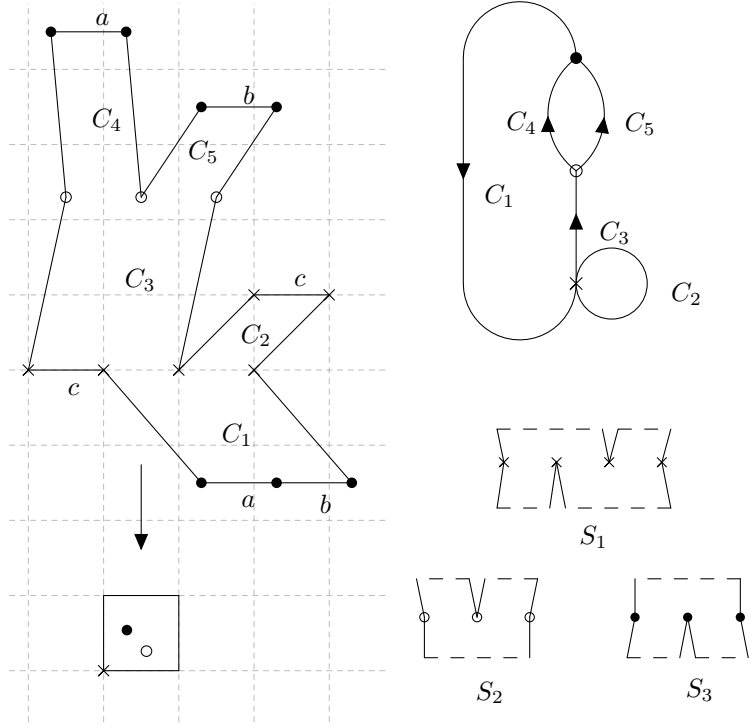
The edges $E(\Gamma)$ of $\Gamma$ are in bijection with the core curves.
An edge~$e$ connects the vertices~$i$ and $j$, if the connected components
of $X^0$ adjacent to the core curve $\gamma(e)$ contain ramification points
lying over the $i$-th and $j$-th branch point in~$E$. Note that this is
well-defined by the branch point normalization, which rules out that 
the $p$-images of two ramification points have the same height. The
case $i=j$, i.e.\ self-edges, is of course possible.
\par
To give an alternative definition, if each $\mu^{(i)}$ in the profile~$\Hmu$ is
a $d$-cycle, the global graph is just the dual graph of the stable curve of the 
curve obtained by degenerating the surface~$X$ in the horizontal direction, 
i.e.\ by applying ${\rm diag}(e^t, e^{-t})$. In the general case, the global
graph is the quotient graph of this dual graph, obtained by identifying 
the vertices whose corresponding branch point have the same number
(in $1,\ldots,n$).
\par
We provide~$\Gamma$ with an orientation as follows and write $G \in \Gamma$ 
for the oriented graph. Fix a oriented closed loop on~$E$ (e.g.\ a vertical 
straight line), intersecting the horizontal straight line once. The preimages 
of this loop are paths in~$X$, each crossing precisely one core 
curve~$\gamma_c$, and we orient the corresponding edge of~$\Gamma$ in the 
direction of this loop. Self-edges are not given any orientation. Note that 
this orientation is well-defined by~$p$ up to flipping all arrows of~$G$.
\par
We call the union of  connected components of $X^0$ that carry the same
label the {\em local surfaces} of $(X,\omega)$. We label these local
surfaces also by an integer in $\{1,\ldots,n\}$ according to the ramification 
point they carry. This labeling is well-defined, since $p$ is a cover without
unramified components.
\par
To reconstruct a torus covering flat surface from a global graph, we need 
two extra data that encode the geometry of the cylinders and the geometry
of the local surfaces, respectively. 
\par
Each cylinder (corresponding to an edge~$e$) has an integral positive 
width $w_e$ and a real positive height~$h_e$. The heights~$h_e$ are not 
arbitrary, but related to the position of the branch points. For an 
edge~$e \in E(\Gamma)$ we denote by $i^+(e)$ (resp.\  $i^-(e)$) 
the label of the terminal (resp.\ initial) vertex of the edge~$e$.  
It is obvious from the construction that the tuple of heights 
$(h_e)_{e \in E(\Gamma)}$ belongs to the {\em height space} 
\be \label{eq:heightspace}
\widetilde{\NN}^{E(G)} \= \{ (h_e)_{e \in E(\Gamma)} \,:\, h_e - \Delta(e)  \in \NN \}\,, 
\ee
where $\Delta(e) =  \ve_{i^+(e)} - \ve_{i^-(e)}$ if $i^+(e) \geq i^-(e)$ and
$\Delta(e) =  1 + \ve_{i^+(e)} - \ve_{i^-(e)}$ otherwise.
\par
The last piece of local information for a cylinder is the {\em twist} 
$t_e \in \ZZ \cap [0,w_e-1]$. The twist depends on the choice 
of a ramification point $P^-(e)$ and $P^+(e)$ in each of the two components 
adjacent to the cylinders and it is defined as the integer part
of the real part $\lfloor  \Re(\int_s \omega) \rfloor$ of the
integral along the unique straight line joining $P^-(e)$ to $P^+(e)$
such that $t_e \in [0,w_e-1]$. The exact values of the twist will hardly 
matter in the sequel. It is important to retain simply that there 
are $w_e$ possibilities for the twist in a given cylinder. 
\par
We now encode the local geometry on the complement of the core curves. The 
restriction of the cover~$p$ to any local surface is metrically a cover of 
an infinite cylinder, branched over one point only. Said differently, for 
any vertex $v \in V(\Gamma)$ the (possibly disconnected) local surface 
$S = S_v$ can be described as a covering $\pi: S \to \PP^1$, ramified 
over $z=0$, $z=1$ and $z=\infty$ only. The restriction of~$\omega$ to~$S$ 
is the pullback of the infinite 
metric cylinder $dz/z$. The ramification profile of~$\pi$ consists 
\begin{itemize}
\item over $z=0$ of the widths $w_e$ of the incoming edges at~$v$, 
\item over $z=\infty$ of the widths $w_e$ of the outgoing edges at~$v$, and 
\item over $z=1$ of the subset of the branching profile~$\mu^{(i)}$, where~$i$
is the label of~$v$.
\end{itemize}
\par
\begin{Prop} \label{prop:corrTCGraph}
There is a bijective correspondence between 
\begin{itemize}
\item[i)] flat surfaces $(X,\omega)$ with covering $p: X \to E$ of 
degree~$d$ of the square torus~$E$ without unramified components and
with $\omega = \pi^* \omega_E$, and 
\item[ii)] isomorphism classes of 
tuples $(G, (w_e,h_e, t_e)_{e \in E(G)}, (\pi_v)_{v\in V(G)})$ consisting of 
\begin{itemize}
\item[$\bullet$] a global graph~$\Gamma$ with marked vertices and
without isolated vertices together with an orientation $G \in \Gamma$,
\item[$\bullet$] a collection of real numbers $(w_e,h_e,t_e)_{e \in E(G)}$ 
representing the width, height and twist of the cylinder corresponding to~$e$. 
The widths $w_e$ are integers, the tuple of heights $(h_e)_{e \in E(G)} \in 
\widetilde{\NN}^{E(G)}$ is in the height space, $t_e \in \ZZ \cap [0,w_e-1]$ 
and these numbers satisfy 
\be \label{eq:whd}
\sum_{e \in E(G)} w_eh_e \= d\,,
\ee
\item[$\bullet$] and a collection of $\PP^1$-coverings $(\pi_v)_{v\in V(G)}$ 
without unramified components, with $\pi_v \in \Cov'(\bfw_v^-,\bfw_v^+,\mu_v)$
 where $\bfw_v^-$ is the tuple of widths at the incoming edges at $v$, $\bfw_v^+$ is 
the tuple of widths at the outgoing edges at $v$, and $\mu_v$ is the ramification
profile given by the labels at the vertex~$v$.
\end{itemize}
up to the action of the group $\Aut(\Gamma)$ of automorphisms of the labeled graph~$\Gamma$.
\end{itemize}
\end{Prop}
\par
Note that an automorphism of the  labeled graph~$\Gamma$ preserves the vertices, i.e.\ 
it simply permutes the sets of edges sharing the same endpoints.
\par 
\begin{proof} To each covering $p$ we can canonically associate the global graph~$\Gamma$
with vertex labels according to the branch point numbering and with widths $w_e$ as above. 
The rest of the correspondence is not canonical but depends on three auxiliary choices. 
First, we provide $\Gamma$ with the orientation~$G$ given by the upwards pointing
vertical direction. Second, we move the branch points so that they satisfy the branch 
point normalization.
Third, we choose for each edge of $\Gamma$ a pair of singularities in the local 
surfaces adjacent to the corresponding cylinder, one on each side. 
\par
We observe that these additional data can also obviously be recorded together with
the tuples listed in ii) so that it suffices to establish a bijection with
the additional data on both sides and to see that the cardinality of the forgetful
maps of the additional data are the same on both sides of the correspondence.
\par
For the second task we remark that from any point in the height space we get
back by reducing mod one the heights up to a common translation (mod one) 
by a real number. But the heights are only well-defined up to  this 
ambiguity of translation anyway. The number of choices of 
the singularities adjacent to each cylinder agrees on both sides of
the correspondence, in fact this number is $\prod_{v \in V(G)} \ell(\bfw_v^-)
\ell(\bfw^+_b)$.
\par
Note that the labeling of the vertices gives a natural partial numbering
of the edges, by proceeding lexicographically, i.e.\ first numbering all
edges between the vertex one and two etc.  The number of choices to complete
this to a full numbering of the edges is precisely $|\Aut(\Gamma)|$.
\par
The correspondence with the additional datum of an edge numbering  
has basically been given prior to
the statement of the proposition: We use the pair of zeros adjacent to each 
cylinder to single out a saddle connection up to 
Dehn twist along the core curve of the cylinder. There is a unique representative
in such a class that has holonomy $t_e + ih_e$ with $t_e \in \ZZ \cap [0,w_e-1]$.
Moreover, we use the numbering of the edges to order the tuple of incoming and
outgoing cylinders in each local surface, i.e.\ to make $\bfw^\pm$ an ordered
tuple rather than a set of integers. This numbering also defines
for every local surface an identification $\sigma_0$ (resp.\ $\sigma_\infty$) of the 
branch points over zero (resp. $\infty)$ with a element in the tuple $\bfw^-$ (resp.\ $\bfw^+$), 
so that local coverings maps $\pi_v$ are in $\Cov'(\bfw_v^-,\bfw_v^+,\mu_v)$ rather than
in the unordered version of this set of coverings. Note also that the prime is justified
here, since the local covering is without unramified components, by the definition
of components of the local surfaces as components of $X^0$.
\par
For the converse correspondence, it suffices to wield together the local surfaces
along a cylinder for each edge of~$G$, where the identifications $\sigma_\infty$ of $v^-(e)$
and $\sigma_0$ of $v^+(e)$ determine which branches of the local surfaces are glued 
together. The widths $w_e$ and heights~$h_e$ determine the shape of the cylinder and
the twist together with the choice of a reference point on each side determines
the way the cylinder is glued in. The action of an element in $\Aut(\Gamma)$ simply
changes the $\sigma_\infty$ and $\sigma_0$ by post-composition. Consequently, any two
tuples in the same  $\Aut(\Gamma)$-orbit give the same covering.
\end{proof}
\par
Similar statements do not hold neither on the level of connected covers of graphs
nor on the level of general graphs without major changes. The problem in the connected
case is that some local surfaces might be disconnected while assembling to a 
connected flat surface in general. The problem in the disconnected case is that 
a covering of a local surface with an unramified component can give rise to a 
flat surface that is also obtained by assembling only connected local surfaces. To
construct the corresponding  graph without unramified components the unramified 
local piece has to be piled on top of the appropriate cylinder and can so be 
gotten rid of.
\par
The above correspondence still gives a bijection if both sides are weighted
with their automorphism group.
\par
\begin{Prop} \label{prop:corrAuto}
In the correspondence of Proposition~\ref{prop:corrTCGraph}, an automorphism $\varphi$ 
of the cover~$\pi$ defines a collection of automorphisms $(\varphi_v)_{v \in V(G)}$ of
the local surfaces. 
\par
Conversely, each collection of automorphisms $(\varphi_v)_{v \in V(G)}$ of
the local surfaces defines an automorphism $\varphi$ of the covering~$\pi$.
\end{Prop}
\par
Again, the correspondence is not canonical but depends on the choice of an 
auxiliary edge labeling.
\par
\begin{proof} 
Any automorphism $\varphi$ of $\pi$ preserves the marked points in~$E$ and hence 
the vertices of $\Gamma$. Moreover, it maps cylinders to cylinders and thus induces an
automorphism $\varphi_\Gamma$ of the graph, preserving vertices.  Let $\ell: E(G) \to \{1,\ldots,|E(G)|\}$ 
be some auxiliary labeling of the edges of $\Gamma$. If we provide edges of the global graph 
of $\pi \circ \varphi$ with the labeling $\ell \circ \varphi_\Gamma$, then the restriction
of $\varphi$ to each local surface $S_v$ is an automorphism $\varphi_v$ that {\em preserves}
the labeling $\sigma_0$ and $\sigma_\infty$ of the preimages of zero and $\infty$. 
\par
The converse of this procedure obviously works as well. 
\end{proof}
\par
\section{Shifted symmetric polynomials and completed cycles} \label{sec:completed}

Let $f:\P\to\QQ$ be an arbitrary function on the set $\P$ of all partitions.
Motivated by the formula~\eqref{eq:NNN} computing the generating function of covers without
unramified components, we associate to~$f$ the formal power series
\be \label{defqbrac} 
  \sbq f \= \frac{\sum_{\l\in\P} f(\l)\,q^{|\l|}}{\sum_{\l\in\P} q^{|\l|}}\;\,\in\;\QQ[[q]]\,,
\ee
which we will call the {\em $q$-bracket}. In the previous section, the argument
was a product of functions $f_{\mu}(\cdot)$ introduced in~\eqref{eq:defssf} and we recall here
an algebra of functions on which $q$-brackets behave nicely, as well as two generating
sets for this algebrea.
\par
The {\em algebra of shifted symmetric polynomials} is defined as 
$\Lambda^* = \varprojlim \Lambda^*(n)$, where $\Lambda^*(n)$ is the algebra of 
symmetric polynomials in the $n$ variables \hbox{$\l_1-1$,\,}$\ldots,\l_n-n$. The 
projective limit is taken with respect to the homomorphisms setting the last variable 
equal to zero.
One of several ways to present a partition is to list the part lengths decreasingly, 
i.e.\ a partition is given by $\l = (\l_1, \l_2,\ldots)$, with $\l_1 \geq \l_2 \geq \cdots$ 
and $\sum_{i=1}^\infty \l_i = |\l|$. With this notation, the functions 
\be \label{eq:def2pk}  P_\ell(\l) \= \sum_{i=1}^\infty \left( (\l_i -i +\h)^\ell - (-i + \h)^\ell \right)  \quad \text{and}\quad
P_\mu \=\prod_i P_{\mu_i} \ee
obviously belong to the algebra symmetric polynomials. It is also convenient to
add constant terms to these function, corresponding to the regularization of the infinite sum, and
we let
\be \label{eq:smallpk}
p_\ell(\lambda) \= P_\ell(\lambda) \+ (1-2^{-\ell})\,\zeta(-\ell).
\ee
The name ``completed cycles'' refers to the functions $P_\ell/\ell$ for the cycles $(\ell)$: they ``complete'' the functions $f_\ell$ defined in~\eqref{eq:defssf}. 
\par
The following result summarizes the main properties of shifted symmetric polynomials we need.
It is a combination of a theorem of Okounkov and Olshanski (\cite{OkOls}) and a theorem
of Kerov and Olshanski.
\par
\begin{Thm}[\cite{KerOls}] \label{thm:KO}
The algebra $\Lambda^*$ is freely generated by all the $p_\ell$ (or equivalently, by the $P_\ell$) 
with~$\ell\ge1$.
The functions $f_\mu$ defined in~\eqref{eq:defssf} belong to $\Lambda^*$. More precisely, 
as $\mu$ ranges over all partitions, these  functions $f_\mu$ form a basis of~$\Lambda^*$.
\end{Thm} 
\par
To convert from the $P_\ell$ to the $f_\mu$, note that $f_1 = P_1$ and $f_2 = P_2/2$, 
and more generally $f_\ell$ starts with $P_\ell/\ell$, for example
\ba
f_3 & \= \frac{1}{3} P_{3} -\frac{1}{2} P_{1}^2  + \frac{5}{12} P_{1}, \quad \quad 
f_4  \= \frac{1}{4}P_4 - P_1P_2 + \frac{11}{8}P_2    \\
f_5 & \= \frac{1}{5}P_5 - P_3P_1 -\frac{1}{2}P_2^2 + \frac{5}{6}P_1^3 - \frac{15}{4}P_1^2  
+ \frac{19}{6}P_3 + \frac{189}{80}P_1\,.    
\ea
We refer to \cite[Section~3.3]{OkPand} or \cite{lassalle} for the conversion 
formulas in general.
\par
At this stage we mention the following important result of Bloch-Okounkov
and refer to Section~\ref{sec:QMF} for the definition of quasimodular forms. Using Theorem~\ref{thm:KO}
one provides the algebra $\Lambda^*$ with a weight grading by assigning $p_\ell$ the weight $k=\ell+1$.
\par
\begin{Thm} [\cite{blochokounkov}]\label{thm:bo}
If $f$ is a shifted symmetric function of weight~$k$, then $\sbq{f}$ is a quasimodular form of weight $k$.
\end{Thm} 
\par
There exists a long list of quite different proofs of this theorem. Already~\cite{blochokounkov}
contains two proofs, one in the spirit of~\cite{KanZag} and one with an explicit formula using determinants
of theta derivatives. A proof based on vertex operators is given in \cite{milasII}. Zagier (\cite{ZagBO})
gave a very short proof, that also provided an efficient recursive method to 
compute $q$-brackets.
\par
In this paper we will not use the Bloch-Okounkov theorem, but rather give yet
another proof, by counting graphs with weights, at the end 
of Section~\ref{sec:GraphSumisQM}. This proof is not our main objective
and the proof is rather roundabout, but it shows that a lot of quasimodularity results
can be ultimately traced back to Theorem~\ref{thm:coeff0QM}. 

\section{Hurwitz numbers and graph sums} \label{sec:doubleHurwitz}

The first part of the section is purely expository and we recall some known (piece-wise) 
polynomiality properties of the cardinalities of the Hurwitz numbers
introduced in~\eqref{eq:defDHN} and its completed cycle variants. Then we combine the definition
of triple Hurwitz numbers with the composition of torus covers into global graphs and local
surfaces to obtain a formula for counting torus covers in terms of a graph count of triple 
Hurwitz numbers. 
\par
\subsection{Triple Hurwitz numbers with completed cycles} \label{sec:DHN}

In general, the generalized double Hurwitz numbers with completed cycles are only piecewise polynomials
in variables $w_i^-$ and $w_i^+$ on the chambers defined by the walls where a partial sum
of the $w_i^-$ agrees with a partial sum of the~$w_i^+$. Recall that the shifted symmetric
function $f_\ell$ satisfies $f_\ell = \tfrac{1}\ell P_\ell + \cdots$. The formal {\em triple 
Hurwitz numbers with completed cycles}
\be 
\ol{A}(\bfw^-,\bfw^+,\mu) \= \frac{1}{\prod_i w^-_i\prod_i w^+_i } 
\sum_{|\lambda| = d} \chi_{\bfw^-}^\lambda
\chi_{\bfw^+}^\lambda  \frac{P_\mu(\lambda)}{\prod \mu_i}
\ee
obtained by replacing $f_\mu$ by the completed cycles $P_\mu  / {\prod \mu_i}$ 
has
much better properties, e.g.\ it is a polynomial outside the walls if
$\mu = (\mu_1)$ is a partition consisting of a single cycle. To remove the jumps on the walls 
we introduce the {\em triple Hurwitz numbers with completed cycles and without unramified
components}
\be \label{eq:olAprimeolA}
\ol{A}'(\bfw^-, \bfw^+, \mu)=\sum_{\substack{\bfu^+\subset\bfw^+, \; \bfu^-\subset\bfw^-\\|\mu|\leq|\bfu^-|=|\bfu^+|}} 
M(\bfu^+,\bfu^-,\bfw^+,\bfw^-)\, \ol{A}(\bfu^-, \bfu^+,\mu)\,.
\ee
by applying the same inclusion-exclusion inversion to $\ol{A}$ as we did in~\eqref{eq:AprimeA} to~$A$. 
The main reason to introduce completed cycles here is the following polynomiality result when $\mu$ is a cycle and triple Hurwitz numbers are in fact 
generalized double Hurwitz numbers.
We learned about this through draft notes of Okounkov.  It can
be combined from results of Shadrin, Spitz and Zvonkine in \cite{SSZ}.
\par
\begin{Thm}\label{thm:SSZ} If $\mu_1 +1-\ell(\bfw^-)- \ell(\bfw^+)$ is even, then 
the triple Hurwitz numbers $\ol{A}'(\bfw^-,\bfw^+, (\mu_1))$ with completed cycles and 
without unramified components for the last argument $\mu = (\mu_1)$ 
being a partition consisting of a single part is an even polynomial
in the variables $w_i^-$ and $w_i^+$.
\par
If $\mu_1+1-\ell(\bfw^-)- \ell(\bfw^+)$ is odd, then $\ol{A}'(\bfw^-,\bfw^+, (\mu_1)) = 0$.
\end{Thm}
\par
\begin{proof} The completed generalized double Hurwitz numbers are piecewise polynomial
functions for any number of ramification points besides the two prescribed ones, but the polynomiality is global for one ramification point, 
as we now explain in detail.
\par
Fix $m=\ell(\bfw^-)$ and $n=\ell(\bfw^+)$. We consider the vector space 
$V=\{(\bfw^-, \bfw^+): \; |\bfw^-|=|\bfw^+|\}$ and for $I\subset\{1, \dots, m\}$ 
and $J\subset\{1, \dots, n\}$ we define the hyperplane 
\[W_{I,J} \= \{(\bfw^-, \bfw^+)\in V; \; |\bfw_I^-|-|\bfw_J^+|=0\}.\]
The sets $W_{I,J}$ are the walls of a hyperplane arrangement. In the interior of 
the chambers, the connected and disconnected Hurwitz numbers obviously coincide. 
For a chamber $\mathfrak c$ of this arrangement, Theorem~6.4 of \cite{SSZ} shows that 
$\ol{A}(\bfw^-,\bfw^+,(\mu_1))_{|\mathfrak c}$ is a homogeneous polynomial of degree 
$\mu_1 + 1-\ell(\bfw^-)- \ell(\bfw^+)$. 
\par
The wall crossing formula (Theorem~6.6 of~\cite{SSZ}) of two adjacent 
chambers~$\mathfrak c_1$ and~$\mathfrak c_2$ of the wall~$W_{I,J}$ can be written 
in the case of the last argument $\mu = (\mu_1)$ being a partition consisting of 
a single part as
\bas
&\phantom{\=\, } \ol{A}'(\bfw^-,\bfw^+,(\mu_1))_{|\mathfrak c_1}
-\ol{A}'(\bfw^-,\bfw^+,(\mu_1))_{|\mathfrak c_2} \\
&\= \delta^2\biggl(\ol{A}'(\bfw^-_I,\bfw^+_J+\delta,(\mu_1))\, 
\ol{A}'(\bfw^-_{I^c}+\delta,\bfw^+_{J^c},\emptyset) \\ 
&\phantom{\=\, \delta^2\biggl( } \+ \ol{A}'(\bfw^-_{I^c}+\delta,\bfw^+_{J^c},(\mu_1))\,
\ol{A}'(\bfw^-_{I},\bfw^+_{J}+\delta,\emptyset)\biggr)\,,
\eas
where $\delta=|\bfw^-|-|\bfw^+|$. In fact, \cite{SSZ} state this formula in term of $\ol{A}$ instead of $\ol{A}'$, but for $\delta=0$ the formula holds trivially by \eqref{eq:olAprimeolA}, and outside the walls the covers have no unramified component. In this expression the terms with no ramification vanish
since $\ol{A}'$ denotes covers without unramified components.
This implies that the polynomials are the same in any two adjacent chambers 
and hence the expression is globally polynomial. 
\par
Moreover, Theorem~6.4 of \cite{SSZ} implies that the polynomial 
$\ol{A}'(\bfw^-,\bfw^+,(\mu_1))$
has the same parity as $\mu_1+1-\ell(\bfw^-)-\ell(\bfw^+)$. It remains to show that the 
polynomial
vanishes when this expression is odd. The triple Hurwitz number without 
completed cycles $A'(\bfw^-,\bfw^+,({\mu_1}))$ vanishes for 
$\mu_1 +1-\ell(\bfw^-)-\ell(\bfw^+)$ odd
by the Riemann-Hurwitz formula and the same statement holds for $f_{\mu_1}$ 
replaced by any $f_\mu$ with ${\rm wt}(\mu)$ odd. Since any $p_\mu$ of odd (resp.\ even) 
weight is a linear combination of $f_\mu$ of   odd (resp.\ even) weight (see \cite{OkPand}, 
Formula~(0.22) for the general statement) the claim follows.
\end{proof}
\par
Later we will need to allow more general functions on partitions, and hence we 
define for any function~$F$ on partitions
\be \label{eq:AwithF}
A(\bfw^-,\bfw^+,F) \= \frac{1}{\prod_i w^-_i\prod_i w^+_i } 
\sum_{|\lambda| = d} \chi_{\bfw^-}^\lambda
\chi_{\bfw^+}^\lambda  F(\lambda)
\ee
and we define $A'(\bfw^-,\bfw^+,F)$ in terms of $A(\bfw^-,\bfw^+,F)$ as in~\eqref{eq:olAprimeolA}.
In this notation we retrieve the previous definition of triple Hurwitz numbers as
\bes
A'(\bfw^-,\bfw^+,\mu) = A'(\bfw^-,\bfw^+,f_\mu) \quad \text{and} \quad
\ol{A}'(\bfw^-,\bfw^+,\mu) = A'\Bigl(\bfw^-,\bfw^+,\frac{P_\mu}
{\prod \mu_i}\Bigr)\,.
\ees
\par

\subsection{Graph sums with triple Hurwitz numbers} \label{sec:GrDHN}

The first goal here is to use Proposition~\ref{prop:corrTCGraph} and Proposition~\ref{prop:corrAuto}  
to write the generating 
series $N'(\Hmu)$ for counting torus covers without unramified components in terms of graph sums 
involving triple Hurwitz numbers without unramified components. First of all, we can
decompose~$N'(\Hmu)$ according to the contribution of the individual graphs, i.e.\ 
$$  N'(\Hmu) \= \frac{1}{|\Aut(\Gamma)|} \, \sum_{\Gamma} N'(\Gamma, \Hmu) \,,$$
where the sum is over all (not necessarily connected) labeled 
graphs~$\Gamma$ with $n = |\Hmu|$ vertices and where $\Aut(\Gamma)$ are the automorphisms
of the graph~$\Gamma$ that respect the vertex labeling. (Note that $\Gamma$ has neither
a labeling nor an orientation on the edges.)
\par
\begin{Prop} \label{prop:N'exp}
The contributions of individual labeled graphs to $N'(\Hmu)$ can be expressed
in terms of triple Hurwitz numbers as
\be \label{eq:NpGamma1} 
N'(\Gamma, \Hmu) \= \sum_{G \in \Gamma } {N}'(G,\Hmu)\,, 
\ee
where 
 \be \label{eq:NpGamma2}
{N}'(G,\Hmu) \= \!\!\sum_{h\in \widetilde{\NN}^{E(G)}, \atop w\in \Z_+^{E(G)}}\prod_{e\in E(G)}w_eq^{h_e w_e}
\prod_{v\in V(G)}A'(\bfw_v^-,\bfw_v^+,\mu_v)\,\, \delta(v)
\ee
and where \be \label{eq:NotDelta}
\delta(v) = \delta\bigl(\sum_{i\in e_+(v)} w_i^+ - \!\sum_{i\in e_-(v)} \!w_i^-\bigr).\ee
\end{Prop}
\par
The delta-function factor is
redundant in this expression by our definition of $A'(\bfw_v^-,\bfw_v^+,\mu_v)$, but keeping
it will be important once we pass from $A'$ to a polynomial expression.
\par
\begin{proof} This is a direct consequence of the correspondence 
in Proposition~\ref{prop:corrTCGraph}. The degree of the covering is encoded
in the $(w_e, h_e)$ by~\eqref{eq:whd} and the factor $w_e$ accounts 
for the number of possible twists (values of $t_e$) for any given edge~$e$.
\end{proof}
\par
The strategy to prove quasimodularity is to reduce to ``expressions as in~\eqref{eq:NpGamma1} 
and~\eqref{eq:NpGamma2}'' but with $\ol{A}'$ as argument, which is polynomial by 
Theorem~\ref{thm:SSZ}.  To formalize this, recall from the combination of~\eqref{eq:NNN}
and~\eqref{defqbrac} that the counting function $N'(\Hmu)$ is a $q$-bracket of a shifted
symmetric function. By Theorem~\ref{thm:KO} it hence suffices to treat $q$-brackets of 
products of the $p_\ell$.
\par
\begin{Thm} \label{thm:bracketgraphsum}
The $q$-bracket of any shifted symmetric function can be
expressed as a graph sum 
\be  
\bq{p_{\ell_1}\cdots p_{\ell_n}}  \= \sum_{\Gamma} \sum_{G \in \Gamma}\, \bG{p_{\ell_1}\cdots p_{\ell_n}} 
\ee
where the sum runs over graphs with $n$ labeled vertices and all orientations~$G$ of~$\Gamma$,
and where 
\be \label{eq:pasGraphsum}
\bG{p_{\ell_1}\cdots p_{\ell_n}} \= \!\!\sum_{h\in \widetilde{\NN}^{E(G)}, \atop w\in \Z_+^{E(G)}}
\prod_{e\in E(G)}w_eq^{h_e w_e}
\prod_{v\in V(G)}\ol{A}'(\bfw_v^-,\bfw_v^+,(\ell_{\# v}))\,\, \delta(v)\,.
\ee
Here $\# v$ denotes the label of the vertex~$v$.
\end{Thm}
\par
\begin{proof} Our strategy is to reduce this to the cases covered by 
Proposition~\ref{prop:N'exp}. For this purpose, we define for any 
function~$F$ on partitions the auxiliary brackets
\be \label{eq:defauxbrack}
[F_1,\ldots,F_n] \=  \sum_{\Gamma} \, [F_1,\ldots,F_n]_\Gamma\,, 
\quad [F_1,\ldots,F_n]_\Gamma \= \sum_{G \in \Gamma } \, [F_1,\ldots,F_n]_G
\ee
where the sum is over all labeled graphs~$\Gamma$ with $n$~vertices and
over all the orientations, respectively, and where 
\be
[F_1,\ldots,F_n]_G \= \sum_{h\in \widetilde{\NN}^{E(G)}, \atop w\in \Z_+^{E(G)}}\prod_{i\in E(G)}w_iq^{h_iw_i}
\prod_{v\in V(G)} A'(w_v^-,w_v^+,F_{\#v})\,\, \delta(v)\,.
\ee
In this notation, we want to show that
\be  \label{eq:pasbracket}
\bq{p_{\ell_1}\cdots p_{\ell_n}} \= [p_{\ell_1},\ldots,p_{\ell_n}]\,. 
\ee
On the other hand, Proposition~\eqref{prop:N'exp} can be restated in this notation as 
\be \label{eq:fasbracket}
\bq{f_{\mu_1}\cdots f_{\mu_n}} \= [f_{\mu_1},\ldots,f_{\mu_n}]\,.
\ee
We can express by Theorem~\ref{thm:KO} each of these generators~$p_\ell$ 
of~$\Lambda^*$ as $p_\ell = \sum_\mu  c_{\ell,\mu} f_\mu$ for some coefficients~$c_{\ell,\mu}$.
Obviously, a $q$-bracket of a product~$n$ shifted symmetric functions is multilinear
in the $n$~arguments. On the other hand, the brackets introduced in~\eqref{eq:defauxbrack}
are multilinear as well, since the arguments of the bracket appear only linearly as
arguments of~$A'$. Consequently, equation~\eqref{eq:pasbracket} is a linear combination
of equations of the form~\eqref{eq:fasbracket}.
\end{proof}
\par

\section{Constant coefficients of 
quasi-elliptic functions} 
\label{sec:QMFProp}

\subsection{Quasimodular forms} \label{sec:QMF}
Kaneko and Zagier introduced the quasimodular forms in~\cite{KanZag} in 
connection with counting simply branched covers of the torus. 
\par
A {\em quasimodular form} for the cofinite Fuchsian group $\Gamma \subset 
\SL\RR$ of weight~$k$ is a function $f: \HH \to \CC$ that is holomorphic
on $\HH$ and the cusps of $\Gamma$ and such that there exists and integer~$p$ and
holomorphic functions $f_i: \HH \to \CC$ such that 
$$ (c\tau + d)^{-k} f\Bigl(\frac{a\tau +b }{c\tau +d}\Bigr) \=  
\sum_{i=0}^p f_i(\tau) \Bigl(\frac{c}{c\tau + d} \Bigr)^i\,,$$ 
for all $\sm abcd \in \Gamma$. 
\par
Note that this definition implies (using the identity matrix)
that $f_0 = f$. The smallest integer~$p$ with the above property
is called the {\em depth} of the quasimodular forms.  By definition, 
quasimodular forms of depth zero are simply modular forms.
The basic examples of quasimodular forms are the Eisenstein series defined by
\[G_{2k}(\tau) \= \frac{(2k-1)!}{2(2\pi i)^{2k}}\sum_{(m,n)\in\Z^2\setminus\{(0,0\}}
\frac{1}{(m+n\tau)^{2k}} \= -\frac{B_{2k}}{4k}+\sum_{n=1}^{\infty}\sigma_{2k-1}(n)q^n\,.\]
Here  $B_l$ is the Bernoulli number, $\sigma_l$ is the divisor sum function
and $q=e^{2\pi i\tau}$. For $k\geq 2$ these are modular forms, while for $k=2$ the
Eisenstein series 
\[G_2(\tau)=-\frac{1}{24}+\sum_{n=1}^{\infty}\sigma_1(n)q^n\] 
is a quasimodular form of weight~$2$ and depth~$1$ for $\SL\ZZ$. Note that 
the $q$-expansion makes sense as a definition for $G_{2k+1}$ but does not give 
a quasimodular form. We will encounter this power series in 
Section~\ref{sec:examples}.
\par
In terms of Eisenstein series, we recall a characterization of
quasimodular forms, that might serve alternatively a definition of that ring
for the special case of the modular group.
\par
\begin{Prop}[\cite{KanZag}] The ring of quasimodular forms for $\G= \SL\ZZ$ 
is equal to
$\C[E_2,E_4,E_6]$, the polynomial ring over $\C$ generated by the first 
three Eisenstein series. This ring is stable under the $q$-derivative
$D_q = q\tfrac{\partial}{\partial q}=\tfrac{1}{2\pi i}\tfrac{\partial}
{\partial \tau}$. More precisely, the $q$-derivative of a quasimodular 
form of weight~$k$ is a quasimodular form of weight $k+2$.
\end{Prop}

\subsection{Coefficients of a two-variable Jacobi form}

The main player of this section is a function $F_\tau(u,v)$ in two 
``Jacobi'' variables $u$, $v$, that was used by Zagier (\cite{ZagPeriods})
in connection with periods of modular forms. For its definition we use 
the  genus~$1$ Jacobi theta function
\[\theta(u)\=\theta(u;\tau)=\sum_{n\in\Z}(-1)^n
q^{\frac{1}{2}\left(n+\frac{1}{2}\right)^2} e^{(n+\tfrac12)u}\,.\] 
We then let
\[F_\tau(u,v)\=\frac{\theta(u+v)\theta'(0)}{\theta(u)\theta(v)}\,\]
where we denote by prime the $u$-derivative $f'(u)=\tfrac{\partial f}{\partial u}$.
The main feature of~$f$ is that we know both its Fourier expansion, which
we will connect to the counting functions we are interested in, and
the Laurent expansion, which is the main tool to prove quasimodularity statements.
\par
\begin{Thm} \label{thm:propofF} (\cite[Theorem~3.1]{ZagPeriods})
The Fourier development of $F_\tau(u,v)$ is 
\be
F_\tau(u,v)\= \frac{1}{2}\left(\coth\frac{u}{2}+\coth\frac{v}{2}\right)
\,-\,2\sum_{n=1}^{\infty}\Bigl(\sum_{d|n}\sinh\left( du+\frac{n}{d}v\right)\Bigr)q^n\,.
\ee
Its Laurent series expansion is
\be 
F_\tau(u,v)\= \frac{1}{u} + \frac{1}{v} -2 \sum_{r,s=0}^\infty D_q^{\min(r,s)}
G_{|r-s|+1}(\tau) \frac{u^r}{r!}\frac{v^s}{s!}\,.
\ee
The function $F_\tau(u,v)$ has the elliptic transformation property
\be\label{eq:ellipticF}
F_\tau(u+2\pi i(n\tau +s), v + 2\pi i (m\tau + r)) \= q^{-mn}\zeta^{-m} \eta^{-n} 
F_\tau(u,v)\,,
\ee
for all $m,n,r,s \in \ZZ$, where $\zeta = e^u$ and $\eta = e^v$, 
and the modular transformation property
\be \label{eq:modF}
F_{\frac{a\tau +b}{c \tau + d}} \bigl(\frac{u}{c\tau +d}, \frac{v}{c\tau + d}\bigr)
\= (c\tau +d) e^{\frac{cuv/2\pi i}{c\tau +d}} F_\tau(u,v)
\ee
for all $\sm abcd \in \SL\ZZ$.
\end{Thm}
\par
In the sequel, we will use three functions derived from coefficients of~$F_\tau(u,v)$, 
namely
\be
Z(z) \= -[v^0]F_\tau(u,v), \quad P(z) \ = Z'(z), \quad 
L(z)  \= -[v^1]F_\tau(u,v)+\frac{1}{12}
\ee
where  $u=2\pi i z$. The first one, $Z(z) = -\tfrac{\theta'(2\pi iz)}{2\pi i
\theta(2\pi i z)} = 
-\zeta(z)/2\pi i+2G_22\pi iz$ is the classical Weierstra\ss\ $\zeta$-function up to normalization and an additive term, 
the second is $P(z) = \tfrac{1}{(2\pi i)^2}\wp(z) + 2G_2$ is the
Weierstra\ss\ $\wp$-function up to normalization and an additive term. The
last one has no classical name but it is the function that makes the
extension to Siegel-Veech weighted counting work. A direct consequence
of Theorem~\ref{thm:propofF} are the Fourier developments in the domain 
$|q|<|\zeta|<1$
\ba\label{eq:FourierLZP} 
Z(z) & \=\frac{1}{2}+\sum_{k\geq 1}(\zeta^k+\sum_{n\geq 1}q^{nk}
(\zeta^k-\zeta^{-k}))\\
P(z) &\=\sum_{k\geq 1}(k\zeta^k+\sum_{n\geq 1}kq^{nk}(\zeta^k+\zeta^{-k})) \\
L(z) &\=\sum_{k\geq 1, n\geq 1} nq^{nk}(\zeta^k+\zeta^{-k})
\ea
and the Laurent series developments 
\ba\label{eq:LaurentLZP}
Z(z) & \=-\frac{1}{u}+2\sum_{k=0}^{\infty}\frac{G_{2k+2}}{(2k+1)!}u^{2k+1}\\
P(z) & \= \frac{1}{u^2}+2\sum_{k=0}^\infty \frac{G_{2k+2}}{(2k)!}  u^{2k} \\
L(z) & \=2G_2+\frac{1}{12}+2\sum_{k=1}^{\infty}\frac{D_qG_{2k}}{(2k)!}u^{2k}
\ea
of our special functions.
\par
We refer to the shift $P$ of the Weierstrass-$\wp$-function 
as the {\em propagator}. (This terminology is used in e.g.\ \cite{Dijk}, 
\cite{BBBM}. It goes back to \cite{BCOV} and the function~$P$ is hence
also called BCOV-propagator in e.g.~\cite{LiFeynman}.)

\subsection{Quasimodular forms as constant coefficients of  
quasi-elliptic functions}  \label{sec:intEF}

We proceed with our main criterion for quasimodularity, 
involving the constant coefficients of products of the functions $Z$, $P$ and $L$
introduced above, and its derivatives. We start with a general remark
on the domains where the expansions are valid.
Suppose that the meromorphic function $f(z_1, z_2,\dots, z_n; \tau)$ 
is periodic under $z_j \mapsto z_j + 1$ for each~$j$ and under 
$\tau \mapsto \tau +1$. We can then write $f(z_1, z_2,\dots, z_o;\tau) 
= \ol{f}(\zeta_1, \ldots, \zeta_n,q)$ where $\zeta_j = e^{2\pi i z_j}$ as above.
For any permutation $\pi \in S_n$ we fix the domain
\be\label{eq:order}
 \Omega_{\pi} \= |q\zeta_{\pi(i+1)}|<|\zeta_{\pi(i)}|<|\zeta_{\pi(i+1)}|<1 \quad 
\text{for all~$i=1,\ldots,n-1$}\,.
\ee
On such a domain the {\em constant term  with respect to all the $\zeta_i$}
is well-defined. It can be expressed as integral
$$[\zeta_n^0, \ldots,\zeta_1^0]_{\pi} \, \ol{f} \= \frac{1}{(2\pi i)^n}\,
\oint_{\gamma_n}\dots \oint_{\gamma_1}f(z_1, \dots, z_n;\tau)dz_1\dots dz_n
$$
along  the integration paths 
\[\gamma_j: [0,1]\to\C, \quad  t\mapsto iy_j+t\,,\]
where  $0\leq y_{\pi(1)}< y_{\pi(2)}<\dots y_{\pi(n)}< 1$. 
We call these our {\em standard integration paths} for the permutation~$\pi$.
If the domain $\Omega_\pi$ is clear from the context we also write
$[\zeta^0]$ or $[\zeta_n^0, \ldots,\zeta_1^0]$ as shorthand for the 
coefficient extraction $[\zeta_n^0, \ldots,\zeta_1^0]_\pi$.
\par
\medskip
Our aim here is to show Theorem~\ref{thm:coeff0QM} that for a large class of 
functions that are quasimodular and quasi-elliptic in $(z_1,\ldots,z_2)$ 
(in a sense made precise below) the constant term with respect to the 
$\zeta_i$ is a quasimodular form.
\par
\medskip
We start with a preliminary definition of a ring of multi-variable Jacobi 
forms. For $n \geq 0$ we let $\mathcal{J}_n^{(k)}$
be the vector space of meromorphic functions~$f$ on 
$\CC^n \times \HH$ in the variables $(z_1,\ldots,z_n; \tau)$ that 
\begin{enumerate}
\item[i)] have poles on $\CC^n$ at most at the 
$(\ZZ + \tau\ZZ)$-translates of the diagonals $z_i - z_j$,
\item[ii)]  are elliptic with respect to the lattice $\ZZ + \tau \ZZ$ in 
the variables $z_i$ for $i=1,\ldots,n$, and
\item[iii)]  are quasi-modular of weight~$k$ for $\SL\ZZ$, i.e.\
$f$~is holomorphic in $\tau$ on $\HH \cup \infty$ and 
there exists some $p \geq 0$ (called {\em depth}) and functions 
$f_i(z_1,\ldots,z_n; \tau)$ that are holomorphic in~$\tau$ and meromorphic in
the~$z_i$ such that
$$ (c\tau + d)^{-k} f\Bigl(\frac{z_1}{c\tau +d},\ldots, \frac{z_n}{c\tau +d}; 
\frac{a\tau +b }{c\tau +d}\Bigr) \=  
\sum_{i=0}^p f_i(z_1,\ldots,z_n; \tau) \Bigl(\frac{c}{c\tau + d} \Bigr)^i\,,$$ 
for all $\sm abcd \in \SL\ZZ$.
\end{enumerate}
\par
Again, taking the identity matrix in this definition implies that $f_0=f$. 
\par
\begin{Prop} \label{prop:Jring}
The direct sum
$$ \mathcal{J}_n \= \oplus_{k \geq 0} \mathcal{J}_n^{(k)}$$
is a graded ring. The derivatives $\partial/\partial z_i$ map 
$\mathcal{J}_n^{(k)}$ to $\mathcal{J}_n^{(k+1)}$ for all $i=1,\ldots,n$
and the derivative $D_q = q\tfrac{\partial}{\partial q}$ maps
$\mathcal{J}_n^{(k)}$ to $\mathcal{J}_n^{(k+2)}$.
\par
The $m$-th derivative of the propagator ${P}^{(m)}(z_i-z_j)$ 
lies in the graded piece $\mathcal{J}_n^{(m+2)}$ of weight $m+2$.
Moreover, $Z_{ijn} = Z(z_n-z_i) + Z(z_j-z_n)  + Z(z_i-z_j)$
lies in the graded piece $\mathcal{J}_n^{(1)}$ of weight one.
\end{Prop}
\par
\begin{proof}
The first two statements obviously follow from the definition and 
differentiation of the quasimodular transformation property. For the 
third statement all but the quasimodularity are well-known (and follow also
from Theorem~\ref{thm:propofF}). The quasimodular transformation property~iii) 
follows from~\eqref{eq:modF} using  $P_1(z_1,\ldots,z_n;\tau) = 1$ and $P_0 = P$.
For the last statement, the ellipticity of $Z_{ijn}$ follows from the properties
$$Z(z+1) \=Z(z)  \quad \text{and} \quad  Z(z+\tau) \=Z(z)+1$$
of the individual summands. Equation~\eqref{eq:modF} again implies 
that $Z(z_i-z_j)$ is quasimodular (in the sense of iii)) 
of weight one and depth one with $Z_i(z_i-z_j) = z_i-z_j$. This 
in turn implies that $Z_{ijn}$ is even quasimodular of weight one and depth zero, 
i.e.\ modular.
\end{proof}
\par
\begin{Prop} \label{prop:Jaddbasis}
For $n=0$ and $n=1$ the ring $\mathcal{J}_n$ consists
only of the quasimodular forms in~$\tau$.
\par
For $n \geq 2$ the ring $\mathcal{J}_n$ is generated as a 
$\mathcal{J}_{n-1}$-module
by the derivatives of the $P$-function  $P^{(m)}_j = P^{(m)}(z_n-z_j)$ for 
all $m \geq 0$ and all $j=1,\ldots,n-1$ and by the linear combinations
$Z_{ijn} = Z(z_n-z_i) + Z(z_j-z_n)  + Z(z_i-z_j)$ for all $1 \leq i < j \leq n-1$. 
\par
More precisely, if~$f \in \mathcal{J}_n^{(k)}$ then we can write
$$ f = \sum a_{m,j} P^{(m)}_j + \sum b_{i,j} Z_{ijn} + c$$
with $a_{m,j} \in \mathcal{J}_{n-1}^{(k-m-2)}$,  $b_{i,j} \in \mathcal{J}_{n-1}^{(k-1)}$ 
and $c \in \mathcal{J}_{n-1}^{(k)}$.
\end{Prop}
\par
\begin{proof} For the first statement we simply note that an elliptic function
without poles is constant. For the second statement we argue by induction
on the pole orders and we may assume that $f$ is homogeneous of weight~$k$. 
Suppose that $f$ has a pole of order~$m$ exactly at $z_n-z_j$. Then the limit
$$f^{[j]}(z_1,\ldots,z_{n-1}) = \lim_{\ve \to 0} \ve^m f(z_1, \ldots,z_j,\ldots,z_{n-1},z_j+\ve)$$
exists and is non-zero. We claim that $f^{[j]} \in \mathcal{J}_{n-1}^{(k-m)}$. 
Conditions i) and ii)
are obvious from the definition. If $f_i$ denotes a component of~$f$ in the
modular transformation~iii), we define $f_i^{[j]}(z_1,\ldots,z_{n-1})$ in the same way
as above for $f=f_0$. This transformation implies
\ba 
& \phantom{\=}  (c\tau + d)^{-k+m} \Bigl(\frac{\ve}{c\tau+d}\Bigr)^m 
f\Bigl(\frac{z_1}{c\tau +d},\ldots,\frac{z_{n-1}}{c\tau +d}, \frac{z_j + \ve}{c\tau +d}; 
\frac{a\tau +b }{c\tau +d}\Bigr) \\ 
&\=   
\ve^m \sum_{i=0}^d f_i(z_1,\ldots,z_{n-1},z_j + \ve; \tau) \Bigl(\frac{c}{c\tau + d} \Bigr)^i\,,
\ea 
and taking the limit $\ve \to 0$ gives the quasimodularity of~$f^{[j]}$.
\par
If $m \geq 2$ then replacing $f$ by $f - f^{[j]} P^{(m-2)}_j (-2\pi i)^m/(m-1)!$ 
decreases the pole order
along $z_n=z_j$ and does not increase the pole order along any of the divisors
$z_n=z_i$ for $i \geq 2$. 
\par
Inductively we may thus suppose that $f$ has at most simple poles along all
the divisors $z_n-z_j$. The residue theorem (for~$f$ considered as function 
in~$z_n$) implies that $\sum_{j=1}^{n-1} f^{[j]} = 0$. Consequently,
$$ g \= f - \sum_{j=2}^{n-1} \,\biggl(\sum_{i=1}^{j-1} f^{[i]} \biggr) \,Z_{j-1,j,n} $$
is still elliptic and has no poles (considered as function in $z_n$). 
This implies that $g \in \mathcal{J}_{n-1}^{(k)}$ and completes 
the inductive argument.
\end{proof}
\par
One is tempted to deduce from this that if $f \in \mathcal{J}_n^{{k}}$ then for 
any permutation~$\pi$ the constant term on $\Omega_\pi$ is a quasimodular 
form of pure weight~$k$. This, however {\em is not true!} In general, 
those constant terms are of mixed weight $\leq k$, as we will see
in the example in Section~\ref{sec:Ex1111Pcomp}. Even if one is only
interested in computing the constant terms of elements in the ring $\mathcal{J}_n^{(k)}$, 
the fact that 
\be \label{eq:Pint}
[\zeta^0_n]P^{(m)}(z_n-z_j) \= 0 \quad (m \geq 0)
\ee
but 
$$[\zeta^0_n]Z_{ijn} \= Z(z_i-z_j)$$
is no longer elliptic forces the consideration of the following more general 
case of quasi-elliptic functions.
\par
\medskip
Let $\Delta = \Delta_\tau$ be the operator on meromorphic functions defined by
$$ \Delta(f)(z) \= f(z+\tau) - f(z)\,.$$
A meromorphic function~$f$ is called {\em quasi-elliptic} if $f(z+1) = f(z)$
and if there exists some positive integer~$e$ such that $\Delta^e(f)$ is elliptic.
The minimal such~$e$ is called the {\em order (of quasi-ellipticity)} of~$f$.
\par
We say that a meromorphic function $f: \CC^n \times \HH \to \CC$ 
is quasi-elliptic, if it is quasi-elliptic in each of the first~$n$ variables.
For such a function we write $\bfe = (e_1,\ldots,e_n)$ for the tuple of
orders of quasi-ellipticity in the $n$~variables. Consequently, 
a quasi-elliptic function of order $(0,\ldots,0)$ is simply an
elliptic function.
\par
We write $\Delta_i$ for the operator $\Delta$ acting on the $i$-th variable.
Note that these operators~$\Delta_i$ commute.
We now state the appropriate generalization of the definition 
of $\mathcal{J}_n^{(k)}$ above.
\par 
\begin{Defi}
We define for $n \geq 0$, $k \geq 0$ and $\bfe \geq 0$
the vector space $\cQQ_{n,\bfe}^{(k)}$ of {\em quasi-elliptic quasimodular forms}
to be the space of meromorphic functions~$f$ on 
$\CC^n \times \HH$ in the variables $(z_1,\ldots,z_n; \tau)$ that 
\begin{enumerate}
\item[i)] have poles on $\CC^n$ at most at the $\ZZ + \tau\ZZ$-
translates of the diagonals $z_i - z_j$,
\item[ii)] are quasi-elliptic of order $\bfe$, and
\item[iii)]  are quasimodular of weight~$k$ for $\SL\ZZ$, i.e.\
$f$~is holomorphic in $\tau$ on $\HH \cup \infty$ and 
there exists some $p \geq 0$ (called {\em depth}) and functions 
$f_i(z_1,\ldots,z_n; \tau)$ that are holomorphic in~$\tau$ and meromorphic in
the~$z_i$ such that
$$ (c\tau + d)^{-k} f\Bigl(\frac{z_1}{c\tau +d},\ldots, \frac{z_n}{c\tau +d}; 
\frac{a\tau +b }{c\tau +d}\Bigr) \=  
\sum_{i=0}^p f_i(z_1,\ldots,z_n; \tau) \Bigl(\frac{c}{c\tau + d} \Bigr)^i\,,$$ 
for all $\sm abcd \in \SL\ZZ$.
\end{enumerate}
\end{Defi} 
\par
We remark that, contrary to the case of (usual) quasimodular forms, 
the functions~$f_i$ do not need to belong to any of the 
spaces~$\cQQ_{n,\bfe}^{(k)}$.
\footnote{
For example, already for $Z(z_i-z_j)$ we have seen $Z_1 = z_i-z_j$, 
which is not $1$-periodic. Our definition that imposes $1$-periodicity 
(and thus breaks the symmetry in the role of the  $\ZZ + \tau\ZZ$-lattice)
is the reason for this. A more conceptual definition would be to
consider functions that are annihilated by some power of both
shift operators~$\Delta_\tau$ and~$\Delta_1$ with respect to every variable.
E.g.\ the function $z_i-z_j$ has this property. Since we do not aim 
e.g.\ for an operation of ${\mathfrak sl}(2)$ on our ring of quasi-elliptic 
quasimodular forms, we will not discuss this generalization here.}
\par
\begin{Prop} \label{prop:Qring}
The direct sum
$$ \cQQ_n \= \bigoplus_{k \geq 0} \cQQ_n^{(k)}, \quad \text{where} 
\quad \cQQ_n^{(k)} \= \bigoplus_{\bfe \geq 0}\, \cQQ_{n,\bfe}^{(k)} \,,$$
is a graded ring. The derivatives $\partial/\partial z_i$ map 
$\cQQ_n^{(k)}$ to $\cQQ_n^{(k+1)}$ for all $i=1,\ldots,n$.
and the derivative $D_q = q\tfrac{\partial}{\partial q}$ maps
$\cQQ_n^{(k)}$ to $\cQQ_n^{(k+2)}$.
\par
For all $i,j \in \{1,\ldots,n\}$ the functions
\be \label{eq:LZ2}
L(z_i-z_j) \= -\frac{1}{2}Z^2(z_i-z_j) + 
\frac{1}{2}P(z_i-z_j)-G_2+\frac{1}{12}
\ee
belong to $\cQQ_n^{(0)} \oplus \cQQ_n^{(2)}$.
\end{Prop}
\par
\begin{proof} For the proof that $\cQQ_n$ is closed under multiplication
we remark that $\Delta_i^m (fg)$ is a sum of products of $\Delta_i^{m_1} (f)$
and $\Delta_i^{m_2} (g)$ with $m_1 + m_2 = m$, evaluated at arguments translated
by multiples of $\tau$. Consequently, if $f \in \cQQ_{n,\bfe_1}^{(k_1)}$ and 
$g \in \cQQ_{n,\bfe_2}^{(k_2)}$, then $fg \in  \cQQ_{n,\bfe_1 + \bfe_2}^{(k_1 + k_2)}$. 
The property of the derivatives can be checked as in the case of 
$\mathcal{J}_n$.
\par
For the last identity we use the elliptic transformation 
law~\eqref{eq:ellipticF}
to deduce that $L+\frac{Z^2}{2}$ is elliptic.
The first terms of the Laurent series of this function are
\[L(z)+\frac{Z^2(z)}{2} \= \frac{1}{2u^2}+\frac{1}{12}+O(z^2)\,.\]
Consequently, $L+\frac{Z^2}{2}-\left(\frac{1}{2}P-G_2+\frac{1}{12}\right)$ is an 
elliptic function with no poles, i.e.\ it is constant with respect to~$z$. 
Moreover this constant is the value at $z=0$, which is 0. 
\end{proof}
\par
\begin{Prop} \label{prop:Naddbasis}
The vector space $\cQQ_n$ is (additively) generated as $\cQQ_{n-1}$-module  
by the functions $Z^e(z_n-z_j) P^{(m)}(z_n-z_j)$ for $j=1,\ldots,n-1$ 
and for all $e \geq 0$ and $m \geq 0$.
\par
More precisely, if~$f \in \cQQ_n^{(k)}$ then we can write
$$ f(z_1,\ldots,z_n) \= \sum_{e,m,j} a_{e,m,j} Z^e(z_n-z_j) P^{(m)} (z_n-z_j)  + 
\sum_{e,j} b_{e,j} Z^e(z_n-z_j) + c$$
with $a_{e,m,j} \in \cQQ_{n-1}^{(k-e-m-2)}$, $b_{e,j} \in \cQQ_{n-1}^{(k-e)}$
and $c \in \cQQ_{n-1}^{(k)}$.\end{Prop}
\par
\begin{proof}
For every~$n$ we argue inductively on the order $e = \min_{j \geq 0}
\{\Delta_n^j(f)\,\, \text{elliptic}\}$ of quasi-ellipticity with respect to 
the last variable. 
Suppose, without loss of generality, that $f \in \cQQ_n^{(k)}$ is homogeneous
of weight~$k$. If $e=0$ then we proceed as in the proof of Proposition~\ref{prop:Jaddbasis}
 and subtract $\cQQ_{n-1}$-multiples of $P^{(m)}(z_n -z_j)$ for appropriate $m$ and 
$j \in \{1,\ldots,n-1\}$ and finally $\cQQ_{n-1}$-multiples of $Z(z_n -z_j) - Z(z_n-z_k)$ 
so that the resulting function is elliptic and without poles in the variable~$z_n$, 
hence constant in~$z_n$.
\par
Now consider the case $e>0$. Using the induction hypothesis we can write
$$ \Delta_n(f)(z_1,\ldots,z_n) \= \sum_{s=0}^{e-1} \sum_{j,m} a_{s,j,m} 
Z^{s}(z_n-z_j) P^{(m)}(z_n-z_j)  + b_{s,j} Z^s(z_n-z_j)\,.$$
Now set 
$$ \tilde{f}(z_1,\ldots,z_n) \= \sum_{s=0}^{e-1}  \frac{1}{s+1}  \sum_{j,m}
a_{s,j,m} Z^{s+1}(z_n-z_j) P^{(m)}(z_n-z_j) + b_{s,j} Z^s(z_n-z_j)\,. $$
Then, since 
$$ \Delta_n(Z^{s+1}(z_n)P^{(m)}(z_n)) \= P^{(m)}(z_n) \,\sum_{t=0}^{s} \binom{s+1}{t} Z^t(z_n) $$
we conclude that
\bas 
\Delta_n(f-\tilde{f})(z_1,\ldots,z_n) &\=  \sum_{s=0}^{e-1} \,\frac{1}{s+1}\, 
\Bigl( \sum_{j,m}  a_{s,j,m} 
P^{(m)}(z_n-z_j)\, + b_{s,j} \Bigr)  \\
&\phantom{\=  \sum_{s=0}^{e-1} \frac{1}{s+1} \Bigl(} 
\cdot \,\Bigl(\,
\sum_{t=0}^{s-1} \binom{s+1}{t} Z^t(z_n-z_j) \Bigr)\\
\eas
has order $\leq e-2$ with respect to~$z_n$. Consequently,  $f-\tilde{f}$ has 
order $\leq e-1$ and another application of
the induction hypothesis implies the claim.
\end{proof}
\par
Using this additive basis we can now prove the main result, that
contains Theorem~\ref{thm:introPint} as special case.
\par
\begin{Thm} \label{thm:coeff0QM}
For any permutation $\pi$ the constant term with respect to the domain 
$\Omega_\pi$ of a function in $\cQQ_{n}^{(k)}$ is a quasimodular form
of mixed weight $\leq k$.
\end{Thm}
\par
\begin{proof} By relabeling the variables of~$f \in \cQQ_{n}^{(k)}$ we may 
assume that $\pi$ is the trivial permutation. We may thus write~$f$
as in Proposition~\ref{prop:Naddbasis} and integrate with respect to~$z_n$ 
first. We are thus reduced to showing that the $z_n$-integrals of the additive 
generators $Z^e(z_n-z_j) P^{(m)} (z_n-z_j)$ and $Z^e(z_n-z_j)$ are 
quasimodular forms of mixed weight less or equal to $e+2+m$. For the
generators including a $P$-derivative we use the integration by parts 
$[Z^e(z)]' = e Z^{e-1}(z)P(z)$ and 
$$[Z^e(z)P^{(m-1)}(z)]' = eZ^{e-1}(z)P(z)P^{(m-1)}(z) + Z^e(z)P^{(m)}(z)$$ 
to reduce the order~$e$ until we can apply~\eqref{eq:Pint}. This involves 
rewriting $P(z)P^{(m-1)}(z)$ as a linear combination of $P^{(j)}(z)$. To do this, 
we use again Proposition~\ref{prop:Jaddbasis} and note that the terms $Z_{ijn}$ 
will not occur in this case, since the expression $P(z)P^{(m-1)}(z)$
has a unique pole at $z=0$ modulo $\ZZ + \tau \ZZ$, and hence zero
residue there. The
integrals of the remaining generators are dealt with in the next
proposition.
\end{proof}
\par
In all the steps so far the weight of the quasimodular form had
been preserved. We isolate the next step since this is the reason
for mixed weight.
\par
\begin{Prop} \label{prop:Zpowint}
The constant coefficient $[\zeta^0] Z^e(z)$ is
a quasimodular form of mixed weight less or equal to $e$.
\end{Prop}
\par
\begin{proof} Since $Z(z)$ is an odd function of~$z$, we obtain 
for $\ell$~odd 
\ba  
\int_0^1 Z^\ell(z+i\ve)dz &\= \int_{-1/2}^{1/2} Z^\ell(z+i\ve)dz
\=-\int_{-1/2}^{1/2} Z^\ell(z-i\ve)dz  \\
&\= -\frac12\, {\rm Res}_0 Z^\ell
\ea
and these residues can be read off from the Laurent series
development~\eqref{eq:LaurentLZP} raised to the $\ell$-th power.
\par
On the other hand, we claim that $[\zeta^0] (Z-1/2)^\ell = 0$
for  $\ell$~odd. To see this, we expand the product of the Fourier
expansions~\eqref{eq:FourierLZP} to obtain
\ba \label{eq:Z12ell}
(Z-1/2)^\ell &\= \sum_{k_1, \ldots, k_\ell \geq 1 
\atop \ve_1, \ldots, \ve_n \in \{\pm 1\}} 
(-1)^{\sum \ve_i} \, \prod_{i=1}^\ell \frac{Q_i(k_i,\ve_i)}{1-q^{k_i}}\,,
\ea  
where $Q_i(k_i,\ve_i) =1$ for $\ve=1$ and  $Q_i(k_i,\ve_i) =q^{k_i} $
for $\ve = -1$.  The constant term $[\zeta^0] (Z-1/2)^\ell$ is equal to the 
sum over all $k_i$ and $\ve_i$ with $\sum k_i \ve_i =0$ of the right
hand side of~\eqref{eq:Z12ell}. This set admits and involution
by swapping the signs of all the~$\ve_i$. This involutions changes
the sign of the prefactor $(-1)^{\sum \ve_i}$ since~$\ell$ is odd. 
We claim that the product is unchanged by the involution. The denominator
obviously does not change. The numerator is $q$ raised to the 
power $\sum_{i: \ve_i = +1} k_i$ in one case while it is $q$ raised to 
the power $\sum_{i: \ve_i = -1} k_i$
in the other case. The two sums are equal by the defining condition
of the constant term.
\par
Knowing the constant terms of $(Z-1/2)^\ell$ and $Z^\ell$
for $\ell$ odd, we can solve the (triangular) system of linear equations
to determine the constant coefficients of $Z^\ell$ for $\ell$~even.
\end{proof}
\par
The first few values of these constant coefficients of $Z^\ell$ are
\bas \label{bla}
[\zeta^0] Z\,\, &= \frac12 &  [\zeta^0] Z^2 &= -2G_2 + \frac16 \\
[\zeta^0] Z^3 &= -3G_2 &  [\zeta^0] Z^4 &= 8G_2^2 - \frac13G_4 -2G_2 
- \frac1{30}\\
[\zeta^0] Z^5 &= 20G_2^2 - \frac56 G_4 &  [\zeta^0] Z^6 &=
 \frac{-1}{60}G_6 + 4G_4G_2 -40G_2^3 + 20G_2^2  - \frac{5}{6}G_4 + G_2
+ \frac1{42}
\eas
%
%
%
%
%

\subsection{An example of mixed weight} 
\label{sec:Ex1111Pcomp}

Here we illustrate that the proof of Theorem~\ref{thm:coeff0QM} 
provides an effective algorithm by computing (with respect to the 
standard order~\eqref{eq:order}) the expression
\begin{flalign}
&\phantom{\=} [\zeta^0]P(z_1-z_2)^2P(z_1-z_4)P(z_2-z_3)P(z_3-z_4)^2 \label{eq:H14domain1}
\\
&=\, 4q^2 + 224q^3 + 3088q^4 + 21888q^5 + 105136q^6 + 388288q^7 + 1197280q^8
+O(q^9) \nonumber \\
&= -256G_2^6 + \frac{640}{3}G_4G_2^4 +  \frac{112}{9}G_6G_2^3 
-\frac{400}{9}G_4^2 G_2^2   - \frac{140}{9}G_6G_4G_2 + \frac{2000}{81}G_4^3  
+ \frac{49}{108}G_6^2 \nonumber\\
&\phantom{\=}  \+\ \left(-\frac{256}{3}G_4G_2^3-\frac{16}{5}G_6G_2^2+\frac{320}{21}G_4^2G_2
+ \frac{28}{9}G_6G_4\right)\,, \nonumber
\end{flalign}
which is a quasimodular form of mixed weight 10 and 12.
\par
To prove this formula we first treat the terms depending on~$z_4$ 
and write 
\bas P(z_4-z_1)^2P(z_4-z_3)&=\frac{1}{6}P(z_1-z_3)P''(z_4-z_1) 
+ \frac{1}{2}P'(z_3-z_1) P'(z_4-z_1)\\
& + \Bigl(\frac12 P''(z_1-z_3)  + 4G_2 P(z_1-z_3) \Bigr)P(z_4-z_1) \\ 
& + P(z_1-z_3)^2P(z_4-z_3)\\
& + \Bigl(4G_2 P'(z_3-z_1)- \frac16 P'''(z_1 - z_3)\Bigr) \,Z_{134} \\
& + \frac{1}{3}G_2 P''(z_1-z_3) - 2G_2P(z_1-z_3)^2 - 8G_2^2P(z_1-z_3) \\
& + 8G_2^3 - \frac{10}{3}G_4G_2
\eas
in the additive basis given in Proposition~\ref{prop:Naddbasis}.
This allows to integrate with respect to $z_4$ and we obtain
\bas \label{eq}
[\zeta_4^0]P(z_4-z_1)^2P(z_4-z_3)&\= \bigl(4G_2P'(z_3-z_1) - \frac{1}{6}P'''(z_3-z_1)\bigr)
Z(z_1-z_3) \\
& \phantom{\=}\+ \frac{4}{3}G_2 P''(z_3-z_1) + 8G_2^3 -\frac{10}{3}G_4G_2\,. \eas
The integral $[\zeta_2^0]P(z_1-z_2)P(z_2-z_3)^2$ is the same as the previous one, 
replacing $z_4$ by $z_2$. The product of these two integrals contains 
the term $$T = P'''(z_1-z_3)P'(z_1-z_3)Z(z_1-z_3)^2$$ 
and several terms that can be treated similarly as~$T$, or that can be
computed by integration by parts as in the proof of Theorem~\ref{thm:coeff0QM} 
and that finally yield a contribution that is pure of weight~12.
To integrate~$T$ with respect to $z_1$ we decompose $T/Z(z_1-z_3)^2$ in the additive basis
\[P'''(z)P'(z)=-12G_6P(z) + 24G_6G_2 -\frac{800}{7}G_4^2 - 8P''(z)G_4 + \frac{1}{105}P^{(6)}(z)\]
given by Proposition~\ref{prop:Jaddbasis}. The $Z^2$-multiples of the terms containing~$P$ 
or its derivatives can be computed using integration by parts and contribute purely to 
weight~$12$. 
Finally, the constant (in $z$) multiples of~$Z^2$ are integrated using 
Proposition~\ref{prop:Zpowint} and cause the contribution of mixed weight.

\section{Quasimodularity of graph sums} 
\label{sec:GraphSumisQM}
Motivated by Theorem~\ref{thm:bracketgraphsum} and~\ref{thm:SSZ} we
consider here the graph sums
\be \label{eq:SGammam} {S}(\Gamma,\mm)=\sum_{G \in \Gamma}{S}(G,\mm)\ee
over all orientations $G$ of $\Gamma$, where 
\be \label{eq:SGm} {S}(G,\mm)=\sum_{h\in\widetilde\NN^{E(G)}, w\in \Z_+^{E(G)}}\prod_{i\in E(G)}w_i^{m_i+1}q^{h_iw_i}\prod_{v\in V(G)}\delta(v)\,. \ee
Here $\widetilde\NN^{E(G)}$ is the height space introduced in~\eqref{eq:heightspace}
and the Dirac-symbol  $\delta(v)$ was introduced in~\eqref{eq:NotDelta}.
The goal of this section is to show the quasimodularity of these graph sums.
\par
\begin{Thm} \label{thm:QMforS}
If $\mm = (m_1,\ldots,m_{|E(\Gamma)|})$ is a tuple of even integers, 
then the graph sums 
$S(\Gamma, \mm)$ are quasimodular forms of mixed 
weight at most $k(\mm):= \sum_{i} (m_i +2)$.
\end{Thm}
\par
The proof consists of splitting the sum into the contribution from loops
and the rest, and then to apply the coefficient extraction results from 
Section~\ref{sec:QMF}.
\par
The combination of this result with the graph sum theorem and the polynomiality
of generalized double Hurwitz numbers with completed cycles immediately 
gives the first quasimodularity result we are aiming for.
\par
\begin{Cor} \label{cor:noname}
For any ramification profile~$\Hmu$ the counting 
function $N^\circ(\Hmu)$ for connected torus covers of profile~$\Hmu$ 
is a quasimodular form of mixed weight less or equal to
${\rm wt}(\Hmu) = |\Hmu| + \ell(\Hmu)$.
\end{Cor}
\par
\begin{proof} As remarked before Theorem~\ref{thm:bracketgraphsum}, 
thanks to Theorem~\ref{thm:KO} we only need to show the quasimodularity of 
$q$-brackets of a product of $p_{\ell_i}$, $i=1\dots n$, i.e.\ of the expressions
appearing in~\eqref{eq:pasGraphsum} that are sums over graphs with $n$ labeled vertices. By Theorem~\ref{thm:SSZ} the expressions 
$\ol{A}'_v = \ol{A}'(\bfw_v^-,\bfw_v^+,(\ell_{\# v}))$ appearing 
on the right hand side of~\eqref{eq:pasGraphsum} are either zero or 
they are polynomials of even degree equal to $\ell_{\# v} + 1 -\ell(\bfw_v^-)
-\ell(\bfw_v^+)$. Consequently, the product over all vertices 
of these polynomials $\ol{A}_v'$ has degree ${\rm wt}(\Hmu) - 2|E(\Gamma)|$.
We can now apply Theorem~\ref{thm:QMforS}.
\end{proof}
\par
\smallskip
The rest of this section consists of the proof of Theorem~\ref{thm:QMforS}.
As a first technical step we show that $S(\Gamma,\mm)$, which a priori
depends on the heights $y_i$ involved in the definition of the height space, 
is in fact independent of this choice and can be computed using the limit
value $y_i=0$ for all~$i$.
\par
\begin{Lemma}\label{lem:heightspace} Replacing the height space  
$\widetilde\NN^{E(G)}$  by 
\[\NN_{E(G)}=\{(n_1,\dots, n_{|E(G)|}), \; n_i\in\{0\}\cup\Z_+\mbox{ if }v_+(i)>v_-(i), \; n_i\in\Z_+\mbox{ otherwise}\}\]
does not change the total sum $S(G,\mm)$. 
\end{Lemma}
\par
\begin{proof}We apply the linear change of variables
$h'_e=h_e-\delta_e$ with $\delta_e=\varepsilon_{i^+(e)}-\varepsilon_{i^-(e)}$, 
that maps $\widetilde\NN^{E(G)}$ onto $\NN_{E(G)}$. Each summand for fixed 
$(w_1,\dots,  w_{|E(G)|})$ is then multiplied by
\bes q^{-\sum_i\delta_iw_i}\prod_{v\in V(G)}\delta(v) \= 
q^{-\frac{1}{2}\sum_v \varepsilon_v\left(\sum_{i\in e_-(v)} w_i - \sum_{i\in e_+(v)} w_i\right)}\prod_{v\in V(G)}\delta(v) \=1\,.
\ees
This implies the claim.
\end{proof}

\subsection{The reduced graph}

We call an edge of a graph $\Gamma$ (or an oriented graph $G$) {\em neutral},
if it is a loop (starting and ending at the same vertex). The other edges 
are called {\em essential}. We denote by $E(G)=E_0(G)\cup E_1(G)$ the splitting 
of the set of edges of $G$ into neutral and essential edges. Furthermore, let
$G_1$ denote the oriented graph obtained from $G$ by removing the neutral edges.
Similarly, $\Gamma_1$ is the graph obtained from $\Gamma$ by removing the neutral 
edges. A graph with no neutral edges we be called {\em reduced}.
\par
\begin{Lemma}\label{lem:SinS0S1} The graph sums $S(\Gamma,\mm)$ can be factored as
\[ S(\Gamma,\mm) \= S_0(\Gamma,\mm_0)S(\Gamma_1,\mm_1)\]
where \[S_0(\Gamma,\mm_0)\ =\sum_{h\in\NN_{E_0(G)}, w\in \Z_+^{E_0(G)}}\prod_{i\in E_0(G)}w_i^{m_i+1}q^{h_iw_i},\] for any orientation $G$ of $\Gamma$, and where 
$\mm_0=(m_i)_{i\in E_0(G)}$, $\mm_1=(m_i)_{i\in E_1(G)}$.
\end{Lemma}
\par
\begin{proof} This follows from decomposing for every $G \in \Gamma$ the
summation in the definition of $S(G,\mm)$ into the 
sum over $h\in\NN_{E_0(G)}$ and $w\in \Z_+^{E_0(G)}$
and the remaining variables. Since the $w_i$ for $i\in E_0(G)$  drop out of 
the functions~$\delta$, each $S(G,\mm)$ splits off a factor
\[S_0(G, \mm_0) \= 
\sum_{h\in\NN_{E_0(G)}, w\in \Z_+^{E_0(G)}}\prod_{i\in E_0(G)}w_i^{m_i+1}q^{h_iw_i}\,\]
which does not depend of the orientation $G$ of $\Gamma$.
\end{proof}
\par
\begin{Lemma}\label{lem:S0}
If $\mm = (m_1,\ldots,m_{|E(G)|})$ is a tuple of even integers, the neutral 
contribution 
$S_0(\Gamma,\mm_0)$ is a quasimodular form of mixed weight $k(\mm_0)$.
\end{Lemma}
\begin{proof} The neutral contribution $S_0(\Gamma,\mm_0)$ is the product of $S_{m_i}$
over all neutral edges, where $S_m=\sum_{w,h=1}^\infty w^{m+1}q^{hw}$
is the $q$-expansion of the Eisenstein series~$G_{m+2}$ without its constant term
for $m$~even.
\end{proof}

The problem is now reduced to computing $S(\Gamma,\mm)$ for all reduced graphs $\Gamma$.

\subsection{Contour integrals}

The idea of the proof is to write  the delta functions appearing  in the graph sums
as contour integral of some suitably chosen powers of $\zeta=e^{2\pi i z}$ since 
by the residue theorem
\be \label{eq:resthm}
\int_\gamma \zeta^w dz=\int_{\gamma'}\zeta^{w-1}\frac{d\zeta}{2\pi i}=\delta_{0,w}
\ee
where $\gamma(t)=t+iy$ with $0\leq t\leq 1$ is our standard integration path
of height $\Im y \in (0,\Im \tau)$  and $\gamma'=\exp(2\pi i\gamma)$. Recall that
we denote the derivative with respect to $u = 2\pi i z$ by primes. By~\eqref{eq:FourierLZP} 
the Fourier developments of the derivatives of~$P$ are
\[P^{(m)}(z)= \frac{1}{(2\pi i)^m} \frac{\partial^m}{\partial z^m} P(z)=
\sum_{w=1}^{\infty}\sum_{h=1}^{\infty}w^{m+1} q^{hw}(\zeta^w+(-1)^m\zeta^{-w})+ 
\sum_{w=1}^{\infty}w^{m+1} \zeta^w.\] 
Next, we expand the path integrals of products of derivatives of~$P$. In the following 
lemmas $\gamma_j$ will be the standard path at height~$y_j$ and $y_j<y_{j+1}$.
\par
\begin{Lemma}\label{le:int2}
Let $I\subset\{2, \dots, n\}$ and assume that all the $m_i$ are even. Then on the
domain defined by $|\zeta_i|>|\zeta_1|>|q\zeta_i|$ for all $i\in I$ we have the expansion
\bas \oint_{\gamma_1}\zeta_1^{k_1}\prod_{i\in I}P^{(m_i)}(z_1-z_i)dz_1 &\=\sum_{w_i=1}^{\infty}
\prod_{i\in I} w_i^{m_i+1}\cdot \sum_{I=J\cup K} 
\delta\bigl(k_1+\sum_{j\in J}w_j-\sum_{k\in K} w_k\bigr)\\
&\phantom{\=} \cdot\,\prod_{j\in J} \Bigl( \sum_{h_j=0}^\infty q^{h_jw_j}\Bigr) \cdot \prod_{k \in K} 
\Bigl( \sum_{h_k=1}^\infty q^{h_kw_k}\Bigr) \frac{\prod_{K}\zeta_k^{w_k}}{\prod_J \zeta_j^{w_j}}
\eas
for every $k_1 \in \ZZ$.
\end{Lemma}
\par
\begin{proof} Since $P$ is even by definition, so are its derivatives of even order.
The lemma follows by expanding the integrand according to the preceding formula
for~$P^{(m)}$ and by~\eqref{eq:resthm} only the term with 
$k_1+\sum_{j\in J}w_j-\sum_{k\in K} w_k = 0$ survives.
%
\end{proof}
\par
We can use this observation to write the graph sum in terms of derivatives of~$P$.
For this purpose, we introduce the following shorthand notations. Let
$\zz = (z_1,\ldots,z_n)$. For a reduced graph~$\Gamma_1$  with $n$ vertices 
and $N_1$ edges, for $(y_1, \dots, y_n)$ fixed as before, 
and $\mm_1 =(m_1, \dots, m_{N_1})$ an $N_1$-tuple of even integers, we define
\be \label{eq:defPGamma}
P_{\Gamma_1, \mm_1}(\zz)\=\prod_{i\in E(\Gamma_1)}
P^{(m_i)}(z_{v_1(i)}-z_{v_2(i)})\,,
\ee
where $v_1(i)$ and $v_2(i)$ are the two ends of the edge $i$. Note that by our 
parity assumption the function $P^{(m_i)}$ is even and so the expression is independent
of the labeling of the ends of edges.
%
\par
\begin{Prop} \label{prop:S1} For a tuple of even integers $\mm_1$ we can express 
the graph sum as
\ba \label{eq:SasContInt}
{S}(\Gamma_1,\mm_1)\= [\zeta_n^0, \dots, \zeta_1^0]\,
P_{\Gamma_1, \mm_1}(\zz)\,,
\ea
where the coefficient extraction is for the expansion on the domain 
 $|q\zeta_{i+1}|<|\zeta_i|<|\zeta_{i+1}|<1$ for all~$i$.
\end{Prop}
\par
\begin{proof} By the general observation in Section~\ref{sec:intEF}, coefficient 
extraction is the same operation as the computation of the path integrals along the standard
paths. We compute the right hand side inductively and show that the final expressing
coincides with the graph sum. We denote by $E_1$ the set of the (labels of the) edges 
adjacent to the vertex~$v_1$. For two subsets of edge labels we define the shorthand
notation
$$ \delta(J,K)\= \delta\left(\sum_{j\in J} w_j-\sum_{k \in K} w_k\right)\,.$$
Using that~$\Gamma_1$ has no loops and the parity of~$P$
we rewrite the integrand as
\[P_{\Gamma_1, \mm_1}(\zz) \= \prod_{j\in E_1} P^{(m_j)}(z_{1}-z_{v_{other}(j)})
 \cdot \prod_{i\in E\setminus E_1} P^{(m_j)}(z_{v_1(j)}-z_{v_2(j)})\] 
where $v_{\textrm{other}}(j)=v_{o}(j)$ is the second extremity of the edge~$j$ 
for each edge label~$j$ adjacent to~$v_1$. We now apply Lemma~\ref{le:int2} with
$k_1=0$ and obtain
\bas 
\oint_{\gamma_1} P_{\Gamma_1, \mm_1}(\zz) \=   \prod_{i\in E\setminus E_1} P^{(m_i)}(z_{v_1(i)}-z_{v_2(i)})\cdot\mathop{ \sum_{w_i=1}^\infty}_{i\in E_1}\left(\prod_{i \in E_1} w_i^{m_i+1}\right)\\
\cdot \sum_{E_1=J_1\sqcup K_1}\delta(J_1,K_1) \prod_{j\in J_1}\Bigl(\sum_{h_j=1}^\infty 
q^{h_jw_j}\Bigr)\prod_{k\in K_1}\Bigl(\sum_{h_k=0}^\infty q^{h_kw_k}\Bigr)
\frac{\prod_{K_1}\zeta_{v_o(k)}^{w_k}}{\prod_{J_1} \zeta_{v_o(j)}^{w_j}}
\eas
\par
Consider the new graph $\Gamma^{(1)}$ obtained by removing the edges in~$E_1$ 
from~$\Gamma_1$. We denote by~$E_2$ the set of edges adjacent to~$v_2$ in this 
new graph, 
and for each partition $J_1\sqcup K_1$ of $E_1$ we iterate the process by 
integrating \[\prod_{i\in E\setminus E_1} P^{(m_i)}(z_{v_1(i)}-z_{v_-(i)})
\cdot \frac{\prod_{K_1}\zeta_{v_o(k)}^{w_k}}{\prod_{J_1} \zeta_{v_o(j)}^{w_j}}\] along $\gamma_2$, 
using again Lemma~\ref{le:int2}. We then consider the graph $\Gamma^{(2)}$ obtained
by removing the edges in $E_2$ from $\Gamma^{(1)}$ and so forth. At the end of 
this procedure, we obtain
\bas
& \phantom{\= \cdot} [\zeta_n^0, \dots, \zeta_1^0]\, P_{\Gamma_1, \mm_1}(\zz) \\
& \= \sum_{h_i,w_i}\prod_i w_i^{m_i+1}\cdot\sum_{E_1=J_1\sqcup K_1}
\delta(J_1,K_1) \cdot \prod_{j\in J_1} \sum_{h_j=1}^\infty q^{h_jw_j}\cdot
\prod_{k\in K_1}\sum_{h_k=0}^\infty q^{h_kw_k}\\
&\phantom{\=} \cdot \sum_{E_2=J_2\sqcup K_2}\delta(J_2\sqcup K_1^2, K_2\sqcup J_1^2)\cdot
 \prod_{j\in J_2}\sum_{h_j=1}^\infty q^{h_jw_j} \cdot \prod_{k\in K_2}
\sum_{h_k=0}^\infty q^{h_kw_k}\\
& \phantom{\=} \dots \!\!\!\!\!\sum_{E_n=J_n\sqcup K_n} \!\!\!\! \!
\delta(J_n\sqcup K_{n-1}^{n}\!\sqcup \dots,\, K_n\sqcup J_{n-1}^n \!\sqcup\dots)
\cdot \prod_{j\in J_n} \sum_{h_j=1}^\infty q^{h_jw_j}\cdot \prod_{k\in K_n}
\sum_{h_k=0}^\infty q^{h_kw_k}\!\!\!\!\!\!\!\!\!\!\!\! ,\phantom{q^{h_kw_k}\;} \eas
where $K_i^j$ denotes the subset of $K_i$ formed by the edges adjacent to $v_j$. 
We recognize the definition of $ S(\Gamma_1, m)$, since partitioning $E$ into $E_1, E_2, \dots E_n$ as above gives an orientation $G_1$ of $\Gamma_1$, and non 
realizable orientations have coefficient~$0$. 
Note that the last integration with respect to the variable $z_n$ has no effect
since the powers of~$\zeta_n$ cancel out thanks to the delta functions.
\end{proof}

\begin{proof}[Proof of Theorem~\ref{thm:QMforS}]
In view of the factorization of the quantity we are interested in the 
loop contribution and the reduced contribution in Lemma~\ref{lem:SinS0S1}.
The loop contribution is quasimodular by Lemma~\ref{lem:S0}, so 
it suffices to show that the right hand side of~\eqref{eq:SasContInt} is 
a quasimodular form of mixed weight $\leq k(\mm_1)$. This follows from
Theorem~\ref{thm:coeff0QM}, since $P_{\Gamma_1,\mm_1} \in 
\mathcal{J}_n^{(k(\mm_1))} \subset \cQQ_n^{(k(\mm_1))}$ by 
Proposition~\ref{prop:Qring}.
\end{proof}

\begin{proof}[Proof of Theorem~\ref{thm:bo}]
Let $f$ be a shifted symmetric function of weight $k$. By Theorem~\ref{thm:KO}, $f$ is a linear combination of products of $P_l$, each product being of weight smaller or equal to $k$. Let $P_{\ell_1}\dots P_{\ell_n}$ be such a product. We claim that $\bq{P_{\ell_1}\dots P_{\ell_n}}$ is quasimodular of weight smaller than $k$. By Theorem~\ref{thm:bracketgraphsum}, such a term decomposes as a graph sum \eqref{eq:pasGraphsum}, where the completed Hurwitz numbers $\overline A'$ that appear in the graph sum are some even polynomials in the $w_i$ (Theorem~\ref{thm:SSZ}). The product $\prod_{v\in G} \overline A'(\bfw_v^-, \bfw_v^+,(\ell_{\# v}))$ is an even polynomial of degree $\leq k-2|E(G)|$. Considering each monomial of degree $\mm=(m_1,\dots, m_{|E(G)|})$, we get the graph sum $S(G, \mm)$ in \eqref{eq:SGm}, which is quasimodular of weight $|\mm|+2|E(G)|$ by Theorem~\ref{thm:QMforS}.  
By linearity we obtain the quasimodularity of weight $\leq k$ of each 
$\bG{P_{\ell_1}\dots P_{\ell_n}}$, hence the quasimodularity each $\bq{P_{\ell_1}\dots P_{\ell_n}}$, 
and finally the quasimodularity of the $\sbq{f}$.
\end{proof}

\section{Siegel-Veech constants} 
\label{sec:SVgeneral}

The study of area Siegel-Veech constants for flat surfaces, briefly recalled
in Section~\ref{sec:reltoSV}, lead in \cite{CMZ} to counting problems for graphs
sums with a Siegel-Veech weight that we introduce in the sequel. 
We show that the Siegel-Veech weighted graph sums 
admit a decomposition into graph sums of triple Hurwitz numbers similar
to the unweighted case (Proposition~\ref{prop:N'exp}). The main result
(Theorem~\ref{thm:SSVQM}) in this section is that these graph sums fit into the 
scope of our quasimodularity machinery of Section~\ref{sec:QMFProp}. We 
use this to give another proof of the quasimodularity of the Siegel-Veech 
weighted generating series observed in \cite{CMZ}. 
\par
Let $\lambda = (\lambda_1 \geq \lambda_2 \geq \cdots \geq \lambda_k)$ be 
a partition. For $p \in \ZZ$ we define the {\em $p$-th Siegel-Veech weight} 
of~$\lambda$ to be
\be \label{eq:def_pSV}
S_p(\lambda) \= \sum_{j=1}^{k} \lambda_j^p\,.
\ee
Let $\alpha^{(j)}$ denote the first element of the Hurwitz tuple~$h_j$, defined in \eqref{eq:HT}.
We define
\begin{equation} \label{eq:cpsimple}
c^*_{p}(d,\Hmu) \=   \sum_{j=1}^{|\Cov^*_d(\Hmu)|} S_p(\alpha^{(j)})\,,
\quad \text{and} \quad c_p^*(\Hmu) = \sum_{d \geq 0} c_{p}^*(d,\Hmu) q^d\,,
\end{equation}
where for $* \in \{', 0, \emptyset\}$ we packaged the Siegel-Veech weighted 
Hurwitz numbers into a  generating series. These series admit the
following graph sum decomposition.
\par
\begin{Prop} \label{prop:svgraphsum}
The generating series $c_{p}'(\Hmu)$ can be expressed in terms of
graph sums of triple Hurwitz numbers  as
\bes 
c_{p}'(\Hmu) \=  \sum_{\Gamma} \frac1{|\Aut(\Gamma)|}c_p'(\Hmu,\Gamma)\, ,
\ees
where 
\bes
c_p'(\Hmu,\Gamma)=\sum_{G \in \Gamma}
\!\sum_{h\in \widetilde{\NN}^{E(G)}, \atop w\in \Z_+^{E(G)}}
\Bigl( \sum_{e \in E(G)} h_e w_e^p \Bigr)
\prod_{e\in E(G)} w_e 
q^{h_e w_e} \!\!\prod_{v\in V(G)}A'(\bfw_v^-,\bfw_v^+,\mu_v)\,\, \delta(v)\,.
\ees
\end{Prop}
\par
As a corollary of this and Theorem~\ref{thm:SSVQM} below, we obtain an
independent proof of the following quasimodularity result
(see also \cite[Theorem 6.4]{CMZ}) without relying on the 
combinatorial machinery of $q$-brackets involving $\widetilde{T_p}$
(cf.\ \cite[Section~14 and~15]{CMZ}).
\par
\begin{Cor} \label{Cor:SVmain}
For any ramification profile $\Hmu$ and any odd $p \geq -1$
the generating series $c_p'(\Hmu)$ for counting covers without unramified
components and with $p$-Siegel-Veech weight as well as the generating 
series $c_p^\circ(\Hmu)$ for connected counting with $p$-Siegel-Veech weight
are quasi-modular forms of mixed weight $ \leq {\rm wt}(\Hmu) +p+1$.
\end{Cor}
\par

\subsection{Relation to area Siegel-Veech constants} \label{sec:reltoSV}

The generating functions $c_{-1}'(\Hmu)$ admit a nice geometric interpretation in 
terms of Siegel-Veech constants, that are responsible for counting closed geodesics 
on flat surfaces. For a flat surface~$X$ we define the counting function
\[N_{\textrm{area}}(T,L) \= \sum_{Z\subset X\;\textrm{ cylinder}, \atop w(Z)\geq L}
\frac{\mbox{Area}(Z)}{\mbox{Area}(X)}.\]
counting the cylinders filled by closed geodesics on $X$, weighted by their area.
This function is well known to have a quadratic asymptotic (see e.g. \cite{eskinmasur}),
and the number 
\[c_{\textrm{area}}(X) \ =\lim\limits_{L\to\infty}\frac{N_{\textrm{area}}(T,L)}{\pi L^2}\] 
is called the (area) Siegel-Veech constant associated to~$X$. This constants 
are interesting both for generic flat surfaces of a given singularity type and
for torus covers. It is shown in \cite[Theorem 4]{ekz} and \cite[Theorem 3.1]{CMZ} 
that the Siegel-Veech constant for a torus cover of degree $d$ and ramification $\Hmu$ 
is  
\[c_{\textrm{area}}(d,\Hmu) \= \frac{3}{\pi^2}\frac{c_{-1}^0(d,\Hmu)}{N_d^0(\Hmu)}\,.\]
Thus the series $c_{-1}^*(\Hmu)$ can be interpreted as generating functions for 
the ``Siegel-Veech numerators'' of Hurwitz spaces. Knowing them to be quasimodular
forms, and thus knowing the asymptotic behaviour of both $c_{-1}^0(d,\Hmu)$ 
and $N_d^0(\Hmu)$ as $d \to \infty$ allows to compute the area Siegel-Veech constant
of a generic surface with a given singularity type.

\subsection{Siegel-Veech weighted graph sums}

In view of the Proposition~\ref{prop:svgraphsum} we consider now some variants 
of the graphs sums discussed Section~\ref{sec:GraphSumisQM}, and show 
that they are quasimodular too. We define again the 
Siegel-Veech weighted graph sums
\be \label{eq:SSVGammam} {S}^{SV}(\Gamma,\mm)=\sum_{G \in \Gamma}{S}^{SV}(G,\mm)\ee
over all orientations $G$ of $\Gamma$, where now 
\be \label{eq:SSVGm} {S}^{SV}(G,\mm)=\sum_{h\in\widetilde\NN^{E(G)}, w\in \Z_+^{E(G)}}
\Bigl(\sum_{i \in E(G)} \frac{h_i}{w_i}\Bigr)
\prod_{i\in E(G)}w_i^{m_i+1}q^{h_iw_i}\prod_{v\in V(G)}\delta(v)\ee
\begin{Thm}\label{thm:SSVQM}
If $\mm=(m_1, \dots , m_{|E(\Gamma)|})$ is a tuple of even integers, then 
the graph sums $S^{SV}(\Gamma, \mm)$ are quasimodular forms of weight 
at most $k(\mm)=\sum_i(m_i+2)$.
\end{Thm}
\par
As in the case of ordinary counting, we can split off the loops and
reduce to a simplified height space. The main new ingredient is that
nearly-elliptic function~$L$ and $q$-derivatives of~$P$ have the 
right Fourier expansion whose constant coefficients capture the new 
graph sums and still fit in the scope of quasimodularity results of
Section~\ref{sec:intEF}.
\par
\smallskip
We decompose  $S^{SV}(\Gamma, \mm)$ according to the edge~$i_0$
that contributes $h_i/w_i$ in the prefactor. That is, we decompose 
\[{S}^{SV}(\Gamma,\mm) \=\sum_{i_0}{S}^{SV}_{i_0}(\Gamma,\mm), \quad
\text{where} \quad {S}^{SV}_{i_0}(\Gamma,\mm)=\sum_{G\in\Gamma}
{S}^{SV}_{i_0}(G,\mm)\] and where 
\[{S}^{SV}_{i_0}(G,\mm) \=\sum_{h\in\widetilde\NN^{E(G)}, w\in \Z_+^{E(G)}} 
\frac{h_{i_0}}{w_{i_0}}\prod_{i\in E(G)}w_i^{m_i+1}q^{h_iw_i}\prod_{v\in V(G)}\delta(v)\,.\]
\par
Next, we replace the height space $\widetilde\NN^E(G)$ by $\NN_{E(G)}$ using 
an analog of Lemma~\ref{lem:heightspace}.
\par
\begin{Lemma}\label{lem:SVheightspace}
Replacing the height space $\widetilde\NN^{E(G)}$ by $\NN_{E(G)}$ in 
each term $S^{SV}(G, \mm)$ does not change the total sum $S^{SV}(\Gamma, \mm)$.
\end{Lemma}
\par
\begin{proof}
Let $-G \in \Gamma$ be the graph with the reversed orientation 
compared to $G\in\Gamma$. We indicate by an additional index the
space over which the summation $h$ is taken in $S^{SV}(G,\mm)$ We 
will show that 
\[S^{SV}(G, \mm)_{\widetilde\NN^{E(G)}}+S^{SV}(-G, \mm)_{\widetilde\NN^{E(-G)}}
\= S^{SV}(G,\mm)_{\NN_{E(G)}}+S^{SV}(-G,\mm)_{\NN_{E(-G)}}\,.\]
On the first term of the left-hand side we apply the change of variables 
$h'_e=h_e-\delta_e$ with $\delta_e=\varepsilon_{i^+(e)}-\varepsilon_{i^-(e)}$ for 
the orientation $G$, that maps $\widetilde\NN^{E(G)}$ to $\NN_{E(G)}$ as in 
Lemma~\ref{lem:heightspace}. On the second term we apply the change of 
variables $h'_e=h_e+\delta_e$ with $\delta_e$ as before, i.e. associated to 
the orientation $G$. It maps $\widetilde\NN^{E(-G)}$ to $\NN_{E(-G)}$.
As in Lemma~\ref{lem:heightspace}, this change of variable does not affect 
the term $q^{h\cdot w}$ (thanks to the delta functions). The statement 
is then obvious since the terms in $\delta_e/w_e$ cancel out. 
\end{proof}
\par
We reduce the problem to the reduced graph using an analog 
of Lemma~\ref{lem:SinS0S1}.
\par
\begin{Lemma}\label{lem:SSVinS0S1}
The Siegel-Veech weighted graph sums factor as
\[ S^{SV}_{i_0}(\Gamma,\mm)\=\begin{cases}S^{SV}_{0,i_0}(\Gamma,\mm_0)\,
S(\Gamma_1,\mm_1)& \text{if $i_0$ is neutral, }\\
S_{0}(\Gamma,\mm_0)\,S^{SV}_{i_0}(\Gamma_1,\mm_1)& \text{otherwise,}\end{cases}\]
where $\Gamma_1$ is the reduced graph underlying $\Gamma$ and where
\[S^{SV}_{0,i_0}(\Gamma,\mm_0)\=\sum_{h\in\NN_{E_0(G)}, w\in \Z_+^{E_0(G)}}\frac{h_{i_0}}{w_{i_0}}\prod_{i\in E_0(G)}w_i^{m_i+1}q^{h_iw_i}\] for any orientation $G$ of $\Gamma$.
\end{Lemma}
\par
The contribution of the loops is easily dealt with.
\par
\begin{Lemma}\label{lem:S0SV}
For $\mm$ a tuple of even integers $S^{SV}_{0,i_0}(\Gamma,\mm_0)$ is a 
quasimodular form of mixed weight $k(\mm_0)$.
\end{Lemma}
\begin{proof}
Clearly \[S^{SV}_{0,i_0}(\Gamma,\mm_0) \=\tilde S_{m_{i_0}}\prod_{i\neq i_0}S_{m_i}\,\] 
where $S_m$ is defined in Lemma~\ref{lem:S0} and where
\[\tilde S_{m_{i_0}}=\sum_{w,h=1}^\infty hw^{m_{i_0}}q^{hw} \=
\begin{cases} D_qS_{m_{i_0}-2} & \text{ if $m_{i_0}\geq 2$} \\
E_2+\frac{1}{24} & \text{ if $m_{i_0}=0$\,.}\end{cases}\] 
\end{proof}
\par
\begin{proof}[Proof of Theorem~\ref{thm:SSVQM}]
It remains to show that if $i_0$ is not neutral, then the graph 
sum $S^{SV}_{i_0}(\Gamma_1,\mm_1)$ is a quasimodular form. We define
$$ P_{\Gamma_1, \mm_1}^{SV,i_0}(\zz) \= 
D_q P^{(m_{i_0}-2)}(z_{v_1(i_0)}-z_{v_2(i_0)})\prod_{i\in E(\Gamma_1)\setminus\{i_0\}}
P^{(m_i)}(z_{v_1(i)}-z_{v_2(i)}) $$
if $m_{i_0}\geq 2$ and in the remaining case $m_{i_0} = 0$ we let
$$ P_{\Gamma_1, \mm_1}^{SV,i_0}(\zz) \=
L(z_{v_1(i_0)}-z_{v_2(i_0)})\prod_{i\in E(\Gamma_1)\setminus\{i_0\}}
P^{(m_i)}(z_{v_1(i)}-z_{v_2(i)})\,.$$
These definitions are designed such that, with the same proof
as in Proposition~\ref{prop:S1} we obtain
\bes S^{SV}_{i_0}(\Gamma_1,\mm_1)\= [\zeta_n^0\dots \zeta_1^0]
P_{\Gamma_1, \mm_1}^{SV,i_0}(\zz) \,.
\ees
By Proposition \ref{prop:Qring}, the function $L$ belongs to 
$\cQQ_n^{(0)}\oplus\cQQ_n^{(2)}$, and the function $D_qP^{(m_{i_0})}$ belongs to $\cQQ_n^{(m_{i_0}+2)}$. So in any case $P_{\Gamma_1, \mm_1}^{SV,i_0}$ belongs to 
$\cQQ_n^{(k(\mm_1)-2)}\oplus\cQQ_n^{(k(\mm_1))}$. Its constant term hence
is quasimodular of mixed weight $\leq  k(\mm_1)$ by Theorem~\ref{thm:coeff0QM}.
\end{proof}

\subsection{Proof of main results}

\begin{proof}[Proof of Proposition~\ref{prop:svgraphsum}]
As in the proof of Proposition~\ref{prop:N'exp} we rely on 
Proposition~\ref{prop:corrTCGraph}. We only need to justify
that counting with Siegel-Veech weight produces an extra
factor $h_ew_e^p$ for each edge. This is exactly the weight to put on each cylinder that corresponds to the weight $S_p(\lambda)$ for a Hurwitz tuple. This correspondence is obtained using standard Siegel-Veech transform as is the proof of Theorem 3.1 of \cite{CMZ}.
\end{proof}
\par
\begin{proof}[Proof of Corollary~\ref{Cor:SVmain}]
First, we claim that
\be \label{eq:cprimep}
c'_p(\Hmu) \= \bq{ T_p \, f_{\mu_1} \cdots f_{\mu_n}} - \bq{T_p} 
\bq{f_{\mu_1} \cdots f_{\mu_n}}
\ee
as difference of $q$-brackets, where $T_p$ is a function on partitions that
we introduce now. With the definition 
\footnote{This function arise more naturally as the {\em hook length moment}
$$T_p(\lambda) =  \sum_{m=1}^\infty m^{p-1} N_m(\lambda),$$ 
where $N_m(\lambda)$ is the number of cells in the Young diagram of $\lambda$
of hook length~$m$. That $T_p$ is a shifted symmetric function if and only 
if $p$ is odd and positive was a main theme in \cite{CMZ}, but all this 
is not relevant here. That this definition coincides with the definition 
of $T_p$ given here is proven in \cite{CMZ}, Corollary 13.2.}
$$ T_p(\lambda) \= \sum_{\tau\in \Part(d)} z_{\tau}  S_p(\tau) \chi^\lambda(\tau)^2$$
Proposition~6.3 of~\cite{CMZ} implies that
$$c_p(\Hmu) \= \sum_{\lambda \in \P}  (T_p \, f_{\mu_1} \cdots f_{\mu_n})(\lambda) q^{\lambda}\,.$$
From  \cite{CMZ}, Proposition~6.2 we deduce 
$$c'_p(\Hmu)  = (q)_\infty c_p(\Hmu) - (q)_\infty c_p() N'_p(\Hmu)\,. $$
Since $(q)_\infty = (\sum_{\lambda \in\P} q^\lambda)^{-1}$, this is equivalent 
to the equation claimed in~\eqref{eq:cprimep}, by the definition of
$q$-brackets.
\par
Second, we use the linearity of the brackets to show
\par
\begin{flalign} \label{eq:Tpdiffex}
& \bq{ T_p\, p_{\ell_1} \cdots p_{\ell_n}}-\bq{T_p} \bq{p_{\ell_1} \cdots p_{\ell_n}} \\ 
= &  \sum_{\Gamma} \frac{1}{|\Aut(\Gamma)|}\sum_{G \in \Gamma}
\!\sum_{h\in \widetilde{\NN}^{E(G)}, \atop w\in \Z_+^{E(G)}} \!\!\Bigl( \sum_{e \in E(G)} h_e w_e^p
\Bigr) \!\!\! \prod_{e\in E(G)}  \!\!\!w_e 
q^{h_e w_e} \!\!\! \prod_{v\in V(G)}\! \ol{A}'(\bfw_v^-,\bfw_v^+,(\ell_{\# v}))\,\, \delta(v).
\nonumber
\end{flalign}
which invokes the polynomial Hurwitz numbers~$\ol{A}'(\bfw_v^-,\bfw_v^+,
(\ell_{\# v}))$. 
In fact, we can consider both sides of this equation as expressions in~$n$
arguments as in  Theorem~\ref{thm:bracketgraphsum}. Since we can write 
$p_\ell = \sum_\mu  c_{\ell,\mu} f_\mu$, multilinearity of both sides reduces
the claim to the case of arguments $f_{\mu_i}$, which is exactly the combination
of~\eqref{eq:cprimep} and the claim in Proposition~\ref{prop:svgraphsum}.
\par
By Theorem~\ref{thm:SSZ} the polynomials $\ol{A}'(\bfw_v^-,\bfw_v^+,
(\ell_{\# v}))$ are even, so if $p$ is odd the parity hypothesis
of Theorem~\ref{thm:SSVQM} is met and this theorem implies that the expression
in~\eqref{eq:Tpdiffex} is a quasimodular form. Using that
the~$f_\mu$ can be expressed as polynomials in the~$p_\ell$
by Theorem~\ref{thm:KO}, this implies the quasimodularity of $c_p'(\Hmu)$
we claimed. The weight can be determined as in the proof of 
Corollary~\ref{cor:noname}.
\end{proof}

\section{Tropical covers and quasimodularity graph by graph}
\label{sect:tropical}

The main result of \cite{BBBM} is the expression of the tropical Hurwitz 
number generating function in term of a sum over Feynman graphs.
The goal of this section is to show that their results are the special
case of simple branch points of our results, when stated in the language 
of tropical covers. More precisely, we show here that our correspondence 
theorem Proposition~\ref{prop:corrTCGraph}
has the Correspondence Theorem \cite[Theorem~2.13]{BBBM} (see also
Theorem~2.20 and Theorem~2.30 in loc.~cit.)
as immediate corollary when stated in the language of tropical covers.
In particular, for simple branching (i.e.\ ramification profile 
$\Hmu = ((2),\ldots,(2))$, our counting problem is equivalent to counting 
tropical covers.

\subsection{Tropical covers}

We recall here the definition of a tropical curve and a tropical cover, 
following~\cite{BBBM}.
\par
\begin{Defi} A {\em tropical curve} $C$ is a connected finite trivalent metric graph. An {\em elliptic tropical curve} consists of one edge forming a circle of certain length. Let $E$ be the elliptic tropical curve of length 1.  A map $\pi:C\to E$ is a {\em tropical cover} of $E$, if it is continuous, non-constant, integer affine on each edge and respects a balancing condition at every vertex of $C$. 
\end{Defi}
\par
In this setting, the {\em weight} $w_e$ of an edge $e$ for the graph  
is defined as the slope of $\pi_{|e}$, and the {\em degree} of the cover as 
\[d\= \sum_{P\in C, \, \pi(P)=p} w_{e_P}\,,\] where $p$ is a generic point. For each vertex we can group the outgoing half-edges according to the half edges of $E$ they map to.
The cover is called {\em balanced} if for each vertex the sum of the weights of the two groups agree.
\par
The appropriate way to count tropical covers is coded in the notion
of tropical Hurwitz numbers.
\par
\begin{Defi}Fix branch points $p_1, \dots p_{2g-2}$ in the tropical elliptic curve $E$. The {\em tropical Hurwitz number} is the weighted number of isomorphism classes of degree $d$ covers from a genus $g$ curve $C$, having the branch points at the $p_i$. Here, a tropical cover~$\pi$ is weighted by
the multiplicity
\[{\rm mult}(\pi)\=\frac{1}{|\Aut(\pi)|}\prod_e w_e\,.\]
\end{Defi}
\par
The {\em combinatorial type} of a tropical curve is its homeomorphism class, 
i.e.\ the underlying graph without length on the edges. These graphs
are called {\em Feynman graphs} in physics literature. They correspond
to our notion of (associated) {\em global graph}, with the vertex labeling removed.
\par
Our correspondence theorem implies the following correspondence 
result for simple ramification covers branched over $n=2g-2$ points
in term of tropical covers. The number of covers is independent, both 
for flat surface covers and for tropical covers, on the base elliptic 
curve and the branch point location. In the following corollary we 
thus fix~$E$ and the branch points on the flat side as in 
Section~\ref{sec:globalgraph}, more precisely for convenience to 
be $\ve_i = i/n$ and on the tropical side we locate the branch points 
at $p_i = i/n$. In particular the heights $h_e \in \tfrac1n \ZZ_{\geq 1}$.
\par
\begin{Cor} \label{cor:tropcor}
There is a bijective correspondence between
\begin{itemize}
\item[i)] flat surfaces $(X,\omega)$ with connected coverings $p: X \to E$ of 
degree~$d$ of the square torus~$E$,
with $\omega = p^* \omega_E$, and with simple ramification profile, weighted
by $1/|\Aut(\pi)|$, and
\item[ii)] isomorphism classes of weighted tropical covers $\pi: C \to E$
of degree~$d$, where the weight of a tropical cover corresponds to 
its multiplicity ${\rm mult}(\pi)$.
\end{itemize}
\end{Cor}
\par
\begin{proof}
This is a consequence of Proposition~\ref{prop:corrTCGraph}. Each cover corresponds to a trivalent graph with a collection of numbers $(w_e,h_e,t_e)$ corresponding to the widths, the heights and the twists of the cylinders. We define
the tropical curve~$C$ to be the global graph~$\Gamma$ with edge lengths
$\ell_e = h_e/w_e$. If~$e$ is an edge from vertex~$i$ to vertex~$j$ we
define the tropical cover on~$e$ to be the map of slope~$w_e \in \NN$ 
to the multi-segment from~$p_i$ to~$p_j$ making $\lfloor h_e \rfloor$ full 
turns. Note that this is well-defined, since for such an edge
$h_e - (j-i)/n \in \ZZ$. The tropical balancing condition is a 
restatement of $|\bfw_v^+| = |\bfw^-_v|$ for every vertex~$v$. Note that
on the one hand $d = \sum_e w_eh_e$, but on the other hand the flat picture
immediately implies $d = \sum_{e \ni \pi^{-1}(P)} w_e$, where the sum is over
all edges~$e$ such that the corresponding cylinder contains the preimage of a given
point~$P \in E$. This implies that the tropical cover we defined has
indeed degree~$d$. The map we define forgets the twist $t_e$, but
this is accounted for in the multiplicity $\textrm{mult}(\pi)$.
Note that the double Hurwitz number of a trivalent local surface with 
simple branch point is equal to one and can hence be omitted.
Besides the twist, our map has an obvious converse, associating to
a tropical cover the slopes $w_e$ and $h_e = w_e \ell_e$ and the
latter are indeed in $\tfrac1n \ZZ_{\geq 1}$ by our convention on the
location of the points~$p_i$.
\par 
Proposition~\ref{prop:corrTCGraph} is stated at the level of coverings
without unramified components. The correspondence descends under the
given hypothesis of simple branching (more generally: in case of 
only one ramified point over each branch point) to a correspondence
of connected covers by the usual inclusion-exclusion principle, 
since the obstruction of disconnected local surfaces (mentioned after
the proof of Proposition~\ref{prop:corrTCGraph}) is ruled out by this
hypothesis.
\end{proof}
\par
In this correspondence the lengths of the edges of the tropical curves 
are the reciprocal of the {\em modulus} $m_e = w_e/h_e$. This is the natural 
choice
viewing tropical curves as the dual graph of the special fiber in a 
degenerating family of smooth curves. Indeed the reciprocal moduli~$m_e^{-1}$ 
of the cylinders correspond (up to a common rescaling) to the number of
Dehn twists performed under the mondromy around the special fiber
and this in turn corresponds to a local equation $xy=t^{m_e^{-1}}$ in the
stable model of the generating fiber (see e.g.\ \cite{Mo08}, paragraph
preceding Theorem~2.4). Since 
such a singularity is resolved by a chain of $m^{-1}_e-1$ rational curves
in the semistable model with regular total space, the tropicalization
map (\cite{Viviani}) provides this edge with length~$m_e^{-1}$.

\subsection{Counting graph by graph}

For coverings with simple ramification, or equivalently for trivalent graphs, 
the quasimodularity results hold for each individual graph.
\par
\begin{Cor} \label{cor:QMindiv}
Let $\Hmu=((2),\dots, (2))$. Then for any trivalent graph~$\Gamma$ 
the contribution $N'(\Hmu,\Gamma)$ (resp $c'_p(\Hmu,\Gamma)$) of 
the graph $\Gamma$ to the total counting is a quasimodular form 
of mixed weight less or equal to ${\rm wt}(\Hmu) = |\Hmu| + \ell(\Hmu)$.
\end{Cor}
\par
These graph sums have a geometric interpretation: we count only surfaces 
with a fixed type. This result is a refinement of the quasimodularity 
results of~\cite{eo} and~\cite{CMZ} in the case of the principal strata.
\par
Note that the weight of these quasimodular form~$N'(\Hmu,\Gamma)$ is 
not necessary pure as shown by the example of $\Hmu=((2),(2),(2),(2))$ in 
Section \ref{sec:examples}. Note that in our convention, the vertices
of $\Gamma$ are labeled. Our examples show that for fixed underlying 
unlabeled trivalent graph $G$ the {\em sum over all labelings} is a 
quasimodular form of {\em pure} weight ${\rm wt}(\Hmu) = 
|\Hmu| + \ell(\Hmu)$ for $\Hmu=((2),(2))$ and $\Hmu=((2),(2),(2),(2))$. This 
purity result might hold in general.
\par
\begin{proof}
Since for simple ramification $f_2=P_2/2$ is a completed cycle,
the corollary is a straightforward consequence of Theorem~\ref{thm:QMforS} and 
Theorem~\ref{thm:SSVQM}.
\end{proof}
\par

\section{Examples}  \label{sec:examples}

The examples here have four objectives. First we show how to compare the
computations of Zorich in \cite{zoSQ} diagram by diagram  with
the counting by global graphs and local surfaces used here. Second, we
emphasize the main difficulty in the naive computation of the graph sums:
Working with non-completed cycles the graph sum is not quasimodular for
each graph separately, non even after summing over all the orientations. 
This sum belongs to the ring generated by all the ``Eisenstein series''~$G_k$, 
including odd~$k$ (compare section~\ref{sec:QMF}), and its derivatives. Only cancellations that become very
delicate as the complexity of the ramification datum grows ensure that
the total result is a quasimodular form. 
\par
Third, we illustrate the mechanism for proving quasimodularity of
the Siegel--Veech weighted counting. This is most transparent in the
case of the principal stratum in genus two,  the simple branching 
profile~$\Hmu = ((2),(2))$, where no difficulty stemming from completed 
cycles is present. 
\par
Finally, we compute the quasimodular forms individually for the trivalent
graphs corresponding to genus three covers and $\Hmu=((2),(2),(2),(2))$. 

\subsection{Branching profile $\Hmu = (3)$, the stratum $\cHH(2)$} 
\label{sec:H2stratum}

Counting geometrically as in Section~\ref{sec:globalgraph}, the global graphs
of a stratum with just one singularity have just one node, and the 
number of loops is at most two loops for genus two curves.
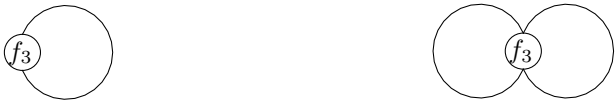
\begin{figure}[h]
\hspace{2cm}
 \begin{minipage}[t]{.40\linewidth}
\begin{tikzpicture}[line cap=round,line join=round,>=triangle 45,x=0.8cm,y=0.8cm]
\clip(-1.5,2.82) rectangle (0.84,5.22);
\draw(-1,4) circle (0.3);
\draw [shift={(-0.28,4)}] plot[domain=-2.75:2.75,variable=\t]({1*0.78*cos(\t r)+0*0.78*sin(\t r)},{0*0.78*cos(\t r)+1*0.78*sin(\t r)});
\draw (-1.4,4.34) node[anchor=north west] { $f_3$ };
\end{tikzpicture}
\end{minipage}
\hfill
\begin{minipage}[t]{.40\linewidth}
\begin{tikzpicture}[line cap=round,line join=round,>=triangle 45,x=0.8cm,y=0.8cm]
\clip(-2.98,2.8) rectangle (0.84,5.2);
\draw(-1,4) circle (0.3);
\draw [shift={(-0.28,4)}] plot[domain=-2.75:2.75,variable=\t]({1*0.78*cos(\t r)+0*0.78*sin(\t r)},{0*0.78*cos(\t r)+1*0.78*sin(\t r)});
\draw [shift={(-1.72,4)}] plot[domain=0.39:5.89,variable=\t]({1*0.78*cos(\t r)+0*0.78*sin(\t r)},{0*0.78*cos(\t r)+1*0.78*sin(\t r)});
\draw (-1.4,4.34) node[anchor=north west] {$ f_3 $};
\end{tikzpicture}
\end{minipage}
\caption{The global graphs for $\cHH(2)$: One loop $(\Gamma_1)$ or two loops ($\Gamma_2$).}
\end{figure}

\subsubsection{Computation by diagrams}

We review the computations of Zorich in~\cite{zoSQ} of torus
covers in this stratum, made with the aim of computing the Masur-Veech
volume of~$\cHH(2)$. 
We compute $N^\circ(G, \Hmu)=N'(G,\Hmu)$ for~$\Gamma_1$ and $\Gamma_2$ and all
their orientations for the ramification profile $\Hmu = (3)$ consisting 
of a three-cycle. 
\par
In terms of square-tiled surfaces in $\cHH(2)$, a first possible pattern 
is presented on the left picture of Figure~\ref{cap:graph1H2}. The picture 
on the right represents the ribbon graph made from a tubular neighborhood of the 
boundary of the horizontal cylinder. It is drawn on a torus since it can not be 
embedded in the plane. It corresponds to the only one ribbon graph with 
one vertex of valency~$6$ and~$2$ faces (boundary components).
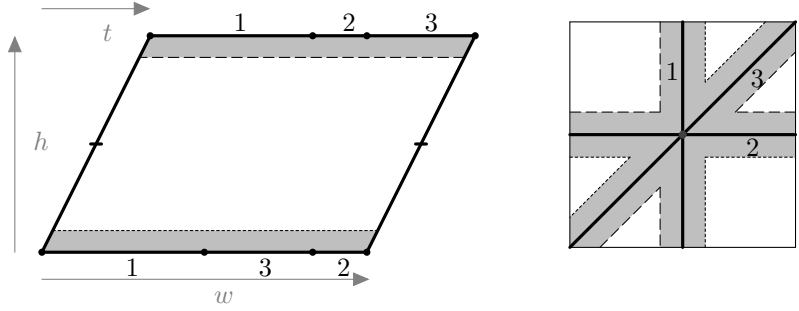
\begin{figure}[h]
\mbox{
\begin{tikzpicture}[line cap=round,line join=round,>=triangle 45,x=0.8cm,y=0.8cm, scale=0.9]
\clip(-3.61,-1.44) rectangle (5.78,5.55);
\fill[color=cqcqcq] (-1,4) -- (5,4) -- (4.8,3.6) -- (-1.2,3.6) -- cycle;
\fill[color=cqcqcq] (-3,0) -- (-2.8,0.4) -- (3.2,0.4) -- (3,0) -- cycle;
\draw [line width=1.2pt] (-3,0)-- (-1,4);
\draw [line width=1.2pt] (-1,4)-- (5,4);
\draw [line width=1.2pt] (5,4)-- (3,0);
\draw [line width=1.2pt] (3,0)-- (-3,0);
\draw [line width=1.2pt] (-2.1,2)-- (-1.9,2);
\draw [line width=1.2pt] (3.9,2)-- (4.1,2);
\draw [line width=0.4pt,dash pattern=on 4pt off 2pt] (4.8,3.6)-- (-1.2,3.6);
\draw [line width=0.4pt,dash pattern=on 1pt off 1pt] (-2.8,0.4)-- (3.2,0.4);
\draw (0.37,4.57) node[anchor=north west] {$1$};
\draw (-1.62,0.05) node[anchor=north west] {$1$};
\draw (2.38,4.57) node[anchor=north west] {$2$};
\draw (2.28,0.05) node[anchor=north west] {$2$};
\draw (3.9,4.57) node[anchor=north west] {$3$};
\draw (0.83,0.05) node[anchor=north west] {$3$};
\draw [->,color=gray] (-3,-0.5) -- (3.05,-0.5);
\draw [->,color=gray] (-3.5,0) -- (-3.5,4);
\draw [->,color=gray] (-3,4.5) -- (-1,4.5);
\draw [color=gray] (-2.05,4.38) node[anchor=north west] {$t$};
\draw [color=gray] (-3.33,2.4) node[anchor=north west] {$h$};
\draw [color=gray] (0.01,-0.56) node[anchor=north west] {$w$};
\begin{scriptsize}
\fill [color=black] (-3,0) circle (1.5pt);
\fill [color=black] (-1,4) circle (1.5pt);
\fill [color=black] (5,4) circle (1.5pt);
\fill [color=black] (3,0) circle (1.5pt);
\fill [color=black] (2,4) circle (1.5pt);
\fill [color=black] (3,4) circle (1.5pt);
\fill [color=black] (0,0) circle (1.5pt);
\fill [color=black] (2,0) circle (1.5pt);
\end{scriptsize}
\end{tikzpicture}}
\mbox{\begin{tikzpicture}[line cap=round,line join=round,>=triangle 45,x=1.0cm,y=1.0cm]
\clip(-1.58,-0.1) rectangle (2.56,4.54);
\fill[color=cqcqcq] (0.2,4) -- (0.2,2.8) -- (-1,2.8) -- (-1,2.2) -- (-0.2,2.2) -- (-1,1.4) -- (-1,1) -- (-0.6,1) -- (0.2,1.8) -- (0.2,1) -- (0.8,1) -- (0.8,2.2) -- (2,2.2) -- (2,2.8) -- (1.2,2.8) -- (2,3.6) -- (2,4) -- (1.6,4) -- (0.8,3.2) -- (0.8,4) -- cycle;
\draw  (-1,4)-- (-1,1);
\draw  (-1,1)-- (2,1);
\draw  (2,1)-- (2,4);
\draw  (2,4)-- (-1,4);
\draw [line width=1.2pt] (0.5,4)-- (0.5,1);
\draw [line width=1.2pt] (-1,2.5)-- (2,2.5);
\draw [line width=1.2pt] (-1,1)-- (2,4);
\draw [line width=0.4pt,dash pattern=on 4pt off 2pt] (0.2,4)-- (0.2,2.8);
\draw [line width=0.4pt,dash pattern=on 4pt off 2pt] (0.2,2.8)-- (-1,2.8);
\draw [line width=0.4pt,dash pattern=on 1pt off 1pt] (-1,2.2)-- (-0.2,2.2);
\draw [line width=0.4pt,dash pattern=on 1pt off 1pt] (-0.2,2.2)-- (-1,1.4);
\draw [line width=0.4pt,dash pattern=on 4pt off 2pt] (-0.6,1)-- (0.2,1.8);
\draw [line width=0.4pt,dash pattern=on 4pt off 2pt] (0.2,1.8)-- (0.2,1);
\draw [line width=0.4pt,dash pattern=on 1pt off 1pt] (0.8,1)-- (0.8,2.2);
\draw [line width=0.4pt,dash pattern=on 1pt off 1pt] (0.8,2.2)-- (2,2.2);
\draw [line width=0.4pt,dash pattern=on 4pt off 2pt] (2,2.8)-- (1.2,2.8);
\draw [line width=0.4pt,dash pattern=on 4pt off 2pt] (1.2,2.8)-- (2,3.6);
\draw [line width=0.4pt,dash pattern=on 1pt off 1pt] (1.6,4)-- (0.8,3.2);
\draw [line width=0.4pt,dash pattern=on 1pt off 1pt] (0.8,3.2)-- (0.8,4);
\draw (0.15,3.58) node[anchor=north west] {$1$};
\draw (1.28,3.48) node[anchor=north west] {$3$};
\draw (1.22,2.57) node[anchor=north west] {$2$};
\begin{scriptsize}
\fill [color=uququq] (0.5,2.5) circle (1.5pt);
\end{scriptsize}
\end{tikzpicture}}
\caption{Diagram with one cylinder} \label{cap:graph1H2}
\end{figure}
\par
As integer parameters for this square-tiled surface, we use the width of the 
cylinder~$w$, the height~$h$, the twist~$t$, and the lengths of the saddle 
connexions $\ell_1, \ell_2, \ell_3$. They are related by 
\be\label{eq:winl} w\=\ell_1+\ell_2+\ell_3, \mbox{  and  }t\in\{0,1, \dots, w-1\}.\ee
The generating function for square-tiled surfaces of this type is then 
\ba S(\mathcal C_1)&\,:=\,\sum_{h,\ell_1,\ell_2,\ell_3=1, \atop hw\geq 3}^{\infty}\sum_{t=0}^{w-1}q^{hw}\delta(w-\ell_1-\ell_2-\ell_3) \\
&\= \sum_{w,h=1 \atop hw \geq 3}^\infty w q^{hw} \left(\frac{1}{6}w^2-\frac{1}{2}w+\frac{1}{3}\right) \=\frac{1}{6}S_3-\frac{1}{2} S_2+\frac{1}{3} S_1\,,\ea
where $\frac{1}{6}w^2-\frac{1}{2}w+\frac{1}{3}$ is the number of 
solutions of~\eqref{eq:winl}, and where 
$$S_i\=\sum_{w,h=1}^\infty w^{i}q^{hw}\=G_{i+1}+\frac{B_{i+1}}{2(i+1)}\,.$$ 
From this formula we see that this generating function $N'(\Gamma_1,(3))$ 
is not a quasimodular form of weight~$\leq 6$, since the 
``Eisenstein series''~$G_3$ isn't.
\par
Figure~\ref{cap:graph2H2} represents a pattern for a square-tiled surface 
in $H(2)$ corresponding to the graph~$\Gamma_2$, i.e., with~$2$ horizontal cylinders, 
and its associated ribbon graph. 
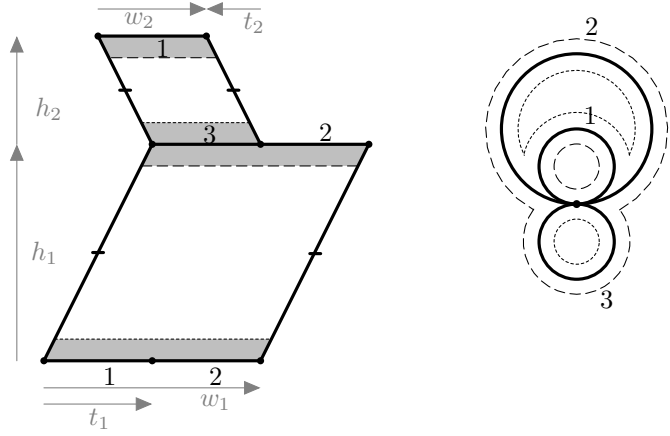
\begin{figure}[h]
\mbox{
\begin{tikzpicture}[line cap=round,line join=round,>=triangle 45,x=0.8cm,y=0.8cm, scale=0.9]
\clip(-3.8,-1.27) rectangle (3.58,6.69);
\fill[color=cqcqcq] (-1,4) -- (3,4) -- (2.8,3.6) -- (-1.2,3.6) -- cycle;
\fill[color=cqcqcq] (-3,0) -- (-2.8,0.4) -- (1.2,0.4) -- (1,0) -- cycle;
\fill[color=cqcqcq] (-1,4) -- (-1.2,4.4) -- (0.8,4.4) -- (1,4) -- cycle;
\fill[color=cqcqcq] (-2,6) -- (0,6) -- (0.2,5.6) -- (-1.8,5.6) -- cycle;
\draw [line width=1.2pt] (-3,0)-- (-1,4);
\draw [line width=1.2pt] (-1,4)-- (3,4);
\draw [line width=1.2pt] (3,4)-- (1,0);
\draw [line width=1.2pt] (1,0)-- (-3,0);
\draw [line width=1.2pt] (-2.1,2)-- (-1.9,2);
\draw [line width=1.2pt] (1.91,1.98)-- (2.11,1.98);
\draw [line width=0.4pt,dash pattern=on 4pt off 2pt] (2.8,3.6)-- (-1.2,3.6);
\draw [line width=0.4pt,dash pattern=on 1pt off 1pt] (-2.8,0.4)-- (1.2,0.4);
\draw (-1.1,6.1) node[anchor=north west] {$1$};
\draw (-2.08,0.04) node[anchor=north west] {$1$};
\draw (1.91,4.54) node[anchor=north west] {$2$};
\draw (-0.12,0.04) node[anchor=north west] {$2$};
\draw (-0.23,4.52) node[anchor=north west] {$3$};
\draw [line width=1.2pt] (-1,4)-- (-2,6);
\draw [line width=1.2pt] (-2,6)-- (0,6);
\draw [line width=1.2pt] (0,6)-- (1,4);
\draw [line width=1.2pt] (-1.6,5)-- (-1.4,5);
\draw [line width=1.2pt] (0.4,5)-- (0.6,5);
\draw [line width=0.4pt,dash pattern=on 1pt off 1pt] (-1.2,4.4)-- (0.8,4.4);
\draw [line width=0.4pt,dash pattern=on 4pt off 2pt] (0.2,5.6)-- (-1.8,5.6);
\draw [->,color=gray] (-3.5,0) -- (-3.5,4);
\draw [->,color=gray] (-3,-0.5) -- (1,-0.5);
\draw [->,color=gray] (1,6.5) -- (0,6.5);
\draw [->,color=gray] (-2,6.5) -- (0,6.5);
\draw [->,color=gray] (-3.5,4) -- (-3.5,6);
\draw [->,color=gray] (-3,-0.8) -- (-1,-0.8);
\draw [color=gray] (-1.69,6.6) node[anchor=north west] {$ w_2 $};
\draw [color=gray] (0.51,6.6) node[anchor=north west] {$ t_2 $};
\draw [color=gray] (-3.33,5.04) node[anchor=north west] {$ h_2 $};
\draw [color=gray] (-3.41,2.37) node[anchor=north west] {$ h_1 $};
\draw [color=gray] (-2.34,-0.72) node[anchor=north west] {$ t_1 $};
\draw [color=gray] (-0.29,-0.45) node[anchor=north west] {$ w_1 $};
\begin{scriptsize}
\fill [color=black] (-3,0) circle (1.5pt);
\fill [color=black] (-1,4) circle (1.5pt);
\fill [color=black] (3,4) circle (1.5pt);
\fill [color=black] (1,0) circle (1.5pt);
\fill [color=black] (-1,0) circle (1.5pt);
\fill [color=black] (0,6) circle (1.5pt);
\fill [color=black] (1,4) circle (1.5pt);
\fill [color=black] (-2,6) circle (1.5pt);
\end{scriptsize}
\end{tikzpicture}}
\mbox{
\begin{tikzpicture}[line cap=round,line join=round,>=triangle 45,x=1.0cm,y=1.0cm]
\clip(-3.11,-1) rectangle (0.75,4.57);
\draw [line width=1.2pt] (-1,3) circle (1cm);
\draw [line width=1.2pt] (-1,2.5) circle (0.5cm);
\draw [line width=1.2pt] (-1,1.5) circle (0.5cm);
\draw [line width=0.4pt,dash pattern=on 4pt off 2pt] (-1,2.5) circle (0.3cm);
\draw [line width=0.4pt,dash pattern=on 1pt off 1pt] (-1,1.5) circle (0.3cm);
\draw [shift={(-1,3)},line width=0.4pt,dash pattern=on 1pt off 1pt]  plot[domain=-0.46:3.61,variable=\t]({1*0.78*cos(\t r)+0*0.78*sin(\t r)},{0*0.78*cos(\t r)+1*0.78*sin(\t r)});
\draw [shift={(-1,2.5)},line width=0.4pt,dash pattern=on 1pt off 1pt]  plot[domain=0.21:2.93,variable=\t]({1*0.72*cos(\t r)+0*0.72*sin(\t r)},{0*0.72*cos(\t r)+1*0.72*sin(\t r)});
\draw [shift={(-1,2.99)},line width=0.4pt,dash pattern=on 4pt off 2pt]  plot[domain=-1.08:4.22,variable=\t]({1*1.21*cos(\t r)+0*1.21*sin(\t r)},{0*1.21*cos(\t r)+1*1.21*sin(\t r)});
\draw [shift={(-1,1.51)},line width=0.4pt,dash pattern=on 4pt off 2pt]  plot[domain=-3.78:0.64,variable=\t]({1*0.71*cos(\t r)+0*0.71*sin(\t r)},{0*0.71*cos(\t r)+1*0.71*sin(\t r)});
\draw (-1.01,4.6) node[anchor=north west] {2};
\draw (-0.81,0.97) node[anchor=north west] {3};
\draw (-1.03,3.41) node[anchor=north west] {1};
\begin{scriptsize}
\fill [color=black] (-1,2) circle (1.5pt);
\end{scriptsize}
\end{tikzpicture}}
\caption{Diagram for two cylinders} \label{cap:graph2H2}
\end{figure}
The integer parameters $w_1,w_2,h_1,h_2,t_1,t_2,\ell_1,\ell_2,\ell_3$ are related by 
\[w_2\=\ell_1\=\ell_3,\; w_1\=\ell_1+\ell_2,\; t_1\in\{0,\dots, w_1-1\},\; t_2\in\{0,\dots, w_2-1\}.\]
The generating function for this type of square-tiled surfaces is
\ba S(\mathcal C_2)&\,:=\,\sum_{w_1,h_1,w_2,h_2,\ell_1,\ell_2=1 \atop w.h\geq 3}^{\infty}\sum_{t_1=0}^{w_1-1}
\sum_{t_2=0}^{w_2-1}q^{h_1w_1+h_2w_2}\,
\delta(w_2-\ell_1)\,\delta(w_1-\ell_1-\ell_2) \\
&\= \sum_{w_1,w_2,h_1,h_2=1 \atop w.h \geq 3}^\infty w_1w_2 q^{h_1w_1+h_2w_2} \mathds 1_{\{w_1>w_2\}}\,.\ea
We now compute that
\[ \sum_{w_1,w_2,h_1,h_2=1}^\infty w_1w_2 q^{h_1w_1+h_2w_2} \mathds 1_{\{w_1>w_2\}} 
\= \frac{1}{2}(A-B)\,,\]
where 
\[A\=\sum_{w_1,w_2,h_1,h_2=1}^\infty w_1w_2 q^{h_1w_1+h_2w_2}=S_1^2\]
and 
\bes B \=\sum_{w_1,w_2,h_1,h_2=1}^\infty w_1w_2 q^{h_1w_1+h_2w_2}\mathds 1_{\{w_1=w_2\}}\,
\=D_qS_1-S_2
\ees
Here again, the non-quasimodularity comes from the factor $S_2=G_3$.
Summing over the two configurations we see that the term in $S_2$ cancel out, 
\[\left(\frac{1}{6}S_3-\frac{1}{2}S_2+\frac{1}{3}S_1\right)+\frac{1}{2}\left(S_1^2-DS_1+S_2\right)\=\frac{3}{2}G_2^2-\frac{1}{4}G_4+\frac{3}{8}G_2+cst\] 
resulting in a quasimodular form, as claimed.

\subsubsection{Computation by local graphs and global graphs}

In the formalism of this paper, the ribbon graph in the first case 
(i.e.\ the global graph $\Gamma_1$) corresponds to a cover of the 
cylinder~$\PP^1$, 
ramified with profile $(3,1,\dots, 1)$ at~$1$, and with~$(w)$ at~$0$ 
and~$\infty$. The corresponding triple Hurwitz number $A'(w,w,(3))$ is 
just the number of such ribbon graphs, so it is the number of 
solutions $(\ell_1, \ell_2, \ell_3)\in\N^3$ of~\eqref{eq:winl}. Consequently, 
\[S(\mathcal C_1)\=N'(\Gamma_1,(3))\=\sum_{w,h=1}^{\infty}wq^{wh}A'(w,w,(3))\,.\]
For the second graph $\Gamma_2$, the ribbon graph corresponds here to a 
cover of~$\PP^1$, ramified of profile $(3,1,\dots, 1)$ over~$1$, $(w_1,w_2)$ 
over~$0$ and~$\infty$. The number of such covers or such ribbon graphs is 
\[A'(\bfw, \bfw, (3))\=\begin{cases} 1\mbox{ if } w_1\neq w_2\\
0\mbox{ if } w_1=w_2\,, \end{cases} \quad (\bfw = (w_1, w_2))\,. \]
\par
Before passing to the computation using completed cycles, we tabulate 
the contribution of each polynomial in shifted symmetric functions appearing 
in the expression of
 \[f_3  \=\frac{1}{3} P_{3} -\frac{1}{2} (P_{1})^2
 + \frac{5}{12} P_{1}\]
to the local polynomials $A^\circ(\bfw^-, \bfw^+,(3))$ and the contribution $\bG\cdot$.
The table below shows that the contribution of each diagram individually is 
not quasimodular because the contribution of $P_1^2$ for graphs with only one 
vertex is not polynomial. It is however piecewise polynomial, showing also 
that all the hypothesis in Theorem~\ref{thm:SSZ} are needed in order to get
a globally polynomial contribution.
\renewcommand{\arraystretch}{1.5}
\[
 \begin{array}{|c|c|c|c|c|c|}
 \hline
 & A^\circ(w,w,(3)) & \bigl\langle \cdot \bigr\rangle_{q,G_1} &\multicolumn{2}{|c|}{ A^\circ(\bfw, \bfw, (3))}&\bigl\langle \cdot \bigr\rangle_{q,G_2}\\
 \hline
 &&&w_1\neq w_2& w_1=w_2 & \\
 \hline
 f_3& \frac{1}{6}w^2 - \frac{1}{2}w + \frac{1}{3}&\frac{1}{6}S_3-\frac{1}{2}S_2+\frac{1}{3}S_1 &1&0& S_1^2-DS_1+S_2\\
 \hline
 \tfrac{P_3}3 & \frac{1}{6}w^2 - \frac{1}{12} &\frac{1}{6}S_3-\frac{1}{12}S_1 & 2&2&2S_1^2\\
  \hline
 P_1^2 & w & S_2& 2&4&2(S_1^2 + DS_1-S_2)\\
  \hline
 P_1 & 1 &S_1 &0&0&0\\
  \hline
  \end{array}\]

\subsubsection{Computation using $q$-brackets of completed cycles}

Last, we compute $\sbq{f_3}$ using
\[\sbq{f_3}\=\frac{1}{3}\sbq{P_3}-\frac{1}{2}\sbq{P_1^2}+\frac{5}{12}\sbq{P_1}\,.\]
The main difference with the previous computation will be the term $\sbq{P_1^2}$ 
that we could interpret as a graph sum for a graph with two vertices.
 \renewcommand{\arraystretch}{1.5}
 \begin{center}
 \begin{tabular}{|c|c|c|}
 \hline
 & $A^\circ(\bfw^-,\bfw^+,\mu)$ & $\bG\cdot$ \\
 \hline
 \raisebox{-0.5\totalheight}{
\begin{tikzpicture}[line cap=round,line join=round,>=triangle 45,x=0.8cm,y=0.8cm]
\clip(-2.98,2.8) rectangle (0.84,5.2);
\draw(-1,4) circle (0.3);
\draw [shift={(-0.28,4)}] plot[domain=-2.75:2.75,variable=\t]({1*0.78*cos(\t r)+0*0.78*sin(\t r)},{0*0.78*cos(\t r)+1*0.78*sin(\t r)});
\draw [shift={(-1.72,4)}] plot[domain=0.39:5.89,variable=\t]({1*0.78*cos(\t r)+0*0.78*sin(\t r)},{0*0.78*cos(\t r)+1*0.78*sin(\t r)});
\draw (-1.4,4.34) node[anchor=north west] {$ P_3 $};
\end{tikzpicture}}
 & $6$ & $6S_1^2$\\
  \raisebox{-0.5\totalheight}{
\begin{tikzpicture}[line cap=round,line join=round,>=triangle 45,x=0.8cm,y=0.8cm]
\clip(-1.5,2.82) rectangle (0.84,5.22);
\draw(-1,4) circle (0.3);
\draw [shift={(-0.28,4)}] plot[domain=-2.75:2.75,variable=\t]({1*0.78*cos(\t r)+0*0.78*sin(\t r)},{0*0.78*cos(\t r)+1*0.78*sin(\t r)});
\draw (-1.4,4.34) node[anchor=north west] {$ P_3 $};
\end{tikzpicture} }
  & $\frac{1}{2}w^2 - \frac{1}{4}$ &$\frac{1}{2}S_3-\frac{1}{4}S_1$ \\
  \hline
   \raisebox{-0.5\totalheight}{
\begin{tikzpicture}[line cap=round,line join=round,>=triangle 45,x=0.8cm,y=0.8cm]
\clip(-1.96,2.32) rectangle (0.04,4.56);
\draw(-1,4) circle (0.3);
\draw (-1.35,4.3) node[anchor=north west] {$ P_1 $};
\draw(-1,2.8) circle (0.3);
\draw [shift={(-1.07,3.39)}] plot[domain=1.93:4.34,variable=\t]({1*0.43*cos(\t r)+0*0.43*sin(\t r)},{0*0.43*cos(\t r)+1*0.43*sin(\t r)});
\draw [shift={(-0.93,3.4)}] plot[domain=-1.2:1.22,variable=\t]({1*0.43*cos(\t r)+0*0.43*sin(\t r)},{0*0.43*cos(\t r)+1*0.43*sin(\t r)});
\draw (-1.35,3.1) node[anchor=north west] {$ P_1 $};
\end{tikzpicture}}

 & 1 and 1 &$S_1^2+DS_1$\\
  \hline
   \raisebox{-0.5\totalheight}{
\begin{tikzpicture}[line cap=round,line join=round,>=triangle 45,x=0.8cm,y=0.8cm]
\clip(-1.5,2.82) rectangle (0.84,5.22);
\draw(-1,4) circle (0.3);
\draw [shift={(-0.28,4)}] plot[domain=-2.75:2.75,variable=\t]({1*0.78*cos(\t r)+0*0.78*sin(\t r)},{0*0.78*cos(\t r)+1*0.78*sin(\t r)});
\draw (-1.4,4.34) node[anchor=north west] {$ P_1 $};
\end{tikzpicture}}
 & 1 &$S_1$ \\
  \hline
  \end{tabular}
  \vspace{0.5cm}
  \end{center}
In these cases, all the local contribution are globally polynomials and moreover
even functions, so the contribution of the completed cycles for each graph are 
quasimodular forms.

\subsection{Branching profile $\Hmu = ((2),(2))$, the stratum $\cHH(1,1)$}
For this stratum, as for all principal strata, all the contributions of the 
individual graphs to the counting function $N^\circ((12), (12))$ for the 
ramification profile consisting of two transpositions are quasimodular forms, 
since $f_2  \= \frac{1}{2} P_{2}$ is equal to a completed cycle.  
For $\cHH(1,1)$ there is only one possible global graph up to relabeling
the vertices.
 \renewcommand{\arraystretch}{1.5}
 \begin{center}
 \begin{tabular}{|c|c|c|}
 \hline
 & $A^\circ((w_1),(w_2,w_3),(2))$ & $\bG{P_2}$ \\
 \hline
 \raisebox{-0.5\totalheight}{
\begin{tikzpicture}[line cap=round,line join=round,>=triangle 45,x=0.9cm,y=0.8cm]
\clip(-1.96,1.38) rectangle (0.04,4.56);
\draw(-1,4) circle (0.3);
\draw (-1.3,4.3) node[anchor=north west] {$ P_2 $};
\draw(-1,2) circle (0.3);
\draw [shift={(-0.88,2.99)}] plot[domain=1.97:4.3,variable=\t]({1*0.87*cos(\t r)+0*0.87*sin(\t r)},{0*0.87*cos(\t r)+1*0.87*sin(\t r)});
\draw [shift={(-1.11,3)}] plot[domain=-1.2:1.18,variable=\t]({1*0.86*cos(\t r)+0*0.86*sin(\t r)},{0*0.86*cos(\t r)+1*0.86*sin(\t r)});
\draw (-1.3,2.3) node[anchor=north west] {$ P_2 $};
\draw (-1,3.7)-- (-1,2.3);
\end{tikzpicture}}
& 2 and 2 & $2S_{1,1}$\\
\hline
\end{tabular}
\vspace{0.5cm}
\end{center}
Among the 8 orientations of $\Gamma$ the two with all arrows ending at 
the same vertex do not contribute to the total sum~\eqref{eq:NpGamma1}.
By symmetry considerations the remaing 6 orientations have the same
contribution, and contribute with the factor $1/6$ to~\eqref{eq:NpGamma1}.
because of the automorphism group permuting the edges. Consequently, 
\bes
N^\circ((12),(12)) \=\frac{1}{6}S(\Gamma)\,, \quad \text{where}
\quad S(\Gamma) \=[\zeta_1^ 0\zeta_2^0]P^3(z_1-z_2)=[\zeta^ 0]P^3(z)
\ees
and on the other hand (expressed in the height space convention of
Lemma~\ref{lem:heightspace})
\bes
S(\Gamma) \= 6 S_{1,1} \quad \text{where}
\quad S_{1,1} \= \!\!\!\sum_{w_2,w_3,h_1=1,h_2,h_3=0}^{\infty}
\!\!\! w_2w_3(w_2+w_3)q^{w_2(h_1+h_2)+ w_3(h_1+ h_3)}
\ees
We can compute the constant coefficient using the algorithm provided
by Theorem~\ref{thm:coeff0QM} and the decomposition
\[P(z)^3\= \frac{1}{120}P^{(4)}(z) + G_2P''(z) +(12G_2^2 + 3G_4)P(z) 
-16G_2^3 + 4G_4G_2 + \frac{7}{30}G_6.\]
to obtain finally that
\bes N^\circ((12),(12))\= -\frac{8}{3} G_{2}^3 + \frac{2}{3} G_{4} G_{2} + \frac{7}{180} G_{6}.
\ees
As a cross-check, since we know that the graph sum is quasimodular, 
we can determine this quasimodular form, by computing the first terms of the
generating series 
\bes N^\circ((12), (12)) 
\= 2q^2 + 16q^3 + 60q^4 + 160q^5 + 360q^6 + 672q^7 + 1240q^8+O(q^9)
\ees
\par

\subsubsection{Siegel--Veech generating function}

Here we illustrate the method of proof of Theorem \ref{thm:SSVQM} by evaluating
the first interesting contribution $S^{SV}(\cHH(1,1))$. We use the branch point
normalization of heights $0$ and $1/2$ (to exploit the symmetry). 
The generating function is 
\bes 
S^{SV}_{1,1}\=\sum_{(h_1,h_2,h_3)\in(\N+1/2)^3 \atop (w_1,w_2,w_3)\in(\N^*)^3}\left(\frac{h_1}{w_1}+\frac{h_2}{w_2}+\frac{h_3}{w_3}\right)w_1w_2w_3\delta(w_3-w_1-w_2)q^{h.w}\,.
\ees
We will show that
\[S^{SV}_{1,1}\= -\frac{10}{3}G_2^3+\frac{5}{6}G_2G_4+\frac{7}{144}G_6\,.\]
With our convention, since the polynomial contribution of $P_2$ is 1, we get
\[c^\circ((12),(12))=\frac{1}{6}(S^{SV}_1(\Gamma,0)+S^{SV}_2(\Gamma,0)+S^{SV}_3(\Gamma,0))\]
where
\[ S^{SV}_1(\Gamma,0)\=\sum_{G\in\Gamma}\sum_{\underset{h\in\Z_{E(G)}}{w\in\N^*}}h_1w_2w_3\delta\left(\sum_{i\in e^+(v)}w_i-\sum_{i\in e^-(v)}w_i\right)q^{h\cdot w}.\]
One can check directly on this example that 
\[S^{SV}_1(\Gamma,0)\=[\zeta^0]L(z)P^2(z)=S^{SV}_2(\Gamma,0)=S^{SV}_3(\Gamma,0)\]
(the integration with respect to the second variable $z_2$ is not necessary),
 and that \[S^{SV}_{1,1}=\frac{1}{2}S^{SV}_1(\Gamma,0)=c^\circ((12),(12))\,.\] 
\par
The proof of Theorem \ref{thm:SSVQM} also provides
an algorithm for computation. The decomposition of~$L$ in the standard
generators has been given in~\eqref{eq:LZ2}. Using 
\[P^2(z) \=\frac{1}{6}P''(z) + 4G_2P(z) -4G_2^2 +\frac{ 5}{3}G_4\]
and the decomposition of $P^3(z)$ given previously, we obtain
\bes 
[\zeta^0]Z^2(z)P^2(z) \= \frac{16}{3}G_2^3 - \frac{2}{3}G_2^2 - \frac{8}{3}G_4G_2 + \frac{5}{18}G_4 + \frac{7}{180}G_6
\ees
and 
\[[\zeta^0]\left(\frac{1}{2}P(z)-G_2+\frac{1}{12}\right)P^2(z)=-4G_2^3 - \frac{1}{3}G_2^2 + \frac{1}{3}G_4G_2 + \frac{5}{36}G_4 + \frac{7}{60}G_6\,.\] 
Taken together, we conclude that
\begin{align*}
c^\circ((12),(12)) \=\frac{1}{2}[\zeta^0]LP^2\=-\frac{10}{3}G_2^3+\frac{5}{6}G_2G_4+\frac{7}{144}G_6 \=\frac{5}{4}S(\cHH(1,1))\,.
\end{align*}
The Siegel--Veech constant for the stratum $\cHH(1,1)$ is $5/4$ and this
proportionality of generating series is expected by the ``non-varying property'' 
of $\cHH(1,1)$, see the discussion in \cite[Section 17]{CMZ}.
\par
\subsection{Branching profile $\Hmu = ((2),(2),(2),(2))$, the stratum 
$\cHH(1^4)$}

We end this range of examples with the stratum $\cHH(1^4)$, first to show 
the effect of the labeling of the zeros and to compute the quasimodular
forms for individual graphs.
%
%
%
There are only two trivalent connected (multi)graphs with 4 vertices, 
depicted in Figure \ref{fig:H1111global}. The graphs $\Gamma$ are 
obtained by labeling the vertices of these graphs.
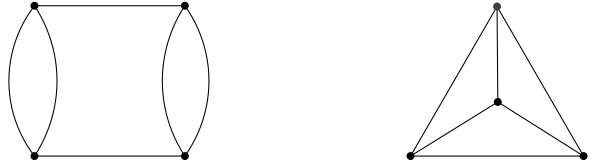
\begin{figure}[h]
\begin{tikzpicture}[line cap=round,line join=round,>=triangle 45,x=1.0cm,y=1.0cm]
\clip(-4.56,0.18) rectangle (5.38,3.8);
\draw (-3,3)-- (-1,3);
\draw (-3,1)-- (-1,1);
\draw (2,1)-- (4.3,1);
\draw (4.3,1)-- (3.15,2.99);
\draw (3.15,2.99)-- (2,1);
\draw [shift={(-1.7,2)}] plot[domain=2.49:3.8,variable=\t]({1*1.64*cos(\t r)+0*1.64*sin(\t r)},{0*1.64*cos(\t r)+1*1.64*sin(\t r)});
\draw [shift={(-4.52,2)}] plot[domain=-0.58:0.58,variable=\t]({1*1.82*cos(\t r)+0*1.82*sin(\t r)},{0*1.82*cos(\t r)+1*1.82*sin(\t r)});
\draw [shift={(0.52,2)}] plot[domain=2.56:3.72,variable=\t]({1*1.82*cos(\t r)+0*1.82*sin(\t r)},{0*1.82*cos(\t r)+1*1.82*sin(\t r)});
\draw [shift={(-2.4,2)}] plot[domain=-0.62:0.62,variable=\t]({1*1.72*cos(\t r)+0*1.72*sin(\t r)},{0*1.72*cos(\t r)+1*1.72*sin(\t r)});
\draw (3.15,2.99)-- (3.16,1.72);
\draw (2,1)-- (3.16,1.72);
\draw (3.16,1.72)-- (4.3,1);
\begin{scriptsize}
\fill [color=black] (-3,3) circle (1.5pt);
\fill [color=black] (-1,3) circle (1.5pt);
\fill [color=black] (-3,1) circle (1.5pt);
\fill [color=black] (-1,1) circle (1.5pt);
\fill [color=black] (2,1) circle (1.5pt);
\fill [color=black] (4.3,1) circle (1.5pt);
\fill [color=uququq] (3.15,2.99) circle (1.5pt);
\fill [color=black] (3.16,1.72) circle (1.5pt);
\end{scriptsize}
\end{tikzpicture}
\caption{Global graphs for $\cHH(1^4)$: Type 1 (left) and type 2 (right)}
\label{fig:H1111global}
\end{figure}

For the first graph, we can use the horizontal and vertical flip 
and assume that the bottom left vertex is labeled by one. Among the~$6$
ways to label the remaining vertices, our normalization of the integration
paths (or equivalently height spaces, see \ref{sec:intEF}), results in
two essentially different quasimodular forms. 
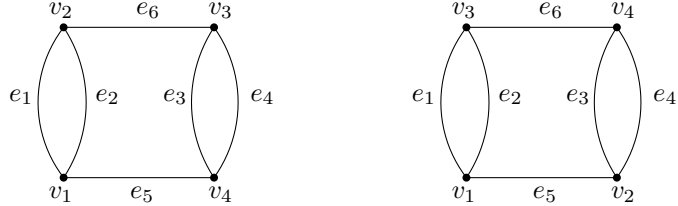
\begin{figure}[h]
\hspace{1cm}
 \begin{minipage}[t]{.40\linewidth}
\begin{tikzpicture}[line cap=round,line join=round,>=triangle 45,x=1.0cm,y=1.0cm]
\clip(-4.56,0.18) rectangle (0.48,3.8);
\draw (-3,3)-- (-1,3);
\draw (-3,1)-- (-1,1);
\draw (2,1)-- (4.3,1);
\draw (4.3,1)-- (3.15,2.99);
\draw (3.15,2.99)-- (2,1);
\draw [shift={(-1.7,2)}] plot[domain=2.49:3.8,variable=\t]({1*1.64*cos(\t r)+0*1.64*sin(\t r)},{0*1.64*cos(\t r)+1*1.64*sin(\t r)});
\draw [shift={(-4.52,2)}] plot[domain=-0.58:0.58,variable=\t]({1*1.82*cos(\t r)+0*1.82*sin(\t r)},{0*1.82*cos(\t r)+1*1.82*sin(\t r)});
\draw [shift={(0.52,2)}] plot[domain=2.56:3.72,variable=\t]({1*1.82*cos(\t r)+0*1.82*sin(\t r)},{0*1.82*cos(\t r)+1*1.82*sin(\t r)});
\draw [shift={(-2.4,2)}] plot[domain=-0.62:0.62,variable=\t]({1*1.72*cos(\t r)+0*1.72*sin(\t r)},{0*1.72*cos(\t r)+1*1.72*sin(\t r)});
\draw (3.15,2.99)-- (3.16,1.72);
\draw (2,1)-- (3.16,1.72);
\draw (3.16,1.72)-- (4.3,1);
\draw (-3.32,1) node[anchor=north west] {$v_1$};
\draw (-3.32,3.45) node[anchor=north west] {$ v_2 $};
\draw (-1.2,3.45) node[anchor=north west] {$ v_3 $};
\draw (-1.2,1.0) node[anchor=north west] {$ v_4 $};
\draw (-3.85,2.3) node[anchor=north west] {$ e_1 $};
\draw (-2.7,2.3) node[anchor=north west] {$ e_2 $};
\draw (-1.8,2.3) node[anchor=north west] {$ e_3 $};
\draw (-0.64,2.3) node[anchor=north west] {$ e_4 $};
\draw (-2.24,1) node[anchor=north west] {$ e_5 $};
\draw (-2.16,3.45) node[anchor=north west] {$ e_6 $};
\begin{scriptsize}
\fill [color=black] (-3,3) circle (1.5pt);
\fill [color=black] (-1,3) circle (1.5pt);
\fill [color=black] (-3,1) circle (1.5pt);
\fill [color=black] (-1,1) circle (1.5pt);
\fill [color=black] (2,1) circle (1.5pt);
\fill [color=black] (4.3,1) circle (1.5pt);
\fill [color=uququq] (3.15,2.99) circle (1.5pt);
\fill [color=black] (3.16,1.72) circle (1.5pt);
\end{scriptsize}
\end{tikzpicture}
\end{minipage}
\hfill
 \begin{minipage}[t]{.40\linewidth}
 \begin{tikzpicture}[line cap=round,line join=round,>=triangle 45,x=1.0cm,y=1.0cm]
\clip(-4.56,0.18) rectangle (0.48,3.8);
\draw (-3,3)-- (-1,3);
\draw (-3,1)-- (-1,1);
\draw (2,1)-- (4.3,1);
\draw (4.3,1)-- (3.15,2.99);
\draw (3.15,2.99)-- (2,1);
\draw [shift={(-1.7,2)}] plot[domain=2.49:3.8,variable=\t]({1*1.64*cos(\t r)+0*1.64*sin(\t r)},{0*1.64*cos(\t r)+1*1.64*sin(\t r)});
\draw [shift={(-4.52,2)}] plot[domain=-0.58:0.58,variable=\t]({1*1.82*cos(\t r)+0*1.82*sin(\t r)},{0*1.82*cos(\t r)+1*1.82*sin(\t r)});
\draw [shift={(0.52,2)}] plot[domain=2.56:3.72,variable=\t]({1*1.82*cos(\t r)+0*1.82*sin(\t r)},{0*1.82*cos(\t r)+1*1.82*sin(\t r)});
\draw [shift={(-2.4,2)}] plot[domain=-0.62:0.62,variable=\t]({1*1.72*cos(\t r)+0*1.72*sin(\t r)},{0*1.72*cos(\t r)+1*1.72*sin(\t r)});
\draw (3.15,2.99)-- (3.16,1.72);
\draw (2,1)-- (3.16,1.72);
\draw (3.16,1.72)-- (4.3,1);
\draw (-3.32,1) node[anchor=north west] {$v_1$};
\draw (-3.32,3.45) node[anchor=north west] {$ v_3 $};
\draw (-1.2,3.45) node[anchor=north west] {$ v_4 $};
\draw (-1.2,1.0) node[anchor=north west] {$ v_2 $};
\draw (-3.85,2.3) node[anchor=north west] {$ e_1 $};
\draw (-2.7,2.3) node[anchor=north west] {$ e_2 $};
\draw (-1.8,2.3) node[anchor=north west] {$ e_3 $};
\draw (-0.64,2.3) node[anchor=north west] {$ e_4 $};
\draw (-2.24,1) node[anchor=north west] {$ e_5 $};
\draw (-2.16,3.45) node[anchor=north west] {$ e_6 $};
\begin{scriptsize}
\fill [color=black] (-3,3) circle (1.5pt);
\fill [color=black] (-1,3) circle (1.5pt);
\fill [color=black] (-3,1) circle (1.5pt);
\fill [color=black] (-1,1) circle (1.5pt);
\fill [color=black] (2,1) circle (1.5pt);
\fill [color=black] (4.3,1) circle (1.5pt);
\fill [color=uququq] (3.15,2.99) circle (1.5pt);
\fill [color=black] (3.16,1.72) circle (1.5pt);
\end{scriptsize}
\end{tikzpicture}
 \end{minipage}
 \hspace{1cm}
\caption{Labeled graph of type 1 for $\cHH(1^4)$} \label{cap:numbering}
\end{figure}

The numbering on the left gives the quasimodular form 
\begin{flalign}
A &\= [\zeta^0]P(z_1-z_2)^2P(z_1-z_4)P(z_2-z_3)P(z_3-z_4)^2 \label{eq:Hdomain1}
\\
&=\, 4q^2 + 224q^3 + 3088q^4 + 21888q^5 + 105136q^6 + 388288q^7 + 1197280q^8
+O(q^9) \nonumber \\
&= \left( -256G_2^6 + \frac{640}{3}G_4G_2^4 +  \frac{112}{9}G_6G_2^3 
-\frac{400}{9}G_4^2 G_2^2   - \frac{140}{9}G_6G_4G_2 + \frac{2000}{81}G_4^3 
\right.  
 \nonumber\\
&\phantom{\=}  \+ \left. \frac{49}{108}G_6^2 \right) \+\ \left(-\frac{256}{3}G_4G_2^3-\frac{16}{5}G_6G_2^2+\frac{320}{21}G_4^2G_2
+ \frac{28}{9}G_6G_4\right)\,, \nonumber
\end{flalign}
that we computed in Section~\ref{sec:Ex1111Pcomp} and there is a second
labeling that results in the same quasimodular form, in fact 
\[ A=[\zeta^0]P(z_1-z_4)^2P(z_1-z_2)P(z_2-z_3)^2P(z_3-z_4)\,.\]
The numbering on the
right of Figure~\ref{cap:numbering} and three other numberings produce
the quasimodular form
\begin{flalign}
B &\=[\zeta^0]P(z_1-z_3)^2P(z_1-z_2)P(z_2-z_4)^2P(z_3-z_4) 
\label{eq:Hdomain2} \\
&\= 40q^4 + 448q^5 + 2848q^6 + 11776q^7 + 41744q^8+O(q^9)  \nonumber\\
&\ = \left( -256G_2^6 + \frac{640}{3}G_4G_2^4 + \frac{112}{9}G_6G_2^3 -\frac{400}{9}G_4^2G_2^2   - \frac{140}{9}G_6G_4G_2 + \frac{2000}{81}G_4^3 \right. \nonumber\\
& \phantom{\=}  + \left. \frac{49}{108}G_6^2 \right) \+ \left(\frac{128}{3}G_4G_2^3+ \frac{8}{5}G_6G_2^2-\frac{160}{21}G_4^2G_2- \frac{14}{9}G_6G_4\right)\,.  \nonumber
\end{flalign}
This difference illustrates a wall-crossing phenomenon: fixing the heights of the zeros (the domain of integration) and changing the labeling of the zeros is the same as fixing the labeling of the zeros and changing the integration domain. Equations \eqref{eq:Hdomain1} and \eqref{eq:Hdomain2} are the contour integral of the same function on two different domains.
Note also that these contributions are of weight 12 and 10, and their weight 12 part coincide, whereas their weight 10 part coincide up to a factor $-2$.
\par
The total contribution of this graph is then 
\bas
2A+4B& =6\,\Bigl(-256G_2^6 + \frac{640}{3}G_4G_2^4 +  \frac{112}{9}G_6G_2^3 -\frac{400}{9}G_4^2 G_2^2   - \frac{140}{9}G_6G_4G_2 \\
&\phantom{\=\,\,} + \frac{2000}{81}G_4^3  + \frac{49}{108}G_6^2 \Bigr)\,.
\eas
The second graph on Figure \ref{fig:H1111global} is totally symmetric: 
there is only one way to label the vertices. Its contribution is
\begin{align*}
C&=[\zeta^0]P(z_1-z_2)P(z_1-z_3)P(z_1-z_4)P(z_2-z_3)P(z_2-z_4)P(z_3-z_4)
\\&=-384G_2^6 + 480G_4G_2^4 - 200G_4^2G_2^2 + \frac{250}{9}G_4^3
\end{align*}
The total connected generating function for the graph of type 1 is then
\bas
&\phantom{=} \,\,\frac{1}{4}(2A+4B)+C \\
&=-768G_2^6 + 800G_4G_2^4 + \frac{56}{3}G_6G_2^3 - \frac{800}{3}G_4^2G_2^2 - \frac{70}{3}G_6G_4G_2 + \frac{1750}{27}G_4^3 + \frac{49}{72}G_6^2
\\&=   2q^2 + 160q^3 + 2448q^4 + 18304q^5 + 90552q^6 + 341568q^7 + 1068928q^8+O(q^9)
\eas
The factor $4$ is due to the automorphism group of each labeled graph 
of the type~$1$.

\bibliography{my}
\end{document}